\documentclass[preprint,11pt,3p]{elsarticle}

\usepackage{color}
\usepackage[colorlinks=true]{hyperref}
\usepackage[title]{appendix}
\usepackage{amsmath,amssymb,amsthm,amscd}
\usepackage{array}
\usepackage{booktabs}
\numberwithin{equation}{section}

\newtheorem{theorem}{Theorem}[section]

\newtheorem{proposition}{Proposition}[section]
\newtheorem{lemma}{Lemma}[section]
\newtheorem{definition}{Definition}[section]
\newtheorem{corollary}{Corollary}[section]
\newtheorem{remark}{Remark}[section]
\newtheorem{example}{Example}[section]

\biboptions{sort&compress}

\usepackage{times}
\begin{document}
\begin{frontmatter}

\title{Asymptotic behaviour of Dirichlet eigenvalues for homogeneous H\"{o}rmander operators and algebraic geometry approach}

\author[label1]{Hua Chen\corref{cor1}}
\ead{chenhua@whu.edu.cn}
\author[label2]{Hong-Ge Chen}
\ead{hongge\_chen@whu.edu.cn}
\author[label1]{Jin-Ning Li}
\ead{lijinning@whu.edu.cn}

\address[label1]{School of Mathematics and Statistics, Wuhan University, Wuhan 430072, China}
\address[label2]{Wuhan Institute of Physics and Mathematics, Innovation Academy for Precision Measurement Science and Technology, Chinese Academy of Sciences, Wuhan 430071, China}

\cortext[cor1]{corresponding author}

\begin{abstract}
We study the Dirichlet eigenvalue problem of homogeneous H\"{o}rmander operators $\triangle_{X}=\sum_{j=1}^{m}X_{j}^{2}$ on a bounded open domain containing the origin, where $X_1, X_2, \ldots, X_m$ are linearly independent smooth vector fields in $\mathbb{R}^n$ satisfying H\"{o}rmander's condition and a suitable homogeneity property with respect to a family of non-isotropic dilations. Suppose that $\Omega$ is an open bounded domain in $\mathbb{R}^n$ containing the origin. We use the Dirichlet form to study heat semigroups and subelliptic heat kernels. Then, by utilizing  subelliptic heat kernel estimates, the resolution of singularities in algebraic geometry, and employing some refined analysis involving convex geometry, we establish the explicit asymptotic behavior $\lambda_k \approx k^{\frac{2}{Q_0}}(\ln k)^{-\frac{2d_0}{Q_0}}$ as $k \to +\infty$, where $\lambda_k$ denotes the $k$-th Dirichlet eigenvalue of $\triangle_{X}$ on $\Omega$, $Q_0$ is a positive rational number, and $d_0$ is a non-negative integer. Furthermore, we provide optimal bounds of index $Q_0$, which depend on the homogeneous dimension associated with the vector fields $X_1, X_2, \ldots, X_m$.
\end{abstract}

\end{frontmatter}

\section{Introduction and main results}
\label{Section1}
For $n\geq 2$, let $X=(X_{1},X_{2},\ldots,X_{m})$ be the real smooth vector fields defined on $\mathbb{R}^n$ that satisfy the following assumptions:
\begin{enumerate}
  \item [(H.1)]

   There exists a family of (non-isotropic) dilations $\{\delta_{t}\}_{t>0}$ of the form
  \[ \delta_{t}:\mathbb{R}^n\to \mathbb{R}^n,\qquad \delta_{t}(x)=(t^{\alpha_{1}}x_{1},t^{\alpha_{2}}x_{2},\ldots,t^{\alpha_{n}}x_{n}), \]
 where $1=\alpha_{1}\leq \alpha_{2}\leq\cdots\leq \alpha_{n}$ are positive integers, such that $X_{1},X_{2},\ldots,X_{m}$ are $\delta_{t}$-homogeneous of degree $1$. That is, for all $t>0$, $f\in  C^{\infty}(\mathbb{R}^n)$, and $j = 1, \ldots, m$,
 \[ X_{j}(f\circ \delta_{t})=t(X_{j}f)\circ \delta_{t}. \]

   \item [(H.2)]
     \label{H-2}
    The vector fields $X_{1},X_{2},\ldots,X_{m}$ are linearly independent in $\mathcal{X}(\mathbb{R}^n)$ as linear differential operators, and satisfy  H\"{o}rmander's condition at $0\in \mathbb{R}^{n}$, i.e.,
       \[ \dim\{Y(0)|~Y\in \text{Lie}(X)\}=n, \]
 where $\text{Lie}(X)$ is the smallest Lie subalgebra in $\mathcal{X}(\mathbb{R}^n)$ containing $X=(X_{1},X_{2},\ldots,X_{m})$. Here, $\mathcal{X}(\mathbb{R}^n)$ denotes the set of all smooth vector fields in $\mathbb{R}^{n}$, which is also a vector space over $\mathbb{R}$ equipped with the natural operations.

\end{enumerate}

We denote the  $\delta_{t}$-homogeneous dimension of $(\mathbb{R}^n,\delta_{t})$ by
\begin{equation}\label{1-1}
  Q:=\sum_{j=1}^{n}\alpha_{j}.
\end{equation}

Then, we define the formally self-adjoint operator  $\triangle_{X}$  generated by vector fields $X_{1},X_{2},\ldots,X_{m}$ as follows:
\[ \triangle_{X}:=-\sum_{i=1}^{m}X_{i}^{*}X_{i},\]
where $X_{i}^{*}=-X_{i}-\text{div}X_{i}$ denotes the formal adjoint of $X_{i}$, and $\text{div}X_{i}$ is the
divergence of $X_{i}$. From assumptions (H.1) and (H.2), we can deduce that the Lie algebra $\text{\rm Lie}(X)$ is nilpotent of step $\alpha_{n}$, and the vector fields $X=(X_{1},X_{2},\ldots,X_{m})$ satisfy H\"{o}rmander's condition in $\mathbb{R}^n$. Specifically,  there exists a smallest positive integer $r$ such that $X_{1},X_{2},\ldots,X_{m}$ together with their commutators of length at most $r$ span the tangent space $T_{x}(\mathbb{R}^n)$ at each point $x\in\mathbb{R}^n$ (see Proposition \ref{prop2-5} below). We therefore  refer to the real smooth vector fields that satisfy assumptions (H.1) and (H.2) as \emph{homogeneous H\"{o}rmander vector fields}, and $r$ is called the H\"{o}rmander index of $X$.  It is worth noting that
 $r=\alpha_{n}$ under assumptions (H.1) and (H.2).

Moreover, by assumption (H.1) and Proposition \ref{prop2-2} below, we have $X_{i}^{*}=-X_{i}$ for $i=1,\ldots,m$. Thus, $\triangle_{X}$ has the sum of square form
\begin{equation}\label{1-2}
  \triangle_{X}=\sum_{i=1}^{m}X_{i}^{2}.
\end{equation}
The operator $\triangle_{X}$ in \eqref{1-2} under assumptions (H.1) and (H.2) is called the \emph{homogeneous H\"{o}rmander operator}.

The class of homogeneous H\"{o}rmander operators is extensive and encompasses significant degenerate operators that have been widely studied in the literature. Examples of such operators include sub-Laplacians on Carnot groups, Grushin operators, and Martinet operators. In recent years, the analysis of homogeneous H\"{o}rmander operators has garnered substantial interest.
Utilizing Folland's global lifting method \cite{Folland1977}, Biagi-Bonfiglioli \cite{Biagi2017} investigated the existence of  global fundamental solutions. Subsequently, Biagi-Bonfiglioli-Bramanti \cite{Biagi-Bonfiglioli-Bramanti2019} established global estimates for the fundamental solution. Concurrently, the existence and Gaussian bounds of global heat kernels for homogeneous H\"{o}rmander operators have been examined by Biagi-Bonfiglioli \cite{Biagi2019}, as well as Biagi-Bramanti \cite{Biagi2020}.
For further results on degenerate elliptic equations related to homogeneous H\"{o}rmander operators, one can refer to \cite{Battaglia2018, BiagiS, Biagi2020-jde, Biagi2021-jmaa, Biagi2021-nodea}.

As it is well-known,  the sub-Riemannian geometry provides the natural geometric framework for subelliptic PDEs.
Under H\"{o}rmander's condition, the vector fields $X=(X_{1},X_{2},\ldots,X_{m})$ induce a canonical sub-Riemannian structure $(D,g)$, such that $\mathbb{R}^n$ endowed with $(D,g)$ forms a sub-Riemannian manifold $(\mathbb{R}^n,D,g)$. Here, $D$ is the distribution consisting of linear subspaces $D_x \subset T_x(\mathbb{R}^n)$, which smoothly depend on $x \in \mathbb{R}^n$, with $D_{x}={\rm span}\{X_{1}(x),X_{2}(x),\ldots,X_{m}(x)\}$. The sub-Riemannian metric is denoted by $g$. When $D$ has a constant rank $m$ on $\mathbb{R}^n$ with $m\leq n$ (i.e., the dimension $\dim D_{x}=m\leq n$ for every $x\in \mathbb{R}^n$), $D$ is the subbundle of $T\mathbb{R}^{n}$. In this case, the vector fields $X=(X_{1},X_{2},\ldots,X_{m})$ are orthonormal concerning the sub-Riemannian metric $g$, and the H\"{o}rmander operator $\triangle_{X}$  coincides with the sub-Laplacian on sub-Riemannian manifold $(\mathbb{R}^n,D,g)$, as shown in \cite{Gordina2016}.
 However, the rank of $D$ may vary in general, and the above setting includes almost sub-Riemannian structures. Further details  on sub-Riemannian geometry can be found in \cite{Montgomery2002,Jean2014,Bellaiche1996,Andrei2019,Gromov1996,Rifford2014,Stefani2014}.

For each $x\in \mathbb{R}^n$ and $1\leq j\leq \alpha_{n}$,  we denote by $D_{x}^{j}$ the subspace of $T_{x}(\mathbb{R}^n)$ spanned by all commutators of $X_{1},X_{2},\ldots,X_{m}$ with length at most $j$ . Setting $\nu_{j}(x)=\dim D_{x}^{j}$ for $1\leq j\leq \alpha_{n}$ with $\nu_{0}(x):=0$, we define

\begin{equation}\label{1-3}
\nu(x):=\sum_{j=1}^{\alpha_{n}}j(\nu_{j}(x)-\nu_{j-1}(x))
\end{equation}
as the pointwise homogeneous dimension at $x$ (see \cite{Morbidelli2000}). A point $x\in \mathbb{R}^n$ is regular if, for every $1\leq j\leq \alpha_{n}$, the dimension $\nu_{j}(y)$ is a constant as $y$ varies in an open neighbourhood of $x$. Otherwise, $x$ is said to be singular. Furthermore, a set $S\subset \mathbb{R}^n$ is equiregular if every point of $S$ is regular, while
a set $S\subset \mathbb{R}^n$ is non-equiregular if it contains singular points. The equiregular assumption is also known as the M\'etivier's condition in PDEs (see \cite{Metivier1976}). For an equiregular connected set $S$, the pointwise homogeneous dimension $\nu(x)$ is a constant $\nu$, which coincides with the Hausdorff dimension of $S$ related to the vector fields $X$. This constant $\nu$ is also referred to as the M\'etivier's index. If the set $S\subset \mathbb{R}^n$ is non-equiregular, we can introduce the so-called generalized M\'{e}tivier's index  by
\begin{equation}\label{1-4}
  \tilde{\nu}_{S}:=\max_{x\in \overline{S}}\nu(x).
\end{equation}
The generalized M\'{e}tivier's index  is
also known as the non-isotropic dimension (see \cite{Yung2015,chen-chen2019,chen-chen2020}), which plays a significant role in the geometry and functional settings associated
with vector fields $X$. Note that $n+\alpha_{n}-1\leq \tilde{\nu}_{S}< n\alpha_{n}$ for $\alpha_{n}>1$, and  $\tilde{\nu}_{S}=\nu$ if the closure of $S$ is equiregular and connected.

In this paper, we study the Dirichlet eigenvalue problem associated with the homogeneous H\"{o}rmander operator $\triangle_X$, given by:
    \begin{equation}\label{1-5}
  \left\{
         \begin{array}{ll}
           -\triangle_{X}u=\lambda u, & \hbox{on $\Omega$;} \\[3mm]
          u\in H_{X,0}^{1}(\Omega), &
         \end{array}
       \right.
\end{equation}
where $\Omega\subset \mathbb{R}^{n}$ is a bounded open domain containing the origin and $H_{X,0}^{1}(\Omega)$ is the weighted Sobolev space associated with vector fields $X$ (see Section \ref{Section2} for details). Following H\"{o}rmander's condition, $-\triangle_{X}$ can be extend to
 a positive self-adjoint operator that exhibits discrete Dirichlet eigenvalues arranged as $0<\lambda_1\leq\lambda_2\leq\cdots\leq\lambda_{k-1}\leq\lambda_k\leq\cdots$,
and $\lambda_{k}\to +\infty $ as $k\to +\infty$.

In the case where $X=(\partial_{x_{1}},\ldots,\partial_{x_{n}})$, $\triangle_{X}$ simplifies to the standard Laplacian  $\triangle$. The classical eigenvalue problems associated with the Laplacian
have been studied extensively since Weyl \cite{Weyl1912} established the celebrated asymptotic formula:
\[ \lambda_{k}\sim 4\pi\left(\frac{|\Omega|}{\Gamma\left(\frac{n}{2}+1\right)}\right)^{-\frac{2}{n}}\cdot k^{\frac{2}{n}}\qquad\mbox{as}~~~k \to+\infty,  \]
where $|\Omega|$ is the $n$-dimensional Lebesgue measure of $\Omega$. We list some literature in this area, including \cite{chenqingmin1, chenqingmin2, Kroger1994, Yau1983, Polya1961, Weyl1912,Weyl1950, Seeley1978, Ivrii1980, Melrose, Kac1966, Hermi2008, cheng2009, Cheng1981}, along with the references cited therein.

The investigation of eigenvalue problems concerning the H\"{o}rmander operators can be traced back to M\'{e}tivier \cite{Metivier1976} in 1976. By imposing the
 equiregular assumption on $\overline{\Omega}$, M\'{e}tivier established the following asymptotic formula:
\begin{equation}\label{1-6}
  \lambda_{k}\sim \left(\int_{\Omega}\gamma(x)dx\right)^{-\frac{2}{\nu}}\cdot k^{\frac{2}{\nu}}\qquad\mbox{as}~~ k\to+\infty,
\end{equation}
where $\gamma(x)$ is a positive continuous function on $\Omega$ and $\nu$ is the  M\'{e}tivier index of $\Omega$. However, the asymptotic behaviour of eigenvalues has not been fully understood when the equiregular assumption is dropped.
To the best of our knowledge, only a few asymptotic results have been obtained for the general non-equiregular case. In 1981, Fefferman-Phong
\cite{Fefferman1981} proved that, for the closed eigenvalue problem of H\"{o}rmander operator on compact (closed) smooth manifold $M$,
\begin{equation}\label{1-7}
	c_1\int_{M}\frac{d\mu}{\mu(B_{d_{X}}(x,\lambda^{-\frac{1}{2}}))}\leq N(\lambda)\leq
	c_2\int_{M}\frac{d\mu}{\mu(B_{d_{X}}(x,\lambda^{-\frac{1}{2}}))}	
\end{equation}
holds for sufficiently large $\lambda$, where $N(\lambda):=\#\{k|\lambda_{k}\leq \lambda\}$ is the spectral counting function, $c_2>c_1>0$ are constants depending on $X$ and $M$, $\mu$ is the smooth measure on $M$, and
$B_{d_{X}}(x,r)$ is the subunit ball (defined in Section \ref{Section2} below). Nevertheless, without additional assumptions, the abstract integral in \eqref{1-7} cannot be explicitly calculated for general  H\"{o}rmander operators. Furthermore, the applicability of \eqref{1-7}  to the Dirichlet eigenvalues of H\"{o}rmander operators is a longstanding open problem, which will be addressed by Theorem \ref{thm1} in homogeneous case.

In a recent study \cite{chen-chen2020}, the first two current authors investigated the Dirichlet eigenvalue problem for H\"{o}rmander operators in the non-equiregular case and provided an explicit asymptotic formula under certain weak conditions. Specifically, consider an open bounded domain $\Omega\subset \mathbb{R}^n$ with smooth boundary that is non-characteristic for vector fields $X$. Let $H:=\{x\in \Omega|\nu(x)=\tilde{\nu}\}$ be the level set at which the pointwise homogeneous dimension attains its maximum value.  It was proven in
 \cite{chen-chen2020} that if $H$ possesses a positive measure, then
\begin{equation}\label{1-8} \lambda_{k}\sim\left(\frac{\Gamma\left(\frac{\tilde{\nu}}{2}+1\right)}{\int_{H}\gamma_{0}(x)dx}\right)^{\frac{2}{\tilde{\nu}}}\cdot
k^{\frac{2}{\tilde{\nu}}}\qquad\mbox{ as }~~k\to +\infty,
\end{equation}
where $\tilde{\nu}:=\max_{x\in \overline{\Omega}}\nu(x)$ is the non-isotropic dimension (generalized M\'{e}tivier's index) of $\Omega$  depending on vector fields $X$, and $\gamma_{0}$ is a positive measurable function on $\Omega$. We mention that \eqref{1-8} generalizes M\'{e}tivier's asymptotic formula \eqref{1-6}. In fact, the equiregular assumption on $\overline{\Omega}$  implies that $\tilde{\nu}=\nu$ and $H=\Omega$. Therefore, the condition $|H|=|\Omega|>0$ is certainly satisfied. Furthermore, in the case of $|H|=0$, \cite{chen-chen2020} derived the following result:
\begin{equation}\label{1-9}
\lim_{k\to+\infty}\frac{k^{\frac{2}{\tilde{\nu}}}}{\lambda_{k}}=0.
\end{equation}
Note that \eqref{1-9} merely suggests that $\lambda_{k}$ grows faster than $k^{\frac{2}{\tilde{\nu}}}$ as $k\to +\infty$. This observation gives rise to a natural question: \emph{If $|H|=0$,  what is the exact growth rate of $\lambda_{k}$ as $k\to +\infty$?}

It is worth pointing out that even if $|H|=0$, the class of the homogeneous H\"{o}rmander operators remains quite large and contains many crucial degenerate operators relevant to sub-Riemannian geometry.  For instance, the Grushin operators, Bony operators, and Martinet operators all belong to this class (see Section \ref{Section6}).

The present work aims to provide the explicit estimations of Dirichlet eigenvalues for homogeneous H\"{o}rmander operators. Our research into the asymptotic behaviour of eigenvalues also follows up on our earlier work \cite{chen-chen2020}, in which we examined the Dirichlet eigenvalue problem for general H\"{o}rmander operators on bounded open domains with boundaries that are smooth and non-characteristic for vector fields $X$. However, addressing the Dirichlet eigenvalue problem on general bounded domains (with the boundaries that may be non-smooth and characteristic for $X$), especially in cases where $|H|=0$, leads to many technical issues in estimating the trace of the subelliptic Dirichlet heat kernel. This introduces new difficulties and challenges in the estimation of Dirichlet eigenvalues.

Utilizing the abstract theory of Dirichlet forms and the heat semigroups, we provide an abstract estimation for the trace of the subelliptic Dirichlet heat kernel  on general bounded domains, achieved through the comparison of heat kernels. This estimation is formulated as an abstract integral involving the reciprocal of the volume of the subunit ball.  Then, by employing an innovative approach based on algebraic geometry and convex geometry, we construct explicit asymptotic estimates for the abstract integral, consequently deriving explicit asymptotic estimates for the trace of the subelliptic Dirichlet heat kernel. As a result, we successfully establish explicit asymptotic estimates for Dirichlet eigenvalues of homogeneous H\"{o}rmander operators, yielding a satisfactory answer to the above question.

Our explicit asymptotic results, particularly in the non-equiregular case, may shed light on eigenvalue problems of general H\"{o}rmander operators as well as degenerate elliptic operators, for which much less is known. A recent preprint \cite{Verdiere2022} used different method to study the asymptotic behaviour of the closed eigenvalues for a class of H\"{o}rmander operators on compact smooth manifolds without boundary. However, compared to the compact boundaryless case, studying the Dirichlet eigenvalue problem for degenerate elliptic operators is more complicated and challenging due to the involvement of estimates on the trace of the Dirichlet heat kernel. Furthermore, the method derived here from algebraic and convex geometry offers a novel and direct approach in the investigation of the eigenvalue problem in degenerate cases. For other results on eigenvalue problems of degenerate elliptic operators, readers can refer to \cite{CCD,chen-chen2019,chen-chen-li2022,YV1998,Arendt2009,Kokarev2013,
Hassannezhad2014,Taylor2020} and the references therein.

\emph{\textbf{Notations}}. Throughout this paper, the notation $f(x)\approx g(x)$ is used to indicate that
$C^{-1}g(x)\leq f(x)\leq Cg(x)$, where $C>0$ is a constant that is  independent of the relevant variables in $f(x)$ and $g(x)$.  Moreover, we say that $f(x)\approx g(x)$  as $x\to x_{0}$ if there exist some constants $C>0$ and $\delta>0$ such that $C^{-1}g(x)\leq f(x)\leq Cg(x)$ holds for all $0<|x-x_{0}|<\delta$.

We now present our main results. First, we derive the following estimates for the subelliptic Dirichlet heat kernel of homogeneous H\"{o}rmander operators.

\begin{theorem}
\label{thm1}
Let $X=(X_{1},X_{2},\ldots,X_{m})$ be the homogeneous H\"{o}rmander vector fields defined on $\mathbb{R}^n$.  Suppose that $\Omega$ is a bounded open domain in $\mathbb{R}^n$ containing the origin. Then the Dirichlet heat kernel $h_D(x,y,t)$ of $\triangle_{X}$ on $\Omega$ satisfies
  \begin{equation}\label{1-10}
\int_{\Omega}h_D(x,x,t) dx\approx\int_{\Omega}\frac{dx}{|B_{d_X}(x,\sqrt{t})|}~~ \mbox{as}~~~t\to 0^+,
\end{equation}
where $|B_{d_{X}}(x,r)|$ denotes the $n$-dimensional Lebesgue measure of subunit ball $B_{d_{X}}(x,r)$.
\end{theorem}

Combining the theory of resolution of singularities and some delicate analysis associated with convex geometry, we proceed to provide explicit estimates for the integral on the right-hand side of \eqref{1-10}.

\begin{theorem}
\label{thm2}
Consider the homogeneous H\"{o}rmander vector fields $X=(X_{1},X_{2},\ldots,X_{m})$ defined on  $\mathbb{R}^n$.
Let $Q$ be the homogeneous dimension given by \eqref{1-1}, and set $w:=\min_{x\in \mathbb{R}^n}\nu(x)$. Suppose that $\Omega\subset\mathbb{R}^n$ is an open bounded domain containing the origin, then we have
  \begin{equation}\label{1-11}
\int_{\Omega}\frac{d x}{|B_{d_X}(x,r)|}\approx r^{-Q_0}|\ln r|^{d_0}~~ \mbox{as}~~~r\to 0^+,
\end{equation}
where $Q_0\in \mathbb{Q}^{+}$ is a positive rational number, and $d_{0}$ is a non-negative integer satisfying the following properties:
\begin{enumerate}[(1)]
  \item If $w=Q$, then $Q_{0}=Q$ and $d_{0}=0$;
  \item If $w\leq Q-1$, then
\begin{equation}\label{1-12}
 n\leq\max\{w,Q-\alpha(X)\}\leq Q_0\leq Q-1~~~\mbox{and}~~~d_{0}  \in\{0,1,\ldots,v\}.
\end{equation}
\end{enumerate}
Here, $v\leq n-1$ denotes the number of degenerate components of $X$, and  $\alpha(X)$ is the sum of degenerate indexes, which are defined in Definition \ref{def2-4} below.
\end{theorem}

According to Theorem \ref{thm1} and Theorem \ref{thm2}, we achieve the following explicit asymptotic estimate of Dirichlet eigenvalues.

\begin{theorem}
\label{thm3}
Let $X=(X_{1},X_{2},\ldots,X_{m})$ and $\Omega$ satisfy the conditions of Theorem \ref{thm1}. Denote by $\lambda_{k}$ the $k$-th Dirichlet eigenvalue of problem \eqref{1-5}. Then we have
   \begin{equation}\label{1-13}
    N(\lambda) \approx \int_{\Omega}\frac{dx}{|B_{d_X}(x,\lambda^{-\frac{1}{2}})|}\approx\lambda^{\frac{Q_{0}}{2}}(\ln\lambda)^{d_{0}}~~~\mbox{as}~~\lambda\rightarrow+\infty,
  \end{equation}
and
  \begin{equation}\label{1-14}
    \lambda_k\approx k^{\frac{2}{Q_{0}}}(\ln k)^{-\frac{2d_{0}}{Q_{0}}}~~~\mbox{as}~~ k\to+\infty,
  \end{equation}
where $N(\lambda):=\#\{k|\lambda_{k}\leq \lambda\}$ is the spectral counting function,  $Q_{0}$ and $d_{0}$ are the indexes in Theorem \ref{thm2}.
\end{theorem}

\begin{remark}
\label{remark1-1}
Theorem \ref{thm2} provides an explicit asymptotic behaviour for the integral \eqref{1-11}, which only involves the term $r^{-Q_0}|\ln r|^{d_0}$. It is worth pointing out that the index $Q_0$ is not restricted to positive integers, as illustrated in Example \ref{ex6-4} with $Q_0=\frac{11}{3}$ and $d_0=0$. This non-integer value for $Q_0$ represents a novel phenomenon that has not been encountered before. Moreover,  the bounds of $Q_{0}$ in \eqref{1-12} are optimal. In particular, by employing the blow-up technique in algebraic geometry and utilizing Proposition \ref{prop4-7} below, one can explicitly calculate the values of $Q_0$ and $d_{0}$ for specific homogeneous H\"{o}rmander vector fields. In Section \ref{Section6}, we will provide classical examples to demonstrate the optimality of our results and discuss the calculation method in detail.
\end{remark}

\begin{remark}
Theorem \ref{thm3} implies that the Fefferman-Phong's estimate \eqref{1-7} also holds for the Dirichlet eigenvalue problem \eqref{1-5}. Additionally, for homogeneous vector fields $X$ defined on $\mathbb{R}^{n}$, Proposition \ref{prop2-7} below shows an ingenious relationship between the homogeneous dimension $Q$ and the pointwise dimension $\nu(x)$. It states that if the domain $\Omega$ contains the origin in $\mathbb{R}^n$, then
$Q=\nu(0)=\max_{x\in \overline{\Omega}}\nu(x)=\tilde{\nu}$. Furthermore, according to Corollary \ref{corollary2-2} below, there are only two situations: $w=Q$ and $w\leq Q-1$. If $w=Q$, then $H=\{x\in \Omega|\nu(x)=\tilde{\nu}\}=\Omega$, and our asymptotic estimate \eqref{1-14} is consistent with the asymptotic formula \eqref{1-8}. If $w\leq Q-1$, we have $|H|=0$, and  \eqref{1-14} provides an explicit growth rate for $\lambda_{k}$ as $k\to +\infty$.
\end{remark}
\begin{remark}
Theorem \ref{thm3} also improves our previous estimates of Dirichlet eigenvalues for homogeneous H\"{o}rmander operators in \cite{chen-chen-li2020}.
\end{remark}

The plan of the rest paper is as follows. In Section \ref{Section2}, we introduce some necessary preliminaries, including  homogeneous functions and homogeneous vector fields,  subunit metric and volume estimates of subunit balls, properties of Nagel-Stein-Wainger polynomials, classification of homogeneous H\"{o}rmander vector fields, weighted Sobolev spaces, and subelliptic estimates and Poincar\'{e} inequality of H\"{o}rmander vector fields. In Section \ref{Section3}, we construct the subelliptic global heat kernel and Dirichlet heat kernel for the self-adjoint homogeneous H\"{o}rmander operators from the perspective of the heat semigroup. We also obtain the pointwise estimates of subelliptic Dirichlet heat kernel. In Section \ref{Section4}, we investigate the explicit asymptotic behaviours of integral \eqref{1-11}. Then, we provide the
 proofs of Theorem \ref{thm1}-Theorem \ref{thm3} in Section \ref{Section5}. Finally, in Section \ref{Section6},
we present some related examples  as further applications of Theorem \ref{thm2} and Theorem \ref{thm3}.

\section{Preliminaries}
\label{Section2}

\subsection{$\delta_{t}$-homogeneous functions and vector fields}

We first give a brief review of the definitions and properties of $\delta_{t}$-homogeneous functions and $\delta_{t}$-homogeneous vector fields. One can refer to \cite{Bonfiglioli2007} for a more detailed discussion.

\begin{definition}[$\delta_{t}$-homogeneous function]
\label{def2-1}
A real function $f$ defined on $\mathbb{R}^n$ is called the $\delta_{t}$-homogeneous of degree $\sigma\in \mathbb{R}$ if $f\not\equiv 0$ and $f$ satisfies
\[ f(\delta_{t}(x))=t^{\sigma}f(x)~~~~~\forall x\in \mathbb{R}^{n},~t>0. \]
\end{definition}
According to Definition \ref{def2-1}, if $f$ is a continuous function with $\delta_{t}$-homogeneous degree $\sigma$ and $f(x_{0})\neq 0$ for some $x_{0}\in \mathbb{R}^n$, then $\sigma\geq 0$. Moreover, the continuous and $\delta_{t}$-homogeneous of degree $0$ functions are precisely the non-zero constants (see \cite[p. 33]{Bonfiglioli2007}).

\begin{proposition}[Smooth $\delta_{t}$-homogeneous functions]
\label{prop2-1}
For any $f\in C^{\infty}(\mathbb{R}^{n})$,  $f$ is $\delta_{t}$-homogeneous of degree $\sigma\in \mathbb{N}$ if and only if $f$ is a polynomial function of the form
\begin{equation}\label{2-1}
  f(x)=\sum_{\sum_{i=1}^{n}\alpha_{i}\beta_{i}=\sigma}c_{\beta_{1},\ldots,\beta_{n}}x_{1}^{\beta_{1}}x_{2}^{\beta_{2}}\cdots x_{n}^{\beta_{n}}
\end{equation}
with some $c_{\beta_{1},\ldots,\beta_{n}}\neq 0$, where $\beta_{1},\beta_{2},\ldots,\beta_{n}$ are some non-negative integers.
\end{proposition}
\begin{proof}
See \cite[Proposition 1.3.4]{Bonfiglioli2007}.
\end{proof}
On the other hand, the $\delta_{t}$-homogeneous vector field is defined as follows.

\begin{definition}
\label{def2-2}
Let $Y$ be a non-identically-vanishing linear differential operator defined on $\mathbb{R}^{n}$. We say $Y$ is $\delta_{t}$-homogeneous of degree $\sigma\in \mathbb{R}$ if
\[ Y(\varphi(\delta_{t}(x)))=t^{\sigma}(Y\varphi)(\delta_{t}(x))\qquad \forall \varphi\in C^{\infty}(\mathbb{R}^{n}),~ x\in \mathbb{R}^{n},~ t>0. \]
\end{definition}

The $\delta_{t}$-homogeneous smooth vector field admit the following properties.

\begin{proposition}[Smooth $\delta_{t}$-homogeneous vector fields]
\label{prop2-2}
Suppose $Y$ is a smooth non-vanishing vector field in $\mathbb{R}^{n}$ such that
\[ Y=\sum_{j=1}^{n}\mu_{j}(x)\partial_{x_{j}}. \]
Then $Y$ is $\delta_{t}$-homogeneous of degree $\sigma\in \mathbb{N}$ if and only if $\mu_{j}$ is a polynomial function $\delta_{t}$-homogeneous of degree $\alpha_{j}-\sigma$ in the form of \eqref{2-1} (unless $\mu_{j}\equiv 0$). Moreover, for
each $\mu_{j}$ with $\mu_{j}\not\equiv 0$, we have $\alpha_{j}\geq\sigma$, and $Y$ satisfies
\[ Y=\sum_{j\leq n,~ \alpha_{j}\geq \sigma}\mu_{j}(x)\partial_{x_{j}}. \]
In particular, if $\sigma\geq 1$, since $\mu_{j}$ is a $\delta_{t}$-homogeneous polynomial function of degree $\alpha_{j}-\sigma$, it follows from \eqref{2-1} that $\mu_{j}(x)=\mu_{j}(x_{1},\ldots,x_{j-1})$ does not depend on the variables $x_{j},\ldots,x_{n}$.

\end{proposition}
\begin{proof}
See \cite[Proposition 1.3.5, Remark 1.3.7]{Bonfiglioli2007}.
\end{proof}

For $1\leq j_{i}\leq m$, we let $J=(j_{1},\ldots,j_{k})$ be a multi-index with length $|J|=k$. Then there exists a commutator $X_{J}$ of length $k$  such that
$$X_{J}=[X_{j_{1}},[X_{j_{2}},\ldots[X_{j_{k-1}},X_{j_{k}}]\ldots]].$$
Adopting the notations above, we have
\begin{proposition}
\label{prop2-3}
For $k\geq 1$, let $X^{(k)}=\{X_{J}|J=(j_{1},\ldots,j_{k}),~1\leq j_{i}\leq m, |J|=k \}$  be the set of all commutators of length $k$. Then for any $Y\in X^{(k)}$, $Y$ is the $\delta_{t}$-homogeneous of degree $k$ unless $Y\equiv 0$ (i.e., $Y$ is a zero vector field). In particular, $X^{(k)}=\{0\}$ for any $k>\alpha_{n}$.
\end{proposition}
\begin{proof}
See \cite[Proposition 1.3.10]{Bonfiglioli2007}.
\end{proof}
\begin{remark}
\label{remark2-1}
Proposition \ref{prop2-2} and Proposition \ref{prop2-3} imply that $X^{(k_{1})}\cap X^{(k_{2})}=\{0\}$ for $k_{1}\neq k_{2}$, and
\[ \text{\rm Lie}(X)={\rm span}~ X^{(1)}\oplus \cdots \oplus {\rm span}~ X^{(\alpha_{n})}.\]
\end{remark}

\subsection{Properties of Nagel-Stein-Wainger polynomials}

In this part, we employ standard notations from  \cite{Stein1985} to  further explore the homogeneous H\"{o}rmander vector fields.

 Let $Y_{1},\ldots,Y_{q}$ be an enumeration of the components of $X^{(1)},\ldots,X^{(\alpha_{n})}$. We say $Y_{i}$ has formal degree $d(Y_{i})=k$ if $Y_{i}$ is an element of $X^{(k)}$.
For each $n$-tuple of integers $I=(i_{1},\ldots,i_{n})$ with $1\leq i_{j}\leq q$, we consider the function
\begin{equation}\label{2-2}
\lambda_{I}(x):=\det(Y_{i_{1}},\ldots,Y_{i_{n}})(x),
\end{equation}
where $\det(Y_{i_{1}},\ldots,Y_{i_{n}})(x)=\det(b_{jk}(x))$ with $Y_{i_{j}}=\sum_{k=1}^{n}b_{jk}(x)\partial_{x_{k}}$. We also define
\begin{equation*}
d(I):=d(Y_{i_{1}})+\cdots+d(Y_{i_{n}}),
\end{equation*}
and
\begin{equation}\label{2-3}
  \Lambda(x,r):=\sum_{I}|\lambda_{I}(x)|r^{d(I)},
\end{equation}
where the sum is taken over all $n$-tuples. The function $ \Lambda(x,r)$ is the so-called Nagel-Stein-Wainger polynomial, which describes the volume of subunit balls (see Proposition \ref{prop2-8} below).

For the $\delta_{t}$-homogeneous vector fields $X$,  the function $\lambda_{I}$  admits the following properties.

\begin{proposition}
\label{prop2-4}
Consider the smooth vector fields $X=(X_{1},X_{2},\ldots,X_{m})$  defined on $\mathbb{R}^n$ that satisfy assumption (H.1). Then every $\lambda_{I}$ given by \eqref{2-2} is a polynomial.  Furthermore, $\lambda_{I}$  satisfies $\lambda_I(\delta_t(x))=t^{Q-d(I)}\lambda_I(x)$ and
\begin{enumerate}[(1)]
  \item If $d(I)<Q$, $\lambda_I(0)=0$;
  \item If $d(I)=Q$, $\lambda_I(x)\equiv\lambda_I(0)$;
  \item If $d(I)>Q$, $\lambda_I(x)\equiv0$.
\end{enumerate}
\end{proposition}
\begin{proof}
From Proposition \ref{prop2-2} and Proposition \ref{prop2-3}, as well as \eqref{2-2}, we see that $\lambda_{I}$ is a polynomial for every $n$-tuple $I$. We then show the homogeneity of $\lambda_{I}$.

For each $n$-tuple of integers $I=(i_{1},\ldots,i_{n})$ with $1\leq i_{j}\leq q$, we let
$Y_{i_{1}},\ldots,Y_{i_{n}}$ be the corresponding vector fields of $I$ such that $Y_{i_{j}}=\sum_{k=1}^{n}b_{jk}(x)\partial_{x_{k}}\in X^{(d(Y_{i_{j}}))}$ for $j=1,\ldots,n$.
 Then, Proposition \ref{prop2-3} implies that $Y_{i_{j}}$ is $\delta_t$-homogeneous of degree $d(Y_{i_{j}})$. By Proposition \ref{prop2-2}, we can deduce that $b_{jk}(x)$ is either zero function or $\delta_t$-homogeneous polynomial function of degree $\alpha_{k}-d(Y_{i_{j}})$. Thus, it follows from \eqref{2-2} that for any $t>0$,
\begin{equation}\label{2-4}
 \begin{aligned}
\lambda_{I}(\delta_{t}(x))&=\det(b_{jk}(\delta_{t}(x)))
=\det(t^{\alpha_{k}-d(Y_{i_{j}})}b_{jk}(x))\\
&=\left(\prod_{k=1}^{n}t^{\alpha_{k}} \right)\left(\prod_{j=1}^{n}t^{-d(Y_{i_{j}})}  \right)\lambda_{I}(x)
=t^{Q-d(I)}\lambda_{I}(x).
\end{aligned}
\end{equation}

If $d(I)<Q$, we let $t\to 0^+$ in \eqref{2-4}, which derives  that $\lambda_I(0)= 0$. For the case where $d(I)=Q$, we let $t\to 0^+$ in \eqref{2-4} again, yielding $\lambda_I(x)\equiv\lambda_I(0)$. Finally, if $d(I)>Q$, we can use  \eqref{2-4} to obtain  $t^{d(I)-Q}\lambda_{I}(\delta_t(x))=\lambda_{I}(x)$, from which it follows that $\lambda_I(x)\equiv0$ by
taking $t\to 0^{+}$. This completes the proof of  Proposition \ref{prop2-4}.
\end{proof}

Proposition \ref{prop2-4} allows us to derive the validity of the H\"{o}rmander condition from the assumptions (H.1) and (H.2).

\begin{proposition}
\label{prop2-5}
Let $X=(X_{1},X_{2},\ldots,X_{m})$ be homogeneous H\"{o}rmander vector fields defined on  $\mathbb{R}^n$
 (i.e., $X$ satisfy the assumptions (H.1) and (H.2)), then $X$ satisfy H\"{o}rmander condition in $\mathbb{R}^n$ with the H\"{o}rmander index $r=\alpha_n$.
\end{proposition}

\begin{proof}
It derives from Proposition \ref{prop2-4} that $\lambda_{I}(0)=0$ holds for any $n$-tuple $I$ with  $d(I)\neq Q$. If $\lambda_I(0)=0$ for all $n$-tuples $I$ such that $d(I)= Q$, then any $n$ vector fields $Y_{i_{1}},\ldots, Y_{i_{n}}$ belonging to $\bigcup_{k= 1}^{\alpha_{n}}X^{(k)}$ cannot span $\mathbb{R}^n$ at the origin, which implies
\[ \dim\{Y(0)|~Y\in \text{Lie}(X)\}<n. \]
This contradicts assumption (H.2). Hence,  there exists an $n$-tuple $I_{0}$ such that $d(I_{0})= Q$ and $\lambda_{I_{0}}(x)\equiv \lambda_{I_{0}}(0)\neq 0$. As a result
\[ \dim\{Y(x)|~Y\in \text{Lie}(X)\}=n~~~~~\forall x\in \mathbb{R}^n, \]
which yields the validity of H\"{o}rmander condition in $\mathbb{R}^n$. Furthermore, Proposition \ref{prop2-3} indicates the H\"{o}rmander index $r\leq \alpha_{n}$.\par

We proceed to show that the H\"{o}rmander index $r=\alpha_n$. By Remark \ref{remark2-1},
\[D_{x}^{j}=\text{\rm span}\{Y(x)|Y\in X^{(1)}\cup \cdots\cup X^{(j)}\}=\bigoplus_{i=1}^{j}\text{\rm span}\{Y(x)|Y\in X^{(i)}\}.\]
If for any $Y\in \bigcup_{i=1}^{\alpha_{n}-1}X^{(i)}$, the last coefficient function $\mu_{n}$ of $Y(x)=\sum_{j=1}^{n}\mu_{j}(x)\partial_{x_{j}}$ vanishes identically, then assumption (H.2) fails. Thus,
 there exists a vector field $Y_{0}\in \bigcup_{i=1}^{\alpha_{n}-1}X^{(i)}$ such that $Y_{0}(x)=\sum_{j=1}^{n}\mu_{j}(x)\partial_{x_{j}}$ with $\mu_{n}\not\equiv 0$. We next verify that the $\delta_t$-homogeneous degree of $\mu_{n}(x)$ cannot be zero. Observe that $Y_{0}\in X^{(i)}$
 for some $1\leq i\leq \alpha_{n}-1$. From Proposition \ref{prop2-2} and Proposition \ref{prop2-3}, we conclude that $\mu_{n}$ is a smooth $\delta_t$-homogeneous function of degree $\alpha_{n}-i$ with $\alpha_{n}-i\geq 1$. This indicates $\mu_{n}(0)=0$ and $\dim D_{0}^{\alpha_{n}-1}<n$. Therefore, the H\"{o}rmander index $r\geq \alpha_n$, and we have $r=\alpha_n$.
\end{proof}

Moreover, the Nagel-Stein-Wainger polynomial $\Lambda(x,r)$ has the following properties.

\begin{proposition}
\label{prop2-6}
For any $x\in \mathbb{R}^n$, the pointwise homogeneous dimension $\nu(x)$ satisfies
\begin{equation}\label{2-5}
  \nu(x)=\min\{d(I)|\lambda_{I}(x)\neq 0\}=\lim_{s\to 0^{+}}\frac{\ln
  \Lambda(x,s)}{\ln s}.
\end{equation}
\end{proposition}
\begin{proof}
 See \cite[Proposition 2.2]{chen-chen2019}.
\end{proof}

\begin{proposition}
\label{prop2-7}
Let $X=(X_{1},X_{2},\ldots,X_{m})$ be the homogeneous H\"{o}rmander vector fields defined on  $\mathbb{R}^n$. Then $n\leq w \leq\nu(x)\leq Q$ and $\nu(0)=Q$, where $w:=\min_{x\in \mathbb{R}^n}\nu(x)$ and $Q$ is the homogeneous dimension defined in \eqref{1-1}.  Additionally, the Nagel-Stein-Wainger polynomial $\Lambda(x,r)$ satisfies
\begin{equation}\label{2-6}
    \Lambda(x,r)=\sum_{k=\nu(x)}^{Q}f_k(x)r^k=\sum_{k=w}^{Q}f_k(x)r^k,
  \end{equation}
where $f_k(x)=\sum_{d(I)=k}|\lambda_I(x)|$ is a non-negative continuous $\delta_t$-homogeneous function of degree $Q-k$. Furthermore,  $f_{w}(x_{0})\neq 0$ for some $x_{0}\in  \mathbb{R}^n$ and $f_Q(x)=f_{Q}(0)>0$ for all $x\in \mathbb{R}^n$.
\end{proposition}
\begin{proof}
Observing that $\nu(x)$ is a function with  integer values and $\nu(x)\geq n$, we obtain  $\nu(x_{0})=\min\limits_{x\in \mathbb{R}^n}\nu(x)=w$ for some $x_{0}\in  \mathbb{R}^n$. The proof of Proposition \ref{prop2-5} implies that
 $\lambda_{I_{0}}(x)\equiv \lambda_{I_{0}}(0)\neq 0$ holds for some $n$-tuple $I_{0}$ satisfying $d(I_{0})=Q$. Thus,  by \eqref{2-5} we have $\nu(x)\leq Q$. Additionally,  Proposition \ref{prop2-4} yields that $\lambda_{I}(0)=0$ for $d(I)\neq Q$, which indicates that $\nu(0)=Q$.

Combining \eqref{2-3}, Proposition \ref{prop2-4} and Proposition \ref{prop2-6},
\begin{equation}\label{2-7}
  \Lambda(x,r)=\sum_{I}|\lambda_I(x)|r^{d(I)}=\sum_{k=\nu(x)}^Qf_k(x)r^k,
\end{equation}
where $f_k(x)=\sum\limits_{d(I)=k}|\lambda_I(x)|$. Moreover, using Proposition \ref{prop2-6} and \eqref{2-7}, we obtain
\begin{equation}\label{2-8}
w=\nu(x_{0})=\lim_{r\to 0^{+}}\frac{\ln\left(
  \sum_{k=\nu(x_{0})}^Qf_k(x_{0})r^k\right)}{\ln r}=\lim_{r\to 0^{+}}\frac{\ln\left(
  \sum_{k=w}^Qf_k(x_{0})r^k\right)}{\ln r},
\end{equation}
  which gives $f_w(x_{0})\neq 0$. That means
  \[ \Lambda(x,r)=\sum_{k=\nu(x)}^{Q}f_k(x)r^k=\sum_{k=w}^{Q}f_k(x)r^k.\]
Finally,  Proposition \ref{prop2-4} indicates that $f_Q(x)=f_Q(0)\geq |\lambda_{I_{0}}(0)|>0$ for all $x\in \mathbb{R}^n$.
\end{proof}

\subsection{Subunit metric and volume estimates of subunit balls}

The essential geometric object  we are interested in is the subunit metric constructed on $\mathbb{R}^n$ using the vector fields $X=(X_{1},X_{2},\ldots,X_{m})$ that satisfy H\"{o}rmander's condition.  This metric plays a crucial role in estimating the heat kernel of homogeneous H\"{o}rmander operator.

\begin{definition}[Subunit metric, see \cite{Stein1985,Morbidelli2000}]
\label{def2-3}
For any $x,y\in \mathbb{R}^n$ and $\delta>0$, let $C(x,y,\delta)$ be the collection of absolutely continuous mapping $\varphi:[0,1]\to \mathbb{R}^n$, which satisfies $\varphi(0)=x,\varphi(1)=y$ and
\[ \varphi'(t)=\sum_{i=1}^{m}a_{i}(t)X_{i}(\varphi(t)) \]
with $\sum_{k=1}^{m}|a_{k}(t)|^2\leq \delta^2$ for a.e. $t\in [0,1]$.
The subunit metric $d_{X}(x,y)$ is defined by
    \[ d_{X}(x,y):=\inf\{\delta>0~|~ \exists \varphi\in C(x,y,\delta)~\mbox{with}~\varphi(0)=x,~ \varphi(1)=y\}. \]
\end{definition}

The Chow-Rashevskii theorem guarantees that the subunit metric $d_{X}(x,y)$ is well-defined (see \cite[Theorem 57]{Bramanti2014}). On the other hand, we would like to mention that the sub-Riemannian manifold $(\mathbb{R}^n,D,g)$ naturally possesses a metric space structure with the Carnot-Carath\'{e}odory distance induced by the sub-Riemannian metric $g$.  According to \cite[Proposition 3.1]{Jerison1987}, the Carnot-Carath\'{e}odory distance on the sub-Riemannian manifold $(\mathbb{R}^n,D,g)$ is equivalent to the subunit metric $d_{X}$. Consequently, we exclusively consider the subunit metric $d_{X}$ throughout this paper.

Given any $x\in \mathbb{R}^n$ and $r>0$, we denote by
    \[ B_{d_{X}}(x,r):=\{y\in \mathbb{R}^n~|~d_{X}(x,y)<r\} \]
 the subunit ball associated with subunit metric $d_{X}(x,y)$. Owing to assumption (H.1), the subunit metric $d_{X}$ and subunit ball $B_{d_{X}}(x,r)$ enjoy the following properties (see \cite{Biagi-Bonfiglioli-Bramanti2019}):
   \begin{enumerate}[(1)]
   \item For any $x,y\in \mathbb{R}^n$ and $t>0$, $d_{X}(\delta_{t}(x),\delta_{t}(y))=td_{X}(x,y)$.
\item For any $x,y\in \mathbb{R}^n$ and $t,r>0$, $y\in B_{d_{X}}(x,r)$ if and only if $\delta_{t}(y)\in B_{d_{X}}(\delta_{t}(x),r)$.
\item For any $t,r>0$ and $x\in \mathbb{R}^n$, $|B_{d_{X}}(\delta_{t}(x),tr)|=t^{Q}|B_{d_{X}}(x,r)|$,  where $|B_{d_{X}}(x,r)|$ denotes the $n$-dimensional Lebesgue measure of $B_{d_{X}}(x,r)$.
   \end{enumerate}

Furthermore, we have the following the volume estimates of subunit ball.

\begin{proposition}[Global version Ball-Box theorem]
\label{prop2-8}
For any $ x\in \mathbb{R}^n $, there exist positive constants $C_{1},C_{2}$  such that for any $x\in \mathbb{R}^n$ and any $r>0$,
\begin{equation}\label{2-9}
  C_{1}\Lambda(x,r)\leq |B_{d_{X}}(x,r)|\leq C_{2}\Lambda(x,r),
\end{equation}
where $|B_{d_{X}}(x,r)|$ is the $n$-dimensional Lebesgue measure of $B_{d_{X}}(x,r)$.
\end{proposition}
\begin{proof}
The local version of Ball-Box theorem was obtained from the deep investigations into subelliptic metric and subunit metric carried out by Nagel-Stein-Wainger \cite{Stein1985} and Morbidelli \cite{Morbidelli2000}. Specifically, if the vector fields $X$ satisfy only H\"{o}rmander's condition, then \eqref{2-9} is constrained to a compact set $K\subset \mathbb{R}^n$ and $0<r\leq r_{0}$, where $r_{0}$ is a positive constant that depends on the vector fields $X$ and the compact set $K$. However, based on the homogeneous assumption (H.1), the global version of Ball-Box theorem can be derived using a local-to-global homogeneity argument, starting from the local version of Ball-Box theorem. A proof of this global estimate can be found in  \cite[Theorem B]{Biagi-Bonfiglioli-Bramanti2019}.
\end{proof}
Proposition \ref{prop2-8} yields the following corollary.
\begin{corollary}
\label{corollary2-1}
For any $x\in \mathbb{R}^n$ and $0<r_1<r_2$, there exists a positive constant $C_{3}>0$ such that
\begin{equation}\label{2-10}
 |B_{d_{X}}(x,r_{2})|\leq C_{3}\left(\frac{r_{2}}{r_{1}} \right)^{Q}|B_{d_{X}}(x,r_{1})|.
\end{equation}
\end{corollary}
\subsection{Classification of homogeneous H\"{o}rmander vector fields}
In this part, we discuss the classification of homogeneous H\"{o}rmander vector fields. For this purpose, we give some useful propositions and definitions.

\begin{proposition}
\label{prop2-9}
 Let $X=(X_{1},X_{2},\ldots,X_{m})$ be the homogeneous H\"{o}rmander vector fields defined on  $\mathbb{R}^n$. Denote by $W:=\{x\in\mathbb{R}^n|\nu(x)=Q\}$. Then we have
  \begin{enumerate}[(i)]
    \item $w=Q$ if and only if $W=\mathbb{R}^n$;
    \item $w\leq Q-1$ if and only if $|W|=0$, where $|W|$ denotes the $n$-dimensional Lebesgue measure of $W$.
  \end{enumerate}
\end{proposition}
\begin{proof}
If $w=Q$, Proposition \ref{prop2-7} implies $\nu(x)=Q$ for any $x\in \mathbb{R}^n$, which yields $W=\mathbb{R}^n$. Clearly, $W=\mathbb{R}^n$ gives $w=\min_{x\in \mathbb{R}^n}\nu(x)=Q$. Thus, conclusion (i) is proved.

We then show conclusion (ii). Suppose that $w\leq Q-1$, by Proposition \ref{prop2-7} we have $f_{w}(x_{0})\neq 0$ for some $x_{0}\in \mathbb{R}^n$. Denoting by $Z(f_w):=\{x\in\mathbb{R}^n|f_w(x)=0\}$  the zeros of function $f_{w}(x)$, we obtain from Proposition \ref{prop2-4} (1) that
\[ Z(f_w)= \bigcap_{d(I)=w}\{x\in\mathbb{R}^n|\lambda_{I}(x)=0\}\neq \varnothing. \]
Since $\lambda_{I}$ is the polynomial for each $n$-tuple $I$, \cite{Mityagi2020} derives that $|\{x\in\mathbb{R}^n|\lambda_{I}(x)=0\}|=0$ and $|Z(f_w)|=0$. Moreover,
for any $x\in W$, we have $w<\nu(x)=Q$. Combining \eqref{2-5} and \eqref{2-6},
\[ \nu(x)=Q=\lim_{r\to 0^{+}}\frac{\ln\left(
  \sum_{k=\nu(x)}^Qf_k(x)r^k\right)}{\ln r}=\lim_{r\to 0^{+}}\frac{\ln\left(
  \sum_{k=w}^Qf_k(x)r^k\right)}{\ln r},\]
which implies $f_{w}(x)=0$. This means $W\subset Z(f_w)$ and $|W|=0$. On the other hand, assuming that $|W|=0$, if $w=Q$, by conclusion (i) we would have $W=\mathbb{R}^n$, which contradicts the fact $|W|=0$.
\end{proof}

Proposition \ref{prop2-9} gives the following obvious corollary.

\begin{corollary}
\label{corollary2-2}
Suppose that  $X=(X_{1},X_{2},\ldots,X_{m})$ and $\Omega$ satisfy the conditions of
Theorem \ref{thm1}.  Then $Q=\tilde{\nu}$, and the set $H=\{x\in \Omega|\nu(x)=\tilde{\nu}\}=W\cap \Omega$ satisfies the following conclusions:
\begin{enumerate}[(1)]
    \item $w=Q$ if and only if $H=\Omega$;
    \item $w\leq Q-1$ if and only if $|H|=0$, where $|H|$ is the $n$-dimensional Lebesgue measure of $H$.
  \end{enumerate}
\end{corollary}

We now introduce the concepts of degenerate component and degenerate index, which will be useful in Section \ref{Section4} below.

\begin{definition}[Degenerate component and degenerate index]
\label{def2-4}
Let $X=(X_{1},X_{2},\ldots,X_{m})$ be the homogeneous H\"{o}rmander vector fields defined on  $\mathbb{R}^n$, and let $x=(x_{1},x_{2},\ldots,x_{n})$ be the point in $\mathbb{R}^n$. For each $1\leq j\leq n$, if the function $\Lambda(x,r)$ depends on the variable $x_{j}$,  we say that $x_{j}$ is the degenerate component of $X$. Otherwise, we call $x_j$ the non-degenerate component of $X$. Furthermore, we specify $\alpha_{j}$ as the degenerate index associated with $x_{j}$, where $\alpha_{j}$ is determined by the dilation $\delta_{t}$. We denote the sum of all degenerate indexes of the vector fields $X$ by
\begin{equation}\label{2-11}
  \alpha(X):=\sum_{j\in \mathcal{A}}\alpha_{j},
\end{equation}
where $\mathcal{A}:=\{1\leq j\leq n|x_{j}~\text{is the degenerate component of}~X\}$.
\end{definition}

\begin{remark}
\label{remark2-2}
From Proposition \ref{prop2-2}, we know the last variable $x_{n}$ is the non-degenerate component of $X$.
\end{remark}
\begin{remark}
The degenerate components of vector fields $X$ describe the  variable dependency of $\Lambda(x,r)$, enabling us to handle $\Lambda(x,r)$ in lower-dimensional space. It should be noted that the degenerate components of vector fields $X$ differ from the dependent variables of vector fields $X$. For instance, consider the vector fields $X_{1}=\partial_{x_{1}}+2x_{2}\partial_{x_{3}},X_{2}=\partial_{x_{2}}-2x_{1}\partial_{x_{3}}$ on $\mathbb{R}^{3}$. Although  $X_{1}$ and $X_{2}$ depend on the variables $x_{1},x_{2}$, the Nagel-Stein-Wainger polynomial $\Lambda(x,r)=24r^{4}$ is independent of the variables of $x$.
\end{remark}

Based on Proposition \ref{prop2-7} and Definition \ref{def2-4},
we can classify  homogeneous H\"{o}rmander vector fields as follows.
\begin{proposition}
\label{prop2-10}
Assuming that  $X=(X_{1},X_{2},\ldots,X_{m})$ and $\Omega$ satisfy the conditions of
Theorem \ref{thm1}, we  have the following classifications:
\begin{enumerate}[(a)]
    \item $w=Q$ if and only if all variables $x_{1},\ldots,x_{n}$ are non-degenerate components of vector fields $X$;
    \item $w\leq Q-1$ if and only if the vector fields $X$ have at least one degenerate component.
  \end{enumerate}
\end{proposition}

\subsection{Weighted Sobolev spaces of H\"{o}rmander vector fields}

We next introduce the weighted Sobolev spaces associated with general H\"{o}rmander vector fields $X=(X_{1},X_{2},\ldots,X_{m})$, which are the natural spaces when dealing with problems related to the H\"{o}rmander operators.

Suppose that $X=(X_{1},X_{2},\ldots,X_{m})$ are the smooth vector fields defined on $\mathbb{R}^n$, satisfying
 H\"{o}rmander's condition. The weighted Sobolev space, also known as the Folland-Stein space (cf. \cite{xu1992}), is a Hilbert space on $\mathbb{R}^n$  defined as
 \[ H_{X}^{1}(\mathbb{R}^n)=\{u\in L^{2}(\mathbb{R}^n)~|~X_{j}u\in L^{2}(\mathbb{R}^n),~ j=1,\ldots,m\}, \]
and endowed with the norm \[\|u\|^2_{H^{1}_{X}(\mathbb{R}^n)}=\|u\|_{L^2(\mathbb{R}^n)}^2+\|Xu\|_{L^2(\mathbb{R}^n)}^2=\|u\|_{L^2(\mathbb{R}^n)}^2+\sum_{j=1}^{m}\|X_{j}u\|_{L^2(\mathbb{R}^n)}^2. \]

For any open set  $\Omega\subset \mathbb{R}^n$, we denote by
   $H^{1}_{X,0}(\Omega)$ the closure of $C_{0}^{\infty}(\Omega)$ in $H_{X}^{1}(\mathbb{R}^n)$. It is well-known that $H^{1}_{X,0}(\Omega)$ is also a Hilbert space. In particular, when $\Omega=\mathbb{R}^n$, we have the following density result.
\begin{proposition}
\label{prop2-11}
Let $X=(X_{1},X_{2},\ldots,X_{m})$ be the smooth vector fields defined on $\mathbb{R}^n$, satisfying
H\"{o}rmander's condition and assumption (H.1). Then the space $C_ {0}^{\infty}(\mathbb{R}^n)$ is dense in $H_{X}^{1}(\mathbb{R}^n)$, which means $H_ {X,0}^{1}(\mathbb{R}^n)=H_{X}^{1}(\mathbb{R}^n)$.
\end{proposition}

\begin{proof}
  The Meyers-Serrin's theorem (refer to \cite[Theorem 1.13]{Garofalo1996}) asserts that
for any $u\in H_{X}^{1}(\mathbb{R}^n)$, there exists a sequence $\{u_{k}\}_{k=1}^{\infty}\subset C^{\infty}(\mathbb{R}^n)\cap H_{X}^{1}(\mathbb{R}^n)$ such that
\begin{equation}\label{2-12}
  \lim_{k\to +\infty}\|u_{k}-u\|_{H_{X}^{1}(\mathbb{R}^n)}=0.
\end{equation}
Let $\zeta\in C_{0}^{\infty}(\mathbb{R}^{n})$ be a cut-off function with $0\leq \zeta\leq 1$, $\zeta\equiv 1 $ on $B_{1}(0)$, and $\text{supp}~\zeta \subset B_{2}(0)$. Here, $B_r(0)=\{x\in \mathbb{R}^{n}||x|<r\}$ denotes the classical Euclidean ball in $\mathbb{R}^n$ centered  at origin with radius $r$. Denoting by $u_{l,k}(x)=u_{k}\cdot \zeta(\delta_{\frac{1}{l}}(x))$, we have $u_{l,k}\in C_{0}^{\infty}(\mathbb{R}^n)$. For each $1\leq j\leq m$, by the homogeneous property (H.1) we have
$ X_{j}u_{l,k}=\zeta(\delta_{\frac{1}{l}}(x))X_{j}u_{k}+\frac{1}{l}(X_{j}\zeta)(\delta_{\frac{1}{l}}(x))u_{k}$.
Note that $X_{j}\zeta\in  L^{\infty}(\mathbb{R}^n)$. Thus, we have for $l\to+\infty$,
\begin{equation}\label{2-13}
\|X_{j}u_{l,k}-X_{j}u_{k}\|_{L^2(\mathbb{R}^n)}\leq \left(\int_{\{x\in \mathbb{R}^n||\delta_{\frac{1}{l}}(x)|\geq 1\}}|X_{j}u_{k}|^{2}dx\right)^{\frac{1}{2}}+\frac{\|X_{j}\zeta\|_{L^{\infty}(\mathbb{R}^n)}}{l}\|u_{k}\|_{L^2(\mathbb{R}^n)}\to 0,
\end{equation}
and
\begin{equation}\label{2-14}
  \|u_{l,k}-u_{k}\|_{L^2(\mathbb{R}^n)}\leq\left(\int_{\{x\in \mathbb{R}^n||\delta_{\frac{1}{l}}(x)|\geq 1\}}|u_{k}|^{2}dx\right)^{\frac{1}{2}}\to 0.
\end{equation}
By \eqref{2-12}-\eqref{2-14}, we conclude that $C_{0}^{\infty}(\mathbb{R}^n)$ is dense in $H_{X}^{1}(\mathbb{R}^n)$, and $H_{X,0}^{1}(\mathbb{R}^n)=H_{X}^{1}(\mathbb{R}^n)$.
\end{proof}

Then, we have the following chain rules in weighted Sobolev space $H_{X}^{1}(\mathbb{R}^n)$.

\begin{proposition}[Chain rules]
\label{prop2-12}
Suppose that $F\in C^{1}(\mathbb{R})$ with $F'\in L^{\infty}(\mathbb{R})$. Then for any $u\in H_{X}^{1}(\mathbb{R}^n)$  we have
\begin{equation}
  X_{j}(F(u))=F'(u)X_{j}u\qquad \mbox{in}~~\mathcal{D}'(\mathbb{R}^n)\qquad \mbox{for}~~j=1,\ldots,m.
\end{equation}
Moreover, if $F(0)=0$, then we further have $F(u)\in H_{X}^{1}(\mathbb{R}^n)$.

\end{proposition}
\begin{proof}
Let $\eta\in C_{0}^{\infty}(\mathbb{R}^n)$ be a non-negative function with $\text{supp}~\eta\subset \{x\in \mathbb{R}^n| |x|\leq 1\}$ and $\int_{\mathbb{R}^n}\eta(x)dx=1$. For $\varepsilon>0$, we denote by $\eta_{\varepsilon}=\frac{1}{\varepsilon^n}\eta\left(\frac{x}{\varepsilon}\right)$ the classical mollifier. For any $v\in L^2(\mathbb{R}^n)$, we set
\[ (J_{\varepsilon}v)(x)=\int_{\mathbb{R}^n}v(y)\eta_{\varepsilon}(x-y)dy. \]
It is well-known that $J_{\varepsilon}v\in C^{\infty}(\mathbb{R}^n)\cap L^2(\mathbb{R}^n)$, $J_{\varepsilon}v(x)\to v(x)$ for almost all $x\in \mathbb{R}^n$, and $\|J_{\varepsilon}v-v\|_{L^2(\mathbb{R}^n)}\to 0$ as $\varepsilon\to 0$.

Suppose that $u\in H_{X}^{1}(\mathbb{R}^n)$. For any $\varepsilon>0$, by using the inequality
\[ |F(u)|\leq |F(u)-F(J_{\varepsilon}u)|+|F(J_{\varepsilon}u)|\leq C|u-J_{\varepsilon}u|+|F(J_{\varepsilon}u)|,\]
we can deduce that $F(u)\in L_{\rm loc}^{1}(\mathbb{R}^n)$ since $F(J_{\varepsilon}(u))\in C^{1}(\mathbb{R}^n)$. Therefore, $F(u)\in \mathcal{D}'(\mathbb{R}^n)$. For any test function $\varphi\in C_{0}^{\infty}(\mathbb{R}^n)$, we let $K$ be the compact set such that $\text{supp}~\varphi\subset K\subset \mathbb{R}^n$. Then, using \cite[Lemma A.3, Proposion A.4]{Garofalo1996} we have for any $j=1,\ldots,m$,
\begin{equation*}
\begin{aligned}
\int_{\mathbb{R}^n}F(u)X_{j}^{*}\varphi dx&=\lim_{\varepsilon \to 0}\int_{K}F(J_{\varepsilon}u)X_{j}^{*}\varphi dx=\lim_{\varepsilon \to 0}\int_{K}F'(J_{\varepsilon}u)X_{j}(J_{\varepsilon}u)\varphi dx\\
&=\int_{K}\varphi F'(u)X_{j}u dx=\int_{\mathbb{R}^n}\varphi F'(u)X_{j}u dx.
\end{aligned}
\end{equation*}
Thus, $X_{j}(F(u))=F'(u)X_{j}u$ in $\mathcal{D}'(\mathbb{R}^n)$ for $j=1,\ldots,m$.

If $F(0)=0$ and $F'$ belongs to $L^{\infty}(\mathbb{R})$, it follows that $|F(u)|\leq C|u|$. This inequality implies that $F(u)\in H_{X}^{1}(\mathbb{R}^n)$ for any $u\in H_{X}^{1}(\mathbb{R}^n)$.
\end{proof}

According to Proposition \ref{prop2-12}, we have the following corollary.
\begin{corollary}
\label{corollary2-3}
For any $u\in H_{X}^{1}(\mathbb{R}^n)$,  the functions $u_{+}, u_{-},|u|\in H_{X}^{1}(\mathbb{R}^n)$ and satisfy
  \[ Xu_{+}=\left\{
            \begin{array}{ll}
              Xu, & \hbox{on $\{x\in \mathbb{R}^n|~u(x)> 0\}$,} \\
              0, & \hbox{on $\{x\in \mathbb{R}^n|~u(x)\leq 0\}$;}
            \end{array}
          \right.\qquad Xu_{-}=\left\{
            \begin{array}{ll}
              0, & \hbox{on $\{x\in \mathbb{R}^n|~u(x)\geq 0\}$,} \\
              -Xu, & \hbox{on $\{x\in \mathbb{R}^n|~u(x)<0\}$;}
            \end{array}
          \right.\]
in $\mathcal{D}'(\mathbb{R}^n)$. Furthermore, we have  $X|u|=\text{sgn}(u)Xu$, where $\text{sgn}(u)$ denotes the sign of
$u$. In addition, for any non-negative constant $c\geq 0$,  $(u-c)_+\in H_{X}^1(\mathbb{R}^n)$ and \[ X(u-c)_+=\left\{
            \begin{array}{ll}
              Xu, & \hbox{on $\{x\in \mathbb{R}^n|~u(x)> c\}$,} \\
              0, & \hbox{on $\{x\in \mathbb{R}^n|~u(x)\leq 0\}$;}
            \end{array}
          \right.\quad \mbox{in}~~~\mathcal{D}'(\mathbb{R}^n).\]
\end{corollary}
\begin{proof}
We only verify that for any non-negative constant $c\geq 0$,  $(u-c)_+\in H_{X}^1(\mathbb{R}^n)$ and \[ X(u-c)_+=\left\{
            \begin{array}{ll}
              Xu, & \hbox{on $\{x\in\mathbb{R}^n|~u(x)> c\}$,} \\
              0, & \hbox{on $\{x\in \mathbb{R}^n|~u(x)\leq 0\}$;}
            \end{array}
          \right.\quad \mbox{in}~~~\mathcal{D}'(\mathbb{R}^n).\]
 The proof for the other situation is similar and we omit here.

For any $\varepsilon>0$ and $c\geq 0$, we define
\[ F_\varepsilon(x)=\left\{
            \begin{array}{ll}
             [(x-c)^2+\varepsilon^2]^{\frac{1}{2}}-\varepsilon, & \hbox{$x>c$;} \\
               0, & \hbox{$x\leq c$.}
            \end{array}
          \right.\]
  Then $F_ {\varepsilon}(0)=0$, $F_\varepsilon\in C^1(\mathbb{R})$, $|F_\varepsilon'(x)|\leq 1$ and

 \[ F_\varepsilon'(x)=\left\{
            \begin{array}{ll}
\frac{x-c}{\sqrt{(x-c)^2+\varepsilon^2}}, & \hbox{$x>c$;} \\
               0, & \hbox{$x\leq c$.}
            \end{array}
          \right.\]
 For any $u\in H_{X}^{1}(\mathbb{R}^n)$ and $j=1,\ldots, m$, Proposition \ref{prop2-12} yields that
\[ \int_{\mathbb{R}^n}F_{\varepsilon}(u)X_{j}^{*}\varphi dx=\int_{\mathbb{R}^n}\varphi F_{\varepsilon}'(u)X_{j}u dx,\qquad \forall \varphi\in C_{0}^{\infty}(\mathbb{R}^n). \]
Letting $\varepsilon\to 0^{+}$, by dominated convergence theorem we obtain
\[ \int_{\mathbb{R}^n}(u-c)_{+}X_{j}^{*}\varphi dx=\int_{\{x\in \mathbb{R}^n| u(x)>c\}}\varphi X_{j}u dx, \qquad\forall \varphi\in C_{0}^{\infty}(\mathbb{R}^n). \]
Therefore, $X_{j}(u-c)_{+}\in L^2(\mathbb{R}^n)$. In addition, the inequality $(u-c)_{+}\leq |u|$ derives that $(u-c)_{+}\in L^2(\mathbb{R}^n)$, which indicates that $(u-c)_{+}\in H_{X}^{1}(\mathbb{R}^n)$.
\end{proof}

\begin{proposition}
\label{prop2-13}
Let $\Omega$ be an open bounded subset of $\mathbb{R}^n$. For any $u\in H_{X}^{1}(\mathbb{R}^n)$, if  $\text{supp}~u$ is a compact subset of $\Omega$, then $u\in H_{X,0}^{1}(\Omega)$.
\end{proposition}
\begin{proof}
Since $\text{supp}~u$ is a compact subset in $\Omega$, there exists a function $f\in C_{0}^{\infty}(\mathbb{R}^n)$ such that $f\equiv 1$ on $\text{supp}~u$ and $\text{supp}~f\subset \Omega$. Owing to the Meyers-Serrin's theorem (see \cite[Theorem 1.13]{Garofalo1996}), we can find  a sequence $\{\psi_{i}\}\subset C^{\infty}(\mathbb{R}^n)\cap H_{X}^{1}(\mathbb{R}^n)$ such that $\psi_{i}\to u$ in $H_{X}^{1}(\mathbb{R}^n)$. Observing that $f\psi_{i}\in C_{0}^{\infty}(\Omega)$ and
\[ \begin{aligned}
\|f\psi_i-u\|_{H_{X}^{1}(\mathbb{R}^n)}^{2}&=\|f\psi_i-fu\|_{H_{X}^{1}(\mathbb{R}^n)}^{2}\leq  \|f(\psi_i-u)\|_{L^2(\Omega)}^{2}+\|X(f\psi_i)-X(fu)\|_{L^2(\Omega)}^{2}\\
&\leq C(\|\psi_i-u\|_{L^2(\Omega)}^{2}+\|X\psi_i-Xu\|_{L^2(\Omega)}^{2})
\end{aligned}
\]
holds for some positive constant $C>0$, we conclude that $u\in H_{X,0}^{1}(\Omega)$.
\end{proof}

\subsection{Subelliptic estimates and Poincar\'{e} inequality}

We next recall the following subelliptic estimates:

\begin{proposition}[Subelliptic estimates I]
\label{subelliptic-estimate-1}
Let $X=(X_{1},X_{2},\ldots,X_{m})$ be the smooth vector fields defined on $\mathbb{R}^n$, satisfying
H\"{o}rmander's condition with H\"{o}rmander index $r$. For any open bounded subset $\Omega\subset \mathbb{R}^n$, there exists a constant $C>0$ such that
\begin{equation}\label{2-16}
  \|u\|_{H^{\frac{1}{r}}(\mathbb{R}^n)}^2\leq C\left(\sum_{i=1}^{m}\|X_{i}u\|_{L^2(\mathbb{R}^n)}^{2}+\|u\|_{L^2(\mathbb{R}^n)}^2\right)\qquad\forall u\in H_{X,0}^{1}(\Omega),
\end{equation}
where $\|u\|_{H^{s}(\mathbb{R}^n)}=\left(\int_{\mathbb{R}^n}(1+|\xi|^{2})^{s}|\hat{u}(\xi)|^{2}d\xi\right)^{\frac{1}{2}}$ is the classical fractional Sobolev norm.

\end{proposition}
\begin{proof}
 See \cite[Theorem 17]{Stein1976}.
\end{proof}
\begin{proposition}[Sub-elliptic estimates II]
\label{subelliptic-estimate-2}
Let $X=(X_{1},X_{2},\ldots,X_{m})$ be the smooth vector fields defined on $\mathbb{R}^n$, satisfying
H\"{o}rmander's condition with H\"{o}rmander index $r$. Assume that $\eta,~\eta_1\in C_{0}^{\infty}(\mathbb{R}^n)$ are some functions such that $\eta_1\equiv 1$ on the support of $\eta$. Then for every $s\geq 0$, there is a constant $C>0$ such that
\begin{equation}\label{2-17}
\|\eta u\|_{H^{s+\frac{2}{r}}(\mathbb{R}^n)}\leq C\left(\|\eta_1\triangle_{X}u\|_{H^s(\mathbb{R}^n)}+\|\eta_1 u  \|_{L^2(\mathbb{R}^n)}\right)\qquad \forall u\in L^2_{\rm loc}(\mathbb{R}^n),
\end{equation}
where $\|u\|_{H^{s}(\mathbb{R}^n)}=\left(\int_{\mathbb{R}^n}(1+|\xi|^{2})^{s}|\hat{u}(\xi)|^{2}d\xi\right)^{\frac{1}{2}}$ is the classical fractional Sobolev norm.
\end{proposition}
\begin{proof}
 See \cite[Theorem 18]{Stein1976}.
\end{proof}
 Proposition \ref{subelliptic-estimate-2} implies the following corollary.
\begin{corollary}
\label{corollary2-4}
Let $N\in \mathbb{N}^{+}$ and  $\eta(x)\in C_{0}^{\infty}(\mathbb{R}^n)$. If $(\triangle_{X})^{k}u\in L^2(\mathbb{R}^n)$ for $0\leq k\leq N$, then we have
\begin{equation}
 \|\eta u\|_{H^{\frac{2N}{r}}(\mathbb{R}^n)}\leq C\sum_{k=0}^{N}\|(\triangle_{X})^{k}u\|_{L^2(\mathbb{R}^n)}.
\end{equation}
\end{corollary}

Using subelliptic estimates, we can obtain the following Friedrichs-Poincar\'{e} type inequality for $H_{X,0}^{1}(\Omega)$.

 \begin{proposition}[Friedrichs-Poincar\'{e} Inequality]
\label{prop2-16}
Let $X=(X_{1},X_{2},\ldots,X_{m})$ be the smooth vector fields defined on $\mathbb{R}^n$, satisfying
H\"{o}rmander's condition with H\"{o}rmander index $r$. For any open bounded subset $\Omega\subset \mathbb{R}^n$, there exists a positive constant $C>0$ such that
\begin{equation}\label{2-21}
 \int_{\Omega}{|u|^2dx}\leq C\int_{\Omega}|Xu|^2dx,\qquad \forall u\in H_{X,0}^{1}(\Omega).
  \end{equation}
\end{proposition}
It should be noted that \eqref{2-21} is entirely different from the Poincar\'{e}-Wirtinger type inequality extensively investigated in \cite{Jerison1986duke,Hajlasz2000,Garofalo1996}. The original statement of \eqref{2-21} in \cite[Lemma 5]{Xu1990} and \cite[Lemma 3.2]{Jost1998} assumes both smoothness and non-characteristic conditions on the boundary $\partial\Omega$. However, these assumptions might be too restrictive in certain cases. To address this issue, we provide a more general version of \eqref{2-21} that does not impose any additional conditions on $\partial\Omega$.
\begin{proof}[Proof of Proposition \ref{prop2-16}]
We prove \eqref{2-21} by contradiction. Suppose that
\[ \inf_{\|\varphi\|_{L^{2}(\Omega)}=1,~~ \varphi\in H_{X,0}^{1}(\Omega)}\|X\varphi\|^2_{L^2(\Omega)}=0.\]
 Then there exists a sequence
$\{\varphi_{j}\}_{j=1}^{\infty}$ in  $H_{X,0}^{1}(\Omega)$ such that
$\|X\varphi_{j}\|_{L^2(\Omega)}\to 0$ with
$\|\varphi_{j}\|_{L^2(\Omega)}=1$.  Proposition \ref{subelliptic-estimate-1} shows that
$H_{X,0}^{1}(\Omega)$ is continuously embedded into $H_{0}^{\frac{1}{r}}(\Omega)$, where $H_{0}^{\frac{1}{r}}(\Omega)$ denotes the closure of $C_{0}^{\infty}(\Omega)$ in the classical fractional Sobolev space $H^{\frac{1}{r}}(\mathbb{R}^n)$. Since for any bounded open set $\Omega$, $H_{0}^{\frac{1}{r}}(\Omega)$ is compactly embedded into $L^2(\Omega)$ (e.g., see \cite[Corollary 2.8]{Brasco2014}), we  conclude  that
 $H_{X,0}^{1}(\Omega)$ is
compactly embedded into $L^2(\Omega)$.  Using \cite[Chapter I, Theorem 1.2]{Struwe2000}, there exists
$\varphi_{0}\in H_{X,0}^{1}(\Omega)$ with
$\|\varphi_{0}\|_{L^2(\Omega)}=1$, $\triangle_{X}\varphi_{0}=0$ and $\|X\varphi_{0}\|_{L^2(\Omega)}=
0$. The hypoellipticity of  $\triangle_{X}$ yields that $\varphi_{0}\in C^{\infty}(\Omega)$. Moreover, since $X_{j}\varphi_0=0$ on $\Omega$ for $1\leq j\leq m$ and $\|\varphi_{0}\|_{L^2(\Omega)}=1$, we can deduce from  H\"{o}rmander's condition that $\partial_{x_{j}}\varphi_0=0$ on $\Omega$ for $1\leq j\leq n$. This implies  $\varphi_0$ must be a non-zero constant on $\Omega$.

Next, we choose the sequence $\{u_k\}_{k=1}^{\infty}\subset C_{0}^{\infty}(\Omega)$ such that $u_{k}\to \varphi_{0}$ in $H_{X}^{1}(\mathbb{R}^n)$, and we denote by \[ \overline{u_{k}}:=\left\{
                     \begin{array}{ll}
                      u_{k}, & \hbox{on $\Omega$,} \\
                       0, & \hbox{on $\mathbb{R}^n\setminus \Omega$;}
                     \end{array}
                   \right.\qquad \mbox{and}\qquad\overline{\varphi_{0}}:=\left\{
                     \begin{array}{ll}
                       \varphi_{0}, & \hbox{on $\Omega$,} \\
                       0, & \hbox{on $\mathbb{R}^n\setminus \Omega$.}
                     \end{array}
                   \right.\]
It follows that $\{\overline{u_{k}}\}_{k=1}^{\infty}\subset C_{0}^{\infty}(\mathbb{R}^n)$ is a Cauchy sequence in $H_{X,0}^{1}(\mathbb{R}^n)$ with $\overline{u_{k}}\to \overline{\varphi_{0}}$ in $L^2(\mathbb{R}^n)$. As a result, we have
$\overline{\varphi_{0}}\in H_{X,0}^{1}(\mathbb{R}^n)$,
$ \|\overline{u_{k}}-\overline{\varphi_{0}}\|_{H_{X}^{1}(\mathbb{R}^n)}=\|u_{k}-\varphi_{0}\|_{H_{X}^{1}(\mathbb{R}^n)}\to 0$ and
\[ \int_{\mathbb{R}^n}|X\overline{\varphi_{0}}|^{2}dx=\int_{\Omega}|X\varphi_{0}|^{2}dx=0.\]
For any $v\in H_{X,0}^{1}(\mathbb{R}^n)$, there exists a sequence $\{v_k\}_{k=1}^{\infty}\subset  C_{0}^{\infty}(\mathbb{R}^n)$ such that $v_{k}\to v$ in $H_{X,0}^{1}(\mathbb{R}^n)$.

Since $Xv_{k}$ is bounded in $L^2(\mathbb{R}^n)$ and $\lim_{k\to \infty}\|X\overline{u_{k}}\|_{L^2(\mathbb{R}^n)}=\|X\overline{\varphi_{0}}
\|_{L^2(\mathbb{R}^n)}=0$. Then
\begin{equation*}
\begin{aligned}
\left|\int_{\mathbb{R}^n}X\overline{\varphi_{0}}\cdot Xvdx\right|&=\lim_{k\to \infty}\left|\int_{\mathbb{R}^n}X\overline{u_{k}}\cdot Xv_k dx\right|\leq \lim_{k\to \infty}\|X\overline{u_{k}}\|_{L^2(\mathbb{R}^n)}\|Xv_{k}\|_{L^2(\mathbb{R}^n)}=0.
\end{aligned}
\end{equation*}
Therefore, $\overline{\varphi_{0}}\in H_{X,0}^{1}(\mathbb{R}^n)$ and satisfies $\triangle_{X}\overline{\varphi_{0}}=0$. The hypoellipticity of  $\triangle_{X}$ indicates that $\overline{\varphi_{0}}\in C^{\infty}(\mathbb{R}^n)$, which leads a contradiction since $\overline{\varphi_{0}}$ is not smooth across $\partial\Omega$.
\end{proof}

\section{Estimates of the subelliptic heat kernels}
\label{Section3}
We mention that in \cite{Biagi2019,Biagi2020}, the authors treated the global fundamental solution as
a subelliptic global heat kernel. Nevertheless, the properties of subelliptic global heat kernel stated in \cite{Biagi2019,Biagi2020} are incomplete, as its lack of $L^2$ framework prevents us from comparing the different subelliptic heat kernels in weak sense. In order to make the discussion reasonably self-contained, we will reconstruct the subelliptic global heat kernel of homogeneous H\"{o}rmander operators through the general $L^2$ theory of the heat semigroup and subelliptic estimates, and prove its  equivalence to the global fundamental solution in \cite{Biagi2019,Biagi2020}.

We first introduce the framework of Dirichlet forms and heat semigroups, as well as heat kernels. For more details, one can refer to \cite{Fukushima1994,Grigoryan2008,Grigoryan2014,Grigoryan2010-jfa,Davies1990}.

\subsection{Abstract theory of Dirichlet forms, heat semigroups and heat kernels}

Let  $(M,d,\mu)$ be a metric measure space, where
 $(M,d)$ is a locally compact separable metric space, and $\mu$ is a Radon measure on $M$ such that $\mu(\Omega)>0$ for any non-empty open subset $\Omega\subset M$.  To simplify notation, we denote by $L^2(M):=L^2(M,\mu)$ the set of all measurable functions $f: M\to \mathbb{R}$  such that
$\int_{M}|f(x)|^2d\mu(x)<\infty$.
\begin{definition}[Dirichlet form]
\label{def3-1}
A pair $(\mathcal{E},\mathcal{F})$ is said to be a Dirichlet form in $L^2(M)$ if
\begin{enumerate}[(1)]
  \item $\mathcal{F}$ is a dense subspace of $L^2(M)$.
  \item $\mathcal{E}(\cdot,\cdot)$ is a bilinear, symmetric, non-negative definite, closed functional on $\mathcal{F}\times \mathcal{F}$, where the closedness means that $\mathcal{F}$ is a Hilbert space equipped with the norm $(\|f\|_{L^2(M)}^2+\mathcal{E}(f,f))^{\frac{1}{2}}$.
  \item $\mathcal{E}(\cdot,\cdot)$ is Markovian, i.e.  $f\in \mathcal{F}$ implies $\tilde{f}:=\max\{\min\{f,1\},0\}\in \mathcal{F}$ and $\mathcal{E}(\tilde{f},\tilde{f})\leq \mathcal{E}(f,f)$.
\end{enumerate}
\end{definition}

\begin{definition}[Regular Dirichlet form]
\label{def3-2}
The Dirichlet form $(\mathcal{E},\mathcal{F})$ is said to be regular if the space $\mathcal{F}\cap C_{0}(M)$ is dense both in $\mathcal{F}$ and in $C_{0}(M)$, where $C_{0}(M)$ denotes the space of all real-valued continuous functions with compact support in $M$, endowed with sup-norm.
\end{definition}

\begin{definition}[Local Dirichlet form]
\label{def3-3}
The Dirichlet form $(\mathcal{E},\mathcal{F})$ is said to be local if $\mathcal{E}(f,g)=0$ for any $f,g\in\mathcal{F}$ with disjoint compact supports in $M$.
\end{definition}

For the Dirichlet form $(\mathcal{E},\mathcal{F})$, it is well-known (e.g., see \cite[Theorem 1.3.1, Corollary 1.3.1]{Fukushima1994}) that there exists a unique generator $\mathcal{L}$, which is a non-negative definite self-adjoint operator in $L^2(M)$ with the domain $\text{dom}(\mathcal{L})\subset \mathcal{F}$ such that
\[ \mathcal{E}(f,g)=(\mathcal{L}f,g)_{L^2(M)}\quad \forall f\in \text{dom}(\mathcal{L}),~ g\in \mathcal{F}. \]

According to the spectral theorem (see \cite[Appendix A.5.4]{Grigor'yan2009}), the self-adjoint operator $\mathcal{L}$ admits a unique spectral resolution $\{E_{\lambda}\}$ in $L^2(M)$  such that
\begin{equation}\label{3-1}
\mathcal{L}=\int_{0}^{+\infty}\lambda dE_{\lambda},
\end{equation}
and we can rewrite the domain $\text{dom}(\mathcal{L})$ of $\mathcal{L}$ as
\[  \text{dom}(\mathcal{L})=\left\{f\in L^2(M)\Bigg|\int_{0}^{+\infty}\lambda^{2} d\|E_{\lambda}f\|_{L^2(M)}^2<+\infty\right\}. \]
Moreover, the
corresponding heat semigroup $\{P_{t}\}_{t\geq 0}$ is given by
\[ P_{t}:=e^{-t\mathcal{L}}=\int_{0}^{+\infty}e^{-t\lambda}dE_{\lambda},\]
which satisfies the following properties (see \cite[Theorem 4.9]{Grigor'yan2009} and \cite[Section 2, p. 509]{Grigoryan2014}):
\begin{enumerate}[(S1)]
  \item $\|P_{t}f\|_{L^2(M)}\leq \|f\|_{L^2(M)}$ for all $f\in L^2(M)$.
  \item $P_{t}P_{s}=P_{t+s}$ for all $t,s\geq 0$.
  \item  $\lim_{t\to 0}\|P_{t}f-f\|_{L^2(M)}=0 $ for all $f\in L^2(M)$.
  \item For all $f,g\in L^2(M)$, $(P_{t}f,g)_{L^2(M)}=(f,P_{t}g)_{L^2(M)}$.
  \item For any $t>0$, $f\geq 0$  a.e. on $M$ implies $P_{t}f\geq 0$ a.e. on $M$. In addition, $f\leq 1$ a.e. on $M$ implies $P_{t}f\leq 1$ a.e. on $M$.
\item If $f\in L^2(M)$ and $t>0$, then $P_{t}f\in \text{dom}(\mathcal{L})$ and
\begin{equation}\label{3-2}
  \frac{d}{dt}(P_{t}f)=-\mathcal{L}(P_{t}f),
\end{equation}
where $\frac{d}{dt}$ denotes the strong derivative in $L^2(M)$.
\end{enumerate}

\begin{definition}[Heat kernel, see \cite{Grigoryan2008}]
\label{def3-4}
A family of measurable functions $\{H_t(x,y)\}_{t>0}$ defined on $M\times M$ is called the heat kernel of $(\mathcal{E},\mathcal{F})$ if, for all $f\in L^2(M)$, $t>0$, and almost all $x\in M$,
\begin{equation}
  P_{t}f(x)=\int_{M}H_t(x,y)f(y)d\mu(y).
\end{equation}
\end{definition}

We mention that the heat kernel does not have to exist, but if it exists, then it is unique (up
to a set of measure zero) and automatically satisfies the following properties, which follow
 from the corresponding properties of the heat semigroup (see \cite{Grigoryan2008,Grigoryan2014}):

\begin{enumerate}[(P1)]
  \item  For all $t>0$ and almost all $x,y\in M$, $H_t(x,y)\geq 0$ and
$\int_{M} H_t(x,y)d\mu(y)\leq 1$.
\item For all $t,s>0$ and almost all $x,y\in M$, $H_t(x,y)=H_t(y,x)$.
\item For all $t>0$ and almost all $x,y\in M$,
$ H_{t+s}(x,y)=\int_{M} H_t(x,z)H_s(z,y)d\mu(y)$.
 \item For any $f\in L^2(M)$,
 $\lim_{t\to0^+}\left\|\int_{M}H_t(\cdot,y)f(y)d\mu(y)-f(\cdot)\right\|_{L^2(M)}=0$.
\end{enumerate}

We next introduce the concept of restricted Dirichlet forms. Let $\Omega$ be a non-empty open subset of $M$. We identify the space $L^2(\Omega)$ as a subspaces of $L^2(M)$ by extending every function $f\in L^2(\Omega)$ to $M$ by setting $f=0$ outside $\Omega$. Assuming that  $(\mathcal{E},\mathcal{F})$ is a regular Dirichlet form, we define $\mathcal{F}(\Omega)$ as follows:
\[\mathcal{F}(\Omega):=\overline{C_0(\Omega)\cap \mathcal{F}}^{\|\cdot\|_{\mathcal{F}}},\]
where $C_0(\Omega)$ denotes the set consisting of all continuous functions with compact support contained in $\Omega$. In particular, $\mathcal{F}(M)=\mathcal{F}$ if $\Omega=M$. The pair $(\mathcal{E},\mathcal{F}(\Omega))$ is usually referred to the restricted Dirichlet form. Moreover, according to \cite[Appendix, p. 112]{Grigoryan2008}, we know that
$(\mathcal{E},\mathcal{F}(\Omega))$ is also a regular Dirichlet form.

Let $(\mathcal{E},\mathcal{F}(\Omega_{1}))$ and $(\mathcal{E},\mathcal{F}(\Omega_{2}))$ be two restricted Dirichlet forms, where $\Omega_{1}$ and $\Omega_{2}$ are two open subset of $M$ with $\Omega_{1}\subset \Omega_{2}$. Then, it follows from \cite{Grigoryan2010-jfa} that for any $0\leq f\in L^2(M)$ and $t>0$,
\begin{equation}\label{3-4}
P_{t}^{\Omega_{1}}f\leq P_{t}^{\Omega_{2}}f\quad \text{a.e. on}~~  M,
\end{equation}
where $\{P_{t}^{\Omega_{i}}\}_{t\geq 0}$ denotes the heat semigroup of $(\mathcal{E},\mathcal{F}(\Omega_{i}))$ for $i=1,2$.

\subsection{Subelliptic global heat kernel}

We investigated the existence of subelliptic global heat kernel using the abstract theory mentioned above.

Let $X=(X_1,X_{2},\ldots,X_m)$ be the homogeneous H\"{o}rmander vector fields defined on $\mathbb{R}^n$.
We consider the bilinear form $\mathcal{Q}(\cdot,\cdot): H_X^1(\mathbb{R}^n)\times H_X^1(\mathbb{R}^n)\to \mathbb{R}$ such that
\begin{equation}\label{3-5}
\mathcal{Q}(u,v)=\int_{\mathbb{R}^n} Xu\cdot Xv dx.
\end{equation}
Then we have
\begin{proposition}
\label{prop3-1}
Assuming that $X=(X_1,X_{2},\ldots,X_m)$ are the homogeneous H\"{o}rmander vector fields defined on $\mathbb{R}^n$ and $\Omega\subset \mathbb{R}^n$ is an open bounded domain, we have the following:
  \begin{enumerate}[(1)]
    \item The pair $(\mathcal{Q}, H_{X}^1(\mathbb{R}^n))$ is a regular and local Dirichlet form.
    \item The pair $(\mathcal{Q}, H_{X,0}^1(\Omega))$ is a restricted Dirichlet form, which is also
 regular and local.
  \end{enumerate}
\end{proposition}

\begin{proof}
We first show that $(\mathcal{Q}, H_{X}^1(\mathbb{R}^n))$ is a Dirichlet form. Clearly, $H_ {X}^1(\mathbb{R}^n)$ is a dense subspace in $L^2(\mathbb{R}^n)$. By \eqref{3-5} we see that $\mathcal{Q}$ is a bilinear, symmetric, non-negative definite, closed functional on $H_{X}^1(\mathbb{R}^n)\times H_{X}^1(\mathbb{R}^n)$. Thus, we only need to verify the Markovian property.

For any $u\in H_{X}^1(\mathbb{R}^n)$, we set $\tilde{u}:=\max\{\min\{u,1\},0\}$. Observing that $\tilde{u}=u_{+}-(u-1)_{+}$, it follows from Corollary \ref{corollary2-3} that $\tilde{u}\in H_{X}^{1}(\mathbb{R}^n)$ and satisfies
\[ \mathcal{Q}(\tilde{u},\tilde{u})=\int_{\mathbb{R}^n}|X\tilde{u}|^2dx=\int_{\{x\in\mathbb{R}^n|0\leq u(x)\leq 1\}}|Xu|^2dx \leq\int_{\mathbb{R}^n}|Xu|^2dx= \mathcal{Q}(u,u).\]
Thus,  $(\mathcal{Q}, H_{X}^1(\mathbb{R}^n))$ is a Dirichlet form.

 The local property of $(\mathcal{Q}, H_{X}^1(\mathbb{R}^n))$ follows from \eqref{3-5}. By Proposition \ref{prop2-11}, we deduce that  $H_ {X}^1(\mathbb{R}^n)\cap C_0(\mathbb{R}^n)$ is dense in $H_{X}^1(\mathbb{R}^n)$. Moreover, due to the
 Stone-Weierstrass theorem, we know that $C_{0}^{\infty}(\mathbb{R}^n)$ is dense in $C_{0}(\mathbb{R}^n)$, which implies $H_ {X}^1(\mathbb{R}^n)\cap C_0(\mathbb{R}^n)$ is dense in $C_{0}(\mathbb{R}^n)$. Consequently, $(\mathcal{Q}, H_{X}^1(\mathbb{R}^n))$ is a regular and local Dirichlet form in $L^2(\mathbb{R}^n)$, and the conclusion (1) is proved.

The proof of local property for $(\mathcal{Q}, H_{X,0}^1(\Omega))$ is similar to (1). Using Proposition \ref{prop2-13} we have
\begin{equation*}
H_{X,0}^1(\Omega)=\overline{H_{X}^1(\mathbb{R}^n)\cap C_0(\Omega)}^{\|\cdot\|_{H_{X}^1(\mathbb{R}^n)}},
\end{equation*}
which implies that $(\mathcal{Q}, H_{X,0}^1(\Omega))$ is a restricted Dirichlet form in $L^2(\Omega)$. In addition, $(\mathcal{Q}, H_{X,0}^1(\Omega))$ is regular since $(\mathcal{Q}, H_{X}^1(\mathbb{R}^n))$ is regular.
\end{proof}

 Proposition \ref{prop3-1} (1) indicates that the Dirichlet form
 $(\mathcal{Q}, H_{X}^{1}(\mathbb{R}^n))$ has a unique generator, denoted by $\mathcal{L}$, which is a non-negative defined self-adjoint operator in $L^2(\mathbb{R}^n)$ with  $\text{dom}(\mathcal{L})\subset H_{X}^{1}(\mathbb{R}^n)$ such that
\begin{equation}\label{3-6}
   \mathcal{Q}(u,v)=(\mathcal{L}u,v)_{L^2(\mathbb{R}^n)}
\end{equation}
for all $u\in \text{dom}(\mathcal{L})$ and $v\in H_{X}^{1}(\mathbb{R}^n)$. Furthermore, the operator $\mathcal{L}$ is the unique self-adjoint extension of $-\triangle_{X}|_{\mathcal{D}(\mathbb{R}^n)}$. The domain $\text{dom}(\mathcal{L})$ can be characterized by
\[ \begin{aligned}
\text{dom}(\mathcal{L})&=\{u\in H_{X}^{1}(\mathbb{R}^n)|\exists c\geq 0~\mbox{such that}~|\mathcal{Q}(u,v)|\leq c\|v\|_{L^{2}(\mathbb{R}^n)}~\forall v\in H_{X}^{1}(\mathbb{R}^n)\}\\
&=\{u\in H_{X}^{1}(\mathbb{R}^n)|\triangle_{X} u\in L^2(\mathbb{R}^n)\}.
\end{aligned}
 \]

Let $\{P_{t}\}_{t\geq 0}$ be the corresponding heat semigroup of $\mathcal{L}$. Then we have
\begin{proposition}
\label{prop3-2}
For any $f\in L^2(\mathbb{R}^n)$ and $t>0$, the function $P_{t}f\in C^{\infty}(\mathbb{R}^n)\cap L^2(\mathbb{R}^n)$. Moreover, for any compact set $K\subset \mathbb{R}^n$ and any non-negative integer $m\geq 0$, we have
\begin{equation}\label{3-7}
\|P_{t}f\|_{C^{m}(K)}\leq  F_{K}(t)\|f\|_{L^2(\mathbb{R}^n)},
\end{equation}
where
\begin{equation}\label{3-8}
  F_{K}(t)=C(1+t^{-N}),
\end{equation}
$N$ is the smallest integer such that $N>\frac{n\alpha_n}{4}+\frac{m\alpha_n}{2}$, and $C$ is a positive constant.
\end{proposition}
\begin{proof}
Let $\{E_{\lambda}\}_{0\leq \lambda<+\infty}$ be the spectral resolution of $\mathcal{L}$ in $L^2(\mathbb{R}^n)$. The heat semigroup $\{P_{t}\}_{t\geq 0}$ is given by
\[ P_{t}f=e^{-t\mathcal{L}}f=\int_{0}^{+\infty}e^{-t\lambda}dE_{\lambda}f\qquad \forall f\in L^2(\mathbb{R}^n).\]
 Consider the function $\Phi(\lambda)=\lambda^k e^{-t\lambda}$, where $t>0$ and $k$ is a positive integer. It follows from spectral theory  that
\begin{equation}\label{3-9}
\mathcal{L}^{k}e^{-t\mathcal{L}}=\Phi(\mathcal{L})=\int_{0}^{+\infty}\lambda^{k}e^{-t\lambda}dE_{\lambda}.
\end{equation}
Since the function   $\Phi(\lambda)$ is bounded on $[0,+\infty)$, the operator $\Phi(\mathcal{L})=\mathcal{L}^{k}e^{-t\mathcal{L}}$ is bounded in $L^2(\mathbb{R}^n)$. Hence, for any $f\in L^2(\mathbb{R}^n)$ and any $k\geq 0$, we have
$ \mathcal{L}^{k} (e^{-t\mathcal{L}}f)\in L^2(\mathbb{R}^n)$,
that is
\begin{equation}\label{3-10}
(-\triangle_{X})^k(P_{t}f)\in L^2(\mathbb{R}^n).
\end{equation}
Proposition \ref{prop2-5}, Corollary \ref{corollary2-4} and \eqref{3-10} give that $P_{t}f\in H_{\rm loc}^{\frac{2k}{\alpha_{n}}}(\mathbb{R}^n)$ for any $k\geq 0$, which implies that $P_{t}f\in C^{\infty}(\mathbb{R}^n)\cap L^{2}(\mathbb{R}^n)$.

Note that the function $\Phi(\lambda)$ takes its maximal value at
$\lambda=\frac{k}{t}$. This implies for any $f\in L^{2}(\mathbb{R}^n)$,
\begin{equation*}
\begin{aligned}
\|\triangle_{X}^{k}P_{t}f\|_{L^2(\mathbb{R}^n)}&=\|\mathcal{L}^{k} (e^{-t\mathcal{L}}f)\|_{L^2(\mathbb{R}^n)}=\left(\int_{0}^{+\infty}(\lambda^{k}e^{-t\lambda})^2 d\|E_{\lambda}f\|^2_{L^2(\mathbb{R}^n)}\right)^{\frac{1}{2}}\\
&\leq 	\sup_{\lambda\geq 0}\big(\lambda^{k}e^{-t\lambda}\big)\left(\int_{0}^{\infty}d\|E_{\lambda}f\|^2_{L^2(\mathbb{R}^n)}
\right)^{\frac{1}{2}}= \left(\frac{k}{t}\right)^{k}e^{-k}\|f\|_{L^2(\mathbb{R}^n)}.
\end{aligned}
\end{equation*}
Then for any positive integer $N$, we have
\begin{equation}\label{3-11}
\begin{aligned}
\sum_{k=0}^{N}\|\triangle_{X}^{k}P_{t}f\|_{L^2(\mathbb{R}^n)}&\leq \left(1+\sum_{k=1}^{N}\left(\frac{k}{t}\right)^{k}e^{-k}  \right)\|f\|_{L^2(\mathbb{R}^n)}\leq C(1+t^{-N})\|f\|_{L^2(\mathbb{R}^n)},
 \end{aligned}
\end{equation}
where $C$ is a constant depending on $N$.

For any compact subset $K\subset \mathbb{R}^n$, we can choose a cut-off function $\eta\in C_{0}^{\infty}(\mathbb{R}^n)$ such that $\eta\equiv 1$ on $K$. Using Proposition \ref{prop2-5}, Corollary \ref{corollary2-4}, \eqref{3-11} and the classical Sobolev embedding $H^{s}(\mathbb{R}^n)\hookrightarrow C^{m}(K)$ for $s>\frac{2}{n}+m$,
 we obtain for $N>\frac{n\alpha_n}{4}+\frac{m\alpha_n}{2}$ and any $f\in L^2(\mathbb{R}^n)$,
\[ \begin{aligned}
\|P_{t}f\|_{C^{m}(K)}&=\|\eta P_{t}f\|_{C^{m}(K)}\leq C_{1}\|\eta P_{t}f \|_{H^{\frac{2N}{\alpha_n}}(\mathbb{R}^n)}\leq C_{2}\sum_{k=0}^{N}\|(\triangle_{X})^{k}P_{t}f \|_{L^2(\mathbb{R}^n)}\\
&\leq C_{3}(1+t^{-N})\|f \|_{L^2(\mathbb{R}^n)}.
\end{aligned}\]
Here, $C_{1},C_{2}$ and $C_{3}$ are some positive constants.
\end{proof}

We provide the following serval lemmas to establish the heat kernel of $\{P_t\}_{t\geq 0}$.
\begin{lemma}
\label{lemma3-1}
For any $x\in \mathbb{R}^n$ and  $t>0$, there exists a unique
function $p_{t,x}\in L^2(\mathbb{R}^n)$ such that
\begin{equation}\label{3-12}
  P_{t}f(x)=\int_{\mathbb{R}^n}p_{t,x}(y)f(y)dy\qquad \forall f\in L^2(\mathbb{R}^n).
\end{equation}
\end{lemma}
\begin{proof}
Assume that $K\subset \mathbb{R}^n$ is a compact subset.
By Proposition \ref{prop3-2}, for all $t>0$ and $f\in L^2(\mathbb{R}^n)$,  we have $P_{t}f\in C^{\infty}(\mathbb{R}^n)$ and
\begin{equation}\label{3-13}
  |P_{t}f(x)|\leq F_{K}(t)\|f\|_{L^2(\mathbb{R}^n)}\qquad \forall x\in K.
\end{equation}
This means, for fixed $t>0$ and $x\in K$, the mapping $f\mapsto P_{t}f(x)$ is a bounded linear functional on $L^2(\mathbb{R}^n)$. By the Riesz representation theorem, there exists a unique function $p_{t,x}\in L^2(\mathbb{R}^n)$ such that
$ P_{t}f(x)=(p_{t,x},f)_{L^2(\mathbb{R}^n)}$ for all $f\in L^2(\mathbb{R}^n)$,
whence \eqref{3-12} follows. Since
for any point $x\in \mathbb{R}^n$ there is a compact set $K$ containing $x$, the function $p_{t,x}$ is defined for all $t > 0$ and $x\in \mathbb{R}^n$.
\end{proof}

\begin{lemma}
\label{lemma3-2}
For any $f\in L^2(\mathbb{R}^n)$, the function $u(x,t)=P_{t}f(x)\in C^{\infty}(\mathbb{R}^n\times\mathbb{R}^{+})$ and satisfies the degenerate heat equation $\partial_{t}u=\triangle_{X}u$ in $\mathbb{R}^n\times\mathbb{R}^{+}$.
\end{lemma}
\begin{proof}
For any $t>0$ and any non-negative
integer $m$, by \eqref{3-9} we have for small $|h|>0$,
\begin{equation*}
  \mathcal{L}^{m}(P_{t+h}f)=\int_{0}^{\infty}\lambda^m e^{-(t+h)\lambda}dE_{\lambda}f.
\end{equation*}
Note that the function $\lambda\mapsto \lambda^{m}e^{-(t+h)\lambda}$ is uniformly bounded as $h\to 0$.  \cite[Lemma 4.8]{Grigor'yan2009} yields that $\mathcal{L}^{m}(P_{t+h}f)\to \mathcal{L}^{m}(P_{t}f)$ in $L^2(\mathbb{R}^n)$ as $h\to 0$. The proof of Proposition \ref{prop3-2} indicates that $u(t,\cdot)\in \text{dom}(\mathcal{L}^m)$ and $\triangle_{X}^{m}u=(-\mathcal{L})^{m}u$. Hence, $\triangle_{X}^{m}(P_{t+h}f-P_{t}f)\to 0$ in $L^2(\mathbb{R}^n)$ as $h\to 0$ for any non-negative
integer $m$, which implies $P_{t+h}f-P_{t}f\to 0$ in $H_{\rm loc}^{s}(\mathbb{R}^n)$ for any $s>0$ due to Corollary \ref{corollary2-4}. According to the classical Sobolev embedding results, we deduce that $u(\cdot,t+h)\to u(\cdot,t)$ in $C^{\infty}$ topology as $h\to 0$, and  $u(x,t)$ is continuous in $t$ locally
uniformly in $x$. Moreover, Proposition \ref{prop3-2} shows that $u(x,t)$ is $C^{\infty}$-smooth in $x$ for any fixed $t>0$, which derives that $u (x,t)\in C(\mathbb{R}^n\times \mathbb{R}^{+})$.

Let $\widetilde{M}:=\mathbb{R}^n\times \mathbb{R}^{+}$ be the product manifold with the measure $d\mu=dxdt$. The continuity of $u(x,t)$ jointly in $(x,t)$ allows us to consider $u$ as a distribution on $\widetilde{M}$. For any  $\varphi\in \mathcal{D}(\widetilde{M})$, we obtain from Fubini's theorem that
\begin{equation}
\begin{aligned}\label{3-14}
\left(u,\frac{\partial \varphi}{\partial t}+\triangle_{X}\varphi\right)_{L^2(\widetilde{M})}&=\int_{\widetilde{M}}u\left(\frac{\partial \varphi}{\partial t}+\triangle_{X}\varphi\right)d\mu\\
&=\int_{0}^{+\infty}(u,\partial_{t}\varphi)_{L^2(\mathbb{R}^n)}dt+\int_{0}^{+\infty}(u,\triangle_{X}\varphi)_{L^2(\mathbb{R}^n)}dt.
\end{aligned}
\end{equation}
Considering $\varphi(\cdot,t)$ as a path in $L^2(\mathbb{R}^n)$, it follows that the classical partial derivative $\partial_{t}\varphi$
 coincides with the strong derivative $\frac{d\varphi}{dt}$. Then
\begin{equation}\label{3-15}
\left(u,\partial_{t}\varphi\right)_{L^2(\mathbb{R}^n)}=\left(u,\frac{d\varphi}{dt}\right)_{L^2(\mathbb{R}^n)}=\frac{d}{dt}(u,\varphi)_{L^2(\mathbb{R}^n)}-\left( \frac{du}{dt},\varphi\right)_{L^2(\mathbb{R}^n)}.
\end{equation}
Since $\varphi(\cdot,t)\in \mathcal{D}(\widetilde{M})$ vanishes outside some time interval $[a,b]$ where $0<a<b$, we obtain
\begin{equation}\label{3-16}
  \int_{0}^{+\infty}\frac{d}{dt}(u,\varphi)_{L^2(\mathbb{R}^n)}dt=0.
\end{equation}
Besides, by \eqref{3-2} we have $\frac{du}{dt}=\triangle_{X}u$, which yields that
\begin{equation}\label{3-17}
  \int_{0}^{+\infty}\left(\frac{du}{dt},\varphi\right)_{L^2(\mathbb{R}^n)}dt= \int_{0}^{+\infty}\left(\triangle_{X}u,\varphi\right)_{L^2(\mathbb{R}^n)}dt=\int_{0}^{+\infty}\left(u,\triangle_{X}\varphi\right)_{L^2(\mathbb{R}^n)}dt.
\end{equation}
Combining \eqref{3-14}-\eqref{3-17}, we obtain
$ \left(u,\frac{\partial \varphi}{\partial t}+\triangle_{X}\varphi\right)_{L^2(\widetilde{M})}=0$ for all $\varphi\in \mathcal{D}(\widetilde{M})$. This means $\partial_{t}u-\triangle_{X}u=0$ in $\mathcal{D}'(\widetilde{M})$. The hypoellipticity of $\partial_{t}-\triangle_{X}$ derives that $u(x,t)\in C^{\infty}(\mathbb{R}^n\times\mathbb{R}^{+})$ and satisfies  $\partial_{t}u=\triangle_{X}u$ in $\mathbb{R}^n\times\mathbb{R}^{+}$.
\end{proof}

\begin{proposition}
\label{prop3-3}
The heat semigroup $\{P_{t}\}_{t\geq 0}$ admits a unique heat kernel $h(x,y,t)$ such that for any $f\in L^2(\mathbb{R}^n)$,
\begin{equation}\label{3-18}
  P_{t}f(x)=e^{-t\mathcal{L}}f(x)=\int_{\mathbb{R}^n}h(x,y,t)f(y)dy
\end{equation}
holds for  all $x\in \mathbb{R}^n$ and $t>0$. The heat kernel $h(x,y,t)\in C^{\infty}(\mathbb{R}^n\times\mathbb{R}^n\times\mathbb{R}^{+})$ and satisfies the above properties (P1)-(P4) pointwise. Furthermore, we have
\begin{itemize}
  \item For any $f_{0}\in C_{0}^{\infty}(\mathbb{R}^n)$,
 \begin{equation} \label{3-19}
\lim_{t\to 0^{+}}\int_{\mathbb{R}^n}h(x,y,t)f_{0}(y)dy=f_{0}(x).
\end{equation}
  \item For any fixed point $y\in \mathbb{R}^n$, $h(x,y,t)$ is the solution of
\begin{equation}
\label{3-20}
  \left(\frac{\partial}{\partial t}-\triangle_{X}\right)h(x,y,t)=0 \quad \mbox{for all}~~ (x,t)\in\mathbb{R}^n\times \mathbb{R}^{+}.
\end{equation}
  \item For all $x,y\in \mathbb{R}^{n}$ and $s<t$, $h(x,y,t-s)$ is identical to the fundamental solution of the degenerate heat operator $\partial_{t}-\triangle_{X}$ at $(s,y)\in \mathbb{R}^{1+n}$, and satisfies
  the global Gaussian bounds:
\begin{equation}
 \label{3-21}
    \frac{1}{A_{1}|B_{d_{X}}(x,\sqrt{t})|}e^{-\frac{A_{1} d_{X}^{2}(x,y)}{t}}\leq h(x,y,t)\leq
 \frac{A_{1}}{|B_{d_{X}}(x,\sqrt{t})|}e^{-\frac{ d_{X}^{2}(x,y)}{A_{1}t}}\qquad \forall x,y\in \mathbb{R}^n,~t>0,
 \end{equation}
where $A_{1}>1$ is a positive constant, and $d_{X}(x,y)$ is the subunit metric (see Definition \ref{def2-3}).
\end{itemize}
\end{proposition}
\begin{proof}
Using $P_{t+s}=P_{s}P_{t}$, \eqref{3-12} and the symmetry of $P_{t}$, we obtain for any $f\in L^2(\mathbb{R}^n)$,
\begin{equation}\label{3-22}
\begin{aligned}
 P_{t+s}f(x)&=P_{s}(P_{t}f)(x)=(p_{s,x},P_{t}f)_{L^2(\mathbb{R}^n)}=(P_{t}p_{s,x},f)_{L^2(\mathbb{R}^n)}\\
&=\int_{\mathbb{R}^n}P_{t}p_{s,x}(z)f(z)dz=\int_{\mathbb{R}^n}(p_{t,z},p_{s,x})_{L^2(\mathbb{R}^n)}f(z)dz.
\end{aligned}
\end{equation}
Then, for any $0<r<s<t$, applying \eqref{3-12} and \eqref{3-22} with $f=p_{r,x}$,
\begin{equation}\label{3-23}
\begin{aligned}
(p_{s,x},p_{t-s,y})_{L^2(\mathbb{R}^n)}&=P_{s}p_{t-s,y}(x)=P_{r}(P_{s-r}p_{t-s,y})(x)=\int_{\mathbb{R}^n}p_{r,x}(z)(p_{s-r,z},p_{t-s,y} )_{L^2(\mathbb{R}^n)}dz\\
&=P_{t-r}p_{r,x}(y)=(p_{t-r,y},p_{r,x})_{L^2(\mathbb{R}^n)}.
\end{aligned}
\end{equation}
Clearly, \eqref{3-23} asserts that for all $x,y\in \mathbb{R}^n$ and $t>0$, $(p_{s,x},p_{t-s,y})_{L^2(\mathbb{R}^n)}$ is independent of $s\in (0,t)$.

For any $x,y\in \mathbb{R}^n$ and $t>0$, we define
\begin{equation}\label{3-24}
h(x,y,t):=\left(p_{\frac{t}{2},x},p_{\frac{t}{2},y}\right)_{L^2(\mathbb{R}^n)}.
\end{equation}
This means $h(x,y,t)=h(y,x,t)$ for any $x,y\in \mathbb{R}^n$ and $t>0$, which gives (P2) pointwise.  It follows from \eqref{3-22} and \eqref{3-24} that
\[ P_{t}f(x)=\int_{\mathbb{R}^n}\left(p_{\frac{t}{2},z},p_{\frac{t}{2},x}\right)_{L^2(\mathbb{R}^n)}f(z)dz=\int_{\mathbb{R}^n}h(x,y,t)f(y)dy, \]
which yields \eqref{3-18}. Additionally, by \eqref{3-23} and \eqref{3-24} we have
\begin{equation}\label{3-25}
  h(x,y,t)=(p_{s,x},p_{t-s,y})_{L^2(\mathbb{R}^n)}\quad \forall 0<s<t.
\end{equation}
Comparison of \eqref{3-12} and \eqref{3-18} shows that for all $x\in \mathbb{R}^n, t>0$,
\begin{equation}\label{3-26}
  h(x,\cdot,t)=p_{t,x}(\cdot)\quad\mbox{a.e. on}~~\mathbb{R}^n.
\end{equation}
Using \eqref{3-25} and \eqref{3-26}, we obtain for all $x,y\in \mathbb{R}^n$ and $t,s>0$,
\begin{equation}\label{3-27}
  \int_{\mathbb{R}^n}h(x,z,t)h(z,y,s)dz=(h(x,\cdot,t),h(y,\cdot,s))_{L^2(\mathbb{R}^n)}=(p_{t,x},p_{s,y})_{L^2(\mathbb{R}^n)}=h(x,y,t+s),
\end{equation}
which derives (P3) pointwise.

We next show that $h(x,y,t)\in C^{\infty}(\mathbb{R}^n\times \mathbb{R}^n\times \mathbb{R}^+)$, which indicates that $h(\cdot,\cdot,t)$ is a measurable function on $\mathbb{R}^n\times \mathbb{R}^n$ for all $t>0$, and thus serves as the heat kernel of $P_{t}$. Consider the map $(x,t)\mapsto p_{t,x}(\cdot)$ from $\mathbb{R}^n\times \mathbb{R}^{+}$ to $L^2(\mathbb{R}^n)$. By Lemma \ref{lemma3-2}, for any $f\in L^2(\mathbb{R}^n)$, the function $P_{t}f(x)=(p_{t,x},f)_{L^2(\mathbb{R}^n)}$ is $C^{\infty}$-smooth in $x,t$, implying that the mapping $(x,t)\mapsto p_{t,x}(\cdot)$ is weakly $C^{\infty}$. Then, \cite[Lemma 7.21]{Grigor'yan2009} derives that the mapping $(x,t)\mapsto p_{t,x}(\cdot)$ is strongly $C^{\infty}$. Furthermore, it follows from \eqref{3-25} that $h(x,y,t+s)=(p_{t,x},p_{s,y})_{L^2(\mathbb{R}^n)}\in C^{\infty}(\mathbb{R}^n\times \mathbb{R}^n)$ for any $t,s>0$. Consequently, we conclude that $h(x,y,t)\in  C^{\infty}(\mathbb{R}^n\times \mathbb{R}^n\times \mathbb{R}^+)$.

The property (S5)  and \eqref{3-12} indicate that
\[
  0\leq P_{t}(p_{t,x})_{-}=(p_{t,x},(p_{t,x})_{-})_{L^{2}(\mathbb{R}^n)}=-((p_{t,x})_{-},(p_{t,x})_{-})_{L^{2}(\mathbb{R}^n)},\]
which yields that $p_{t,x}\geq 0$ a.e. on $\mathbb{R}^n$. Using \eqref{3-26}, we have $h(x,y,t)\geq 0$ on $\mathbb{R}^n\times \mathbb{R}^n\times \mathbb{R}^+$. Also, the property (S5) implies that $\int_{\mathbb{R}^n}h(x,y,t)dy\leq 1$ for all $x\in \mathbb{R}^n$ and $t>0$. Hence, the property (P1) holds pointwise. Moreover, the property (P4) follows from (S3) and \eqref{3-18}.

If $f_0\in C_{0}^{\infty}(\mathbb{R}^n)$, then $f_{0}\in \text{dom}(\mathcal{L}^{m})$ for any positive integer $m$, and we have
\[ \int_{0}^{+\infty}\lambda^{2m}d\|E_{\lambda}f\|_{L^2(\mathbb{R}^n)}^2<\infty. \]
The identities $\mathcal{L}^{m}f=\int_{0}^{+\infty}\lambda^{m}dE_{\lambda}f$ and $\mathcal{L}^{m}P_{t}f=\int_{0}^{+\infty}\lambda^{m}e^{-t\lambda}dE_{\lambda}f$
give that
\[ \|\mathcal{L}^{m}(P_{t}f-f)\|_{L^2(\mathbb{R}^n)}^{2}=\int_{0}^{\infty}\lambda^{2m}(1-e^{-t\lambda})^{2}d\|E_{\lambda}f\|_{L^2(\mathbb{R}^n)}^2. \]
Since $\lambda^{2m}(1-e^{-t\lambda})^2\leq \lambda^{2m}$ for all $t>0$ and $\lambda^{2m}(1-e^{-t\lambda})^2\to 0$ as $t\to 0^{+}$, the dominated convergence theorem yields that $\|\mathcal{L}^{m}(P_{t}f-f)\|_{L^2(\mathbb{R}^n)}=\|\triangle_{X}^{m}(P_{t}f-f)\|_{L^2(\mathbb{R}^n)}\to 0$ as $t\to 0^{+}$ for any integer $m\geq 1$. Thus, Corollary \ref{corollary2-4} and the  classical Sobolev embedding results give \eqref{3-19}.

Fix $s>0$ and $y\in \mathbb{R}^n$, we set $v(x,t):=h(x,y,t+s)$. By \eqref{3-27} we have
$ v(x,t)=(p_{t,x},p_{s,y})_{L^2(\mathbb{R}^n)}=P_{t}p_{s,y}(x)$.
Since $p_{s,y}\in L^2(\mathbb{R}^n)$, Lemma \ref{lemma3-2} yields that $v(x,t)\in C^{\infty}(\mathbb{R}^n\times \mathbb{R}^{+})$ and solves the equation $\partial_{t}u=\triangle_{X}u$. Changing $t$ to $t-s$, we obtain \eqref{3-20}.

Lemma \ref{lemma3-2} and \eqref{3-19} imply that for any $f_0\in C_{0}^{\infty}(\mathbb{R}^n)$, $u(x,t)=\int_{\mathbb{R}^n}h(x,y,t)f_{0}(y)dy$ is the bounded classical solution of the degenerate heat equation
\[ \left\{
     \begin{array}{ll}
 \partial_{t}u-\triangle_{X}u=0, & \hbox{on ~$\mathbb{R}^n\times \mathbb{R}^{+}$;} \\
      \lim_{t\to 0}u(x,t)=f_{0}(x), & \hbox{for all $x\in \mathbb{R}^n$.}
     \end{array}
   \right.\]
According to  \cite[Theorem 1.4]{Biagi2019} and \cite[Lemma 3.4]{Grigoryan2014}, we have $h(x,y,t-s)=\Gamma(s,y;t,x)$ for all $x,y\in \mathbb{R}^{n}$ and $s<t$, where $\Gamma(s,y;t,x)$ denotes the global fundamental solution of degenerate heat operator $\partial_{t}-\triangle_{X}$ at $(s,y)\in \mathbb{R}^{1+n}$. The global Gaussian bounds \eqref{3-21} follows from \cite[Theorem 2.4]{Biagi2020}.
\end{proof}

\subsection{Subelliptic Dirichlet heat kernel}

We now turn our attention to the subelliptic Dirichlet heat kernel.

 Suppose that $\Omega$ is an open bounded domain in $\mathbb{R}^n$. By Proposition \ref{prop3-1} (2),   the restricted Dirichlet form
 $(\mathcal{Q}, H_{X,0}^{1}(\Omega))$ admits a unique generator $\mathcal{L}_{\Omega}$, which is a non-negative defined self-adjoint operator in $L^2(\Omega)$ with  $\text{dom}(\mathcal{L}_{\Omega})\subset H_{X,0}^{1}(\Omega)$ such that
\begin{equation}\label{3-28}
   \mathcal{Q}(u,v)=(\mathcal{L}_{\Omega}u,v)_{L^2(\Omega)}
\end{equation}
for all $u\in \text{dom}(\mathcal{L}_{\Omega})$ and $v\in H_{X,0}^{1}(\Omega)$. Additionally, the operator $\mathcal{L}_{\Omega}$ is the unique self-adjoint extension of $-\triangle_{X}|_{\mathcal{D}(\Omega)}$ with the domain
\[ \begin{aligned}
\text{dom}(\mathcal{L}_{\Omega})&=\{u\in H_{X,0}^{1}(\Omega)|\exists c\geq 0~\mbox{such that}~|\mathcal{Q}(u,v)|\leq c\|v\|_{L^{2}(\Omega)}~\forall v\in H_{X,0}^{1}(\Omega)\}\\
&=\{u\in H_{X,0}^{1}(\Omega)|\triangle_{X} u\in L^2(\Omega)\}.
\end{aligned}
 \]
By using Proposition \ref{subelliptic-estimate-1} and Proposition \ref{prop2-16}, we can easily verify the well-definedness of the Dirichlet eigenvalue problem \eqref{1-5}. Precisely, we have
\begin{proposition}
\label{prop3-4}
Let $X=(X_{1},X_{2},\ldots,X_{m})$ be the homogeneous H\"{o}rmander vector fields defined on $\mathbb{R}^n$.  Suppose that $\Omega$ is a bounded open domain in $\mathbb{R}^n$. Then the Dirichlet eigenvalue problem \eqref{1-5} of homogeneous H\"{o}rmander operator $\triangle_{X}$ is well-defined. Specifically,  the self-adjoint operator $\mathcal{L}_{\Omega}$ admits a sequence of discrete Dirichlet eigenvalues
$ 0<\lambda_{1}\leq \lambda_{2}\leq\cdots\leq\lambda_{k}\leq\cdots$,
and $\lambda_{k}\to +\infty $ as $k\to +\infty$. Moreover, the corresponding Dirichlet eigenfunctions $\{\phi_{k}\}_{k=1}^{\infty}\subset C^{\infty}(\Omega)$ constitute an orthonormal basis of $L^2(\Omega)$ and an orthogonal basis of $H_{X,0}^{1}(\Omega)$.
\end{proposition}

Denote by $\{P_{t}^{\Omega}\}_{t\geq 0}$ the corresponding Dirichlet heat semigroup of $\mathcal{L}_{\Omega}$. Then, we have the following ultracontractivity result.
\begin{proposition}
\label{prop3-5}
For any $f\in L^2(\Omega)$ and $t>0$, we have
\begin{equation}\label{3-29}
  \|P_{t}^{\Omega}f\|_{L^{\infty}(\Omega)}\leq F_{K}(t)\|f\|_{L^2(\Omega)},
\end{equation}
where $F_{K}$ is the function defined in \eqref{3-8}, with $K=\overline{\Omega}$.
\end{proposition}
\begin{proof}
For any $f\in L^2(\Omega)$, we extend $f$ to a function in $L^2(\mathbb{R}^n)$  by setting $f=0$ outside $\Omega$. Since $|f|\geq \pm f$, we have
\begin{equation*}
  |P_{t}^{\Omega}f(x)|\leq  P_{t}^{\Omega}|f(x)|\qquad \mbox{for almost all}~x\in\Omega.
\end{equation*}
As a result of \eqref{3-4} and \eqref{3-13}, we obtain for almost all $x\in \Omega$,
\begin{equation*}
  |P_{t}^{\Omega}f(x)|\leq  P_{t}^{\Omega}|f(x)|\leq P_{t}|f(x)|\leq  \sup_{x\in\overline{\Omega}}|P_{t}f(x)|\leq F_{\overline{\Omega}}(t)\|f\|_{L^2(\mathbb{R}^n)},
\end{equation*}
which gives \eqref{3-29}.
\end{proof}

Thanks to Proposition \ref{prop3-5}, we can establish the subelliptic Dirichlet heat kernel of $\triangle_{X}$.
\begin{proposition}
\label{prop3-6}
The Dirichlet heat semigroup
$\{P_{t}^{\Omega}\}_{t\geq 0}$ admits a unique heat kernel $h_{D}(x,y,t)$, such that for any $f\in L^{2}(\Omega)$,
\begin{equation}\label{3-30}
  P_{t}^{\Omega}f(x)=e^{-t\mathcal{L}_{\Omega}}f(x)=\int_{\Omega}h_{D}(x,y,t)f(y)dy
\end{equation}
holds for all $x\in \Omega$ and $t>0$. The heat kernel $h_{D}(x,y,t)\in C^{\infty}(\Omega\times \Omega\times\mathbb{R}^{+})$ and satisfies the above properties (P1)-(P4) pointwise. Moreover, we have
\begin{enumerate}[(i)]
  \item For any $t>0$ and any $k\in \mathbb{N}$, $h_{D}(\cdot,\cdot,t)\in L^2(\Omega\times \Omega)$ and
\begin{equation}\label{3-31}
  \partial_{t}^{k}h_{D}(x,y,t)=\sum_{j=1}^{\infty}(-\lambda_{j})^{k}e^{-\lambda_{j}t}\phi_j(x)\phi_j(y).
\end{equation}
Additionally, $\{\phi_{j}\}_{j=1}^{\infty} \subset L^{\infty}(\Omega)\cap C^{\infty}(\Omega)$, and  the series \eqref{3-31} converges absolutely and uniformly on $\Omega\times\Omega\times [a,+\infty)$ for any $a>0$.
\item For any fixed $y\in \Omega$, $h_D(x,y,t)$ is the solution of
\begin{equation}\label{3-32}
\left(\frac{\partial}{\partial t}-\triangle_{X}\right)h_D(x,y,t)=0 \quad \mbox{for all}~~ (x,t)\in\Omega\times \mathbb{R}^{+}.
\end{equation}
\item For any $f\in L^2(\Omega)$, $P_{t}^{\Omega}f\in C^{\infty}(\Omega\times \mathbb{R}^{+})$ and solves the degenerate heat equation
\[ \left\{
     \begin{array}{ll}
    \partial_{t}u-\triangle_{X}u=0, & \hbox{on ~$\Omega\times \mathbb{R}^{+}$;} \\
     \lim_{t\to 0^{+}}u(x,t)=f(x) , & \hbox{in $L^2(\Omega)$.}
     \end{array}
   \right.\]
\item
For all $t>0$ and $y\in \Omega$,
\begin{equation}\label{3-33}
  h_{D}(\cdot,y,t)\in {\rm dom}(\mathcal{L}_{\Omega})\subset H_{X,0}^{1}(\Omega).
\end{equation}
Hence, $h_{D}(x,y,t)$ is usually referred to the Dirichlet heat kernel of $\triangle_{X}$, since it admits the Dirichlet boundary condition in weak sense.
\end{enumerate}
\end{proposition}
\begin{proof}
It follows from Proposition \ref{prop3-5} and \cite[Lemma 3.7]{Grigoryan2014} that
$P_{t}$ admits a unique heat kernel $h_{D}(x,y,t)$ satisfying \eqref{3-30} for all $t>0$ and almost all $x\in \Omega$. Moreover, we have
\begin{equation}\label{3-34}
  0\leq h_{D}(x,y,t)\leq F_{\overline{\Omega}}\left(\frac{t}{2}\right)^{2}\quad \mbox{for all}~ t>0~~\mbox{and almost all}~~x,y\in \Omega.
\end{equation}
Note that (P2) and \eqref{3-34} imply that $h_{D}(x,\cdot,t)=h_{D}(\cdot,x,t)\in L^2(\Omega)$ holds for almost all $x\in \Omega$ and all $t>0$. Then, we deduce from \eqref{3-30} and \cite[A.28.(b), p. 455]{Grigor'yan2009} that
\begin{equation}\label{3-35}
  \int_{\Omega}h_{D}(\cdot,z,t)\phi_{j}(z)dz=P_{t}^{\Omega}\phi_{j}=e^{-t\mathcal{L}_{\Omega}}\phi_{j}=e^{-t\lambda_{j}}\phi_{j}.
\end{equation}
Proposition \ref{prop3-4} shows that $\{\phi_{k}\}_{k=1}^{\infty}$ is an orthonormal basis of $L^2(\Omega)$. Thus, we obtain $ h_{D}(\cdot,x,t)=\sum_{j=1}^{\infty}e^{-t\lambda_{j}}\phi_{j}\phi_{j}(x)$,
that is,
\begin{equation}\label{3-36}
  h_{D}(x,y,t)=\sum_{j=1}^{\infty} e^{-t\lambda_{j}}\phi_{j}(x)\phi_{j}(y),
\end{equation}
where the series converges in $L^2(\Omega)$ in variable $y$ for almost all $x\in \Omega$ and all $t>0$. The estimate \eqref{3-34} then implies
\begin{equation}\label{3-37}
  \|h_{D}(\cdot,\cdot,t)\|_{L^2(\Omega\times\Omega)}\leq  C|\Omega|(1+2^{N}t^{-N})^{2}<\infty\qquad \forall t>0.
\end{equation}
It follows from \cite[Lemma 10.7, Lemma 10.14]{Grigor'yan2009} and \eqref{3-37} that
\begin{equation}\label{3-38}
  \sum_{j=1}^{\infty}e^{-2t\lambda_{j}}=\|h_{D}(\cdot,\cdot,t)\|_{L^2(\Omega\times\Omega)}^{2}\leq  C^{2}|\Omega|^{2}(1+2^{N}t^{-N})^{4}<\infty\qquad \forall t>0.
\end{equation}
Since $\{\phi_{k}(x)\phi_{k}(y)\}_{k=1}^{\infty}$ is obviously orthonormal in $L^2(\Omega\times\Omega)$, \eqref{3-38} derives that the series \eqref{3-36} converges in $L^2(\Omega\times \Omega)$, and $h_{D}(\cdot,\cdot,t)\in L^2(\Omega\times \Omega)$ for all $t>0$.

Applying Proposition \ref{prop3-5} for $f=\phi_{j}\in C^{\infty}(\Omega)\cap L^2(\Omega)$ and using \eqref{3-35}, we get
\begin{equation}\label{3-39}
  \sup_{x\in \Omega}|e^{-t\lambda_{j}}\phi_{j}(x)|\leq F_{\overline{\Omega}}(t)\leq C(1+t^{-N})\qquad \forall t>0,
\end{equation}
which means $\sup_{x\in \Omega}|\phi_{j}(x)|\leq C(1+\lambda_{j}^{N})$. Meanwhile, by \eqref{3-39} we obtain that
\begin{equation}\label{3-40}
  |e^{-\frac{t}{2}\lambda_{j}}\phi_{j}(x)\phi_{j}(y)|\leq F_{\overline{\Omega}}\left(\frac{t}{4}\right)^{2}\leq C(1+4^{N}t^{-N})^{2} \qquad \forall t>0,~x,y\in \Omega.
\end{equation}
For any $k\in \mathbb{N}$ and any $a>0$, \eqref{3-38} and \eqref{3-40} give that, for all $x,y\in \Omega$ and $t\geq a$,
\begin{equation*}
 \begin{aligned}
 \sum_{j=1}^{\infty}\lambda_{j}^{k}|e^{-t\lambda_{j}}\phi_j(x)\phi_j(y)| &=\sum_{j=1}^{\infty}\lambda_{j}^{k}e^{-\frac{t\lambda_{j}}{4}}|e^{-\frac{t\lambda_{j}}{2}}\phi_j(x)\phi_j(y)|e^{-\frac{t\lambda_{j}}{4}}\\
&\leq C(1+4^{N}a^{-N})^{2}\sum_{j=1}^{\infty}\left(\lambda_{j}^{k}e^{-\frac{a\lambda_{j}}{4}} \right)e^{-\frac{a\lambda_{j}}{4}}\\
 &\leq  C(1+4^{N}a^{-N})^{2}(1+16^{N}a^{-N})^{4}\max_{u\in \mathbb{R}^{+}}(u^{k}e^{-\frac{a}{4}u})<\infty,
\end{aligned}
\end{equation*}
which yields the absolutely and uniformly convergence of \eqref{3-31} on $\Omega\times\Omega\times [a,+\infty)$.

For any fixed $y\in \Omega$ and each $j\geq 1$, the term $e^{-t\lambda_{j}}\phi_j(x)\phi_j(y)$ solves $\partial_{t}u-\triangle_{X}u=0$ on $\Omega\times \mathbb{R}^{+}$. The  local
 uniform convergence of \eqref{3-31} implies that $h_{D}(x,y,t)\in \mathcal{D}'(\Omega\times\mathbb{R}^{+})$ is a weak solution to \eqref{3-32}. Analogously,  $h_{D}(x,y,t)\in \mathcal{D}'(\Omega\times\Omega\times\mathbb{R}^{+})$ is a weak solution of equation $[\partial_{t}-\frac{1}{2}(\triangle_{X}^{x}+\triangle_{X}^{y})]u(x,y,t)=0$, since for each $j\geq 1$, $e^{-t\lambda_{j}}\phi_j(x)\phi_j(y)$ is a solution of $[\partial_{t}-\frac{1}{2}(\triangle_{X}^{x}+\triangle_{X}^{y})]u(x,y,t)=0$. The hypoellipticity of $\partial_{t}-\frac{1}{2}(\triangle_{X}^{x}+\triangle_{X}^{y})$ indicates that $h_{D}(x,y,t)\in C^{\infty}(\Omega\times\Omega\times\mathbb{R}^{+})$. This implies the properties (P1) and (P2) hold pointwise. Additionally, \eqref{3-31} derives that for $t>0$ and $s>0$,
\begin{align*}
\int_{\Omega}h_D(x,z,t)h_D(z,y,s)dz&=\int_{\Omega}\left(\sum_{i=1}^{\infty}
e^{-\lambda_{i}t}\phi_{i}(x)\phi_{i}(z) \right)\left(\sum_{j=1}^{\infty}e^{-\lambda_{j}s}\phi_{j}(z)\phi_{j}(y) \right)dz\\
 &=\sum_{i=1}^{\infty}e^{-\lambda_{i}(t+s)}\phi_{i}(x)\phi_{i}(y)=h_D(x,y,t+s),
\end{align*}
which indicates that the property (P3) holds pointwise.

 Given a function $f_{0}\in L^2(\Omega)$, we have $f_{0}=\sum_{i=1}^{\infty}a_{i}\phi_{i}$ in $L^{2}(\Omega)$ with $a_{i}=(f_0,\phi_{i})_{L^2(\Omega)}$.
In terms of \eqref{3-30} and \eqref{3-31}, we have for any $t>0$,
\begin{align*}
P_{t}^{\Omega}f=\int_{\Omega}\left(\sum_{i=1}^{\infty}e^{-\lambda_{i}t}\phi_{i}\phi_{i}(y)\right)
\left(\sum_{j=1}^{\infty}a_{j}\phi_{j}(y)\right)dy=\sum_{i=1}^{\infty}e^{-\lambda_{i}t}a_{i}\phi_{i}~~\mbox{in}~~ L^2(\Omega).
\end{align*}
Note that the Parseval's identity gives that $\sum_{i=1}^{\infty}a_{i}^2= \|f_{0}\|_{L^2(\Omega)}^{2}<+\infty$. Using \eqref{3-38} and \eqref{3-39}, we have for any $t>0$ and all $x\in \Omega$,
\[ \begin{aligned}
 \sum_{i=1}^{\infty}e^{-\lambda_{i}t}|a_{i}|\cdot|\phi_{i}(x)|&\leq \|f_{0}\|_{L^2(\Omega)}\sum_{i=1}^{\infty}e^{-\frac{\lambda_{i}t}{2}}\cdot e^{-\frac{\lambda_{i}t}{2}}|\phi_{i}(x)|\leq C(1+2^{N}t^{-N})(1+8^{N}t^{-N})^{4},
\end{aligned} \]
which means  $\sum_{i=1}^{\infty}e^{-\lambda_{i}t}a_{i}\phi_{i}(x)$ converges uniformly on $\Omega\times [a,+\infty)$ for any $a>0$. Since each term $e^{-\lambda_{i}t}a_{i}\phi_{i}(x)$ solves $\partial_{t}u-\triangle_{X}u=0$ on $\Omega\times \mathbb{R}^{+}$, using the hypoellipticity of $\partial_{t}-\triangle_{X}$, we deduce that $P_{t}^{\Omega}f\in C^{\infty}(\Omega\times \mathbb{R}^{+})$ and solves $\partial_{t}u-\triangle_{X}u=0$ on $\Omega\times \mathbb{R}^{+}$. This implies \eqref{3-30} holds for all $x\in \Omega$ and $t>0$.

Using \eqref{3-34} and (P1), we have $h_{D}(\cdot,y,t_{1})\in L^2(\Omega)$ for all $t_{1}>0$ and $y\in \Omega$.
Thus,  \eqref{3-30}, (P3) and (S6) imply that
\begin{equation*}
\begin{aligned}
h_{D}(x,y,t)&=\int_{\Omega}h_{D}\left(x,z,t-t_{1}\right)h_{D}\left(z,y,t_{1}\right)dz=\left(P_{t-t_{1}}^{\Omega}h_{D}(\cdot,y,t_{1})\right)(x)\in \text{dom}(\mathcal{L}_{\Omega})
\end{aligned}
\end{equation*}
for all $t>t_{1}>0$, which yields \eqref{3-33}.
\end{proof}

\subsection{Comparison of heat kernels}

Based on the aforementioned results, we can establish the following comparison results for the subelliptic heat kernels,  which constitutes a crucial component in the proof of Theorem \ref{thm1}. Specifically, for any $(x,y,t)\in \Omega\times \Omega\times \mathbb{R}^{+}$, we define
\begin{equation}\label{3-41}
  E(x,y,t):=h(x,y,t)-h_{D}(x,y,t).
\end{equation}
Then, we have
\begin{proposition}
\label{prop3-7}
Assume that $X=(X_{1},X_{2},\ldots,X_{m})$ and $\Omega$ satisfy the conditions of
Theorem \ref{thm1}. Let $\eta(x):=\frac{d_{X}^{2}(x,\partial\Omega)}{A_{1}Q}$ be the continuous function defined on $\Omega$, where $A_{1}$ is the positive constant in \eqref{3-21}, $Q$ is the homogenous dimension, and  $ d_{X}(x,\partial\Omega):=\inf_{y\in \partial\Omega}d_{X}(x,y)$. Then we have the following estimates:
\begin{itemize}
  \item For any $x\in \Omega$ and $0<t\leq \eta(x)$,
\begin{equation}\label{3-42}
E(x,x,t)\leq \frac{2A_{1}C_{3}}{|B_{d_{X}}(x,\sqrt{t})|}e^{-\frac{d_{X}^{2}(x,\partial\Omega)}{A_{1}t}},
\end{equation}
where $C_{3}$ is a positive constant appeared in Corollary \ref{corollary2-1}.
  \item For any $x\in \Omega$ and $t>0$,
\begin{equation}\label{3-43}
  E(x,x,t)\geq 0.
\end{equation}
\end{itemize}

\end{proposition}
\begin{proof}
Proposition \ref{prop3-1} (1) indicates that  $(\mathcal{Q}, H_{X}^1(\mathbb{R}^n))$ is a regular and local Dirichlet form in $L^2(\mathbb{R}^n)$. Note that $h(x,y,t)\in C^{\infty}(\mathbb{R}^n\times \mathbb{R}^n\times \mathbb{R}^{+})$ is locally bounded in $\mathbb{R}^n\times \mathbb{R}^n\times \mathbb{R}^{+}$. By means of \cite[Theorem 5.1]{Grigoryan2010-jfa}, for any compact subset $K\subset \Omega$ and any $(x,t)\in \Omega\times \mathbb{R}^{+}$, we have
\begin{equation}\label{3-44}
  h(x,x,t)\leq h_{D}(x,x,t)+2\sup_{0<s\leq t}\sup_{z\in \Omega\setminus K}h(x,z,s).
\end{equation}

We denote by $\Omega_{j}:=\{x\in \Omega|d_{X}(x,\partial\Omega)>\frac{1}{j}\}$ for any  $j\geq 1$. It follows that $\Omega=\bigcup_{j=1}^{\infty} \Omega_{j}$ and $\overline{\Omega_{j}}\subset \Omega_{j+1}\subset \Omega$. Moreover, if
 $\Omega_{j}\neq \varnothing$, then $\Omega_{j}$ is a pre-compact open subset of $\Omega$. As a result, for any $x\in \Omega$, there exists a positive integer $j_0=j_{0}(x)\geq 1$  such that $x\in \Omega_{j}$ for all $j\geq j_0$. In addition, by Definition \ref{def2-3} we can easily verify the following facts:
\begin{enumerate}[(i)]
  \item For any $x\in \Omega_j$ and $y\in \Omega\setminus\overline{\Omega_j}$, we have $d_ {X}(x,\partial \Omega_{j})\leq d_{X}(x,y)$.
  \item $\lim_{j\to+\infty}d_X(x,\partial \Omega_{j})=d_X(x,\partial \Omega)$.
\end{enumerate}

For any fixed $x\in \Omega$ and $j\geq j_0$ (such that $x\in \Omega_{j}$),  \eqref{3-21}, (i) and Corollary \ref{corollary2-1} imply that
\begin{equation}\label{3-45}
\begin{aligned}
\sup_{0<s\leq t}\sup_{z\in \Omega\setminus \overline{\Omega_j}}h(x,z,s) &\leq \sup_{0<s\leq t}\sup_{z\in \Omega\setminus\overline{\Omega_j}}
 \frac{A_{1}}{|B_{d_{X}}(x,\sqrt{s})|}e^{-\frac{ d_{X}^{2}(x,z)}{A_{1}s}}\leq \sup_{0<s\leq t}\frac{A_{1}}{|B_{d_{X}}(x,\sqrt{s})|}e^{-\frac{ d_{X}^{2}(x,\partial \Omega_j )}{A_{1}s}}\\
&=\sup_{0< s\leq t}\frac{A_{1}|B_{d_{X}}(x,\sqrt{t})|}{|B_{d_{X}}(x,\sqrt{s})|}\cdot \frac{1}{|B_{d_{X}}(x,\sqrt{t})|}e^{-\frac{d_{X}^{2}(x,\partial\Omega_{j})}{A_{1}s} }\\
 &\leq \frac{A_{1}C_{3}}{|B_{d_{X}}(x,\sqrt{t})|}\sup_{0< s\leq t}\left(\frac{t}{s}\right)^{\frac{Q}{2}}e^{-\frac{d_{X}^{2}(x,\partial\Omega_{j})}{A_{1}s} }.\\
 \end{aligned}
\end{equation}
Observe that the function $g(s):=s^{-\frac{Q}{2}}e^{-\frac{d_{X}^{2}(x,\partial\Omega_{j})}{A_{1}s}}$
is increasing on $\left(0,\frac{2d_{X}^{2}(x,\partial\Omega_{j})}{A_{1}Q}  \right]$ and decreasing on $\left(\frac{2d_{X}^{2}(x,\partial\Omega_{j})}{A_{1}Q},+\infty\right)$ with
$\lim_{s\to 0^{+}}g(s)=\lim_{s\to +\infty}g(s)=0$.
Combining \eqref{3-44} and \eqref{3-45}, for any fixed  $x\in\Omega$ and $j\geq j_0$, we have
\begin{equation}\label{3-46}
  E(x,x,t)\leq \frac{2A_{1}C_{3}}{|B_{d_{X}}(y,\sqrt{t})|}e^{-\frac{d_{X}^{2}(y,\partial\Omega_{j})}{A_{1}t}}\qquad \forall 0<t\leq \frac{2d_{X}^{2}(x,\partial\Omega_{j})}{A_{1}Q}.
\end{equation}
Since the positive constants $A_{1},C_{3}, Q$ in  \eqref{3-46} are independent of $j$, we can take $j\to +\infty$ and  derive from (ii) that
\begin{equation*}
  E(x,x,t)\leq \frac{2A_{1}C_{3}}{|B_{d_{X}}(x,\sqrt{t})|}  e^{-\frac{d_{X}^{2}(x,\partial\Omega)}{A_{1}t}}
\end{equation*}
holds for all $x\in \Omega$ and $0<t\leq \eta(x):=\frac{d_{X}^{2}(x,\partial\Omega)}{A_{1}Q}$, which proves \eqref{3-42}.

 On the other hand, using \eqref{3-4}, \eqref{3-18} and \eqref{3-30} we deduce that for any non-negative functions $f,g\in L^2(\Omega)$ and any $t>0$,
\begin{equation*}
  \int_{\Omega} \int_{\Omega}E(x,y,t)f(y)g(x)dxdy= \int_{\Omega} \int_{\Omega}(h(x,y,t)-h_{D}(x,y,t))f(y)g(x)dxdy\geq 0.
\end{equation*}
By \cite[Lemma 3.4]{Grigoryan2014},  $E(x,x,t)\geq 0$ for all $(x,t)\in \Omega\times \mathbb{R}^{+}$, which yields \eqref{3-43}.
\end{proof}

\section{Asymptotic behaviour of integral of $\Lambda(x,r)^{-1}$}
\label{Section4}

In this section, we investigate the explicit asymptotic behaviour of the following integral:
\begin{equation}\label{4-1}
 J_{\Omega}(r):=\int_{\Omega}\frac{dx}{\Lambda(x,r)}\qquad\mbox{as}~~r\to 0^{+},
\end{equation}
where $\Lambda(x,r)$ is the Nagel-Stein-Wainger
polynomial mentioned in \eqref{2-3}, and  $\Omega$ is a bounded open domain in $\mathbb{R}^n$ containing the origin. The asymptotic behaviour of the integral given by \eqref{4-1} will be essential  in the proof of Theorem \ref{thm2}, as it is directly related to the Ball-Box theorem (as stated in Proposition \ref{prop2-8} above).

One potential strategy for dealing with \eqref{4-1},  inspired by Nagel-Stein-Wainger \cite{Stein1985}, is to
divide the domain $\Omega$ into subregions  $\Omega_{I,r}$, where $\Omega_{I,r}:=\{x\in \Omega|\lambda_{I}(x)r^{d(I)}=\max_{J}|\lambda_{J}(x)r^{d(J)}|\}$ for any $0<r<1$. The estimation of
$J_{\Omega}(r)$ then involves estimating  the integrals of $(\lambda_{I}(x)r^{d(I)})^{-1}$ over these subregions. However, since these subregions cannot be expressed explicitly for the abstract vector fields that only satisfy H\"{o}rmander's condition, computing these integrals presents significant challenges.

Fortunately, if we restrict ourselves to homogeneous H\"{o}rmander vector fields, the corresponding smooth functions $\lambda_{I}$ are polynomials. This enables us to study the explicit asymptotic behaviour of integral $J_{\Omega}(r)$ by improving the above idea. We briefly outline our approach. By utilizing the degenerate components and homogeneity property (H.1), we first show that $J_{\Omega}(r)$ is  asymptotically equivalent to the integral of $\Lambda(x,r)^{-1}$ over any $v$-dimensional bounded open domain $\Omega_{0}\subset \mathbb{R}^{v}$ that contains the origin in $\mathbb{R}^{v}$. Here, $v$ denotes the number of degenerate components of the vector fields $X$, and $\mathbb{R}^{v}$ is the projection of $\mathbb{R}^n$ in the directions of all degenerate components. Then, by employing the resolution of singularities in algebraic geometry, we can identify a real analytic map $\rho$ on a real analytic manifold, such that every $\lambda_{I}\circ\rho$ has the form $c(u)m(u)$ in some local coordinates, where $m(u)$ is a monomial, and $c(u)$ is a non-vanishing analytic function. Therefore, we can estimate $J_{\Omega}(r)$ using the following chain of approximations:
\begin{align*}
  J_{\Omega}(r)&=\int_{\Omega}\frac{dx}{\Lambda(x,r)}\approx \int_{\Omega_0}\frac{dx_{i_1}\cdots dx_{i_v}}{\sum_{I}|\lambda_I(x)|r^{d(I)}}&&\text{(Lemma \ref{lemma4-1})}\\
  & \approx \sum_{j=1}^l\int_{(-1,1)^{v}}\frac{|u^{q_{j}}| du}{\sum_{I}|u^{p_{j,I}}|r^{d(I)}} &&\text{(Resolution of singularities)}\\
  &\approx \sum_{j=1}^l\int_{(0,1]^{v}}\frac{u^{q_{j}} du}{\sum_{I}u^{p_{j,I}}r^{d(I)}},&&
\end{align*}
where $l$ is a positive integer, $\mathbf{1}=(1,\ldots,1)$ is the vector in $\mathbb{R}^{v}$, $q_{j}$ and $p_{j,I}$ are $v$-dimensional multi-indexes. By changing the coordinates, we have
\begin{align*}
J_{\Omega}(r)&\approx \sum_{j=1}^l\int_{(0,1]^{v}}\frac{u^{q_{j}} du}{\sum_{I}u^{p_{j,I}}r^{d(I)}}\approx\sum_{j=1}^l\int_{[0,+\infty)^{v}}\frac{e^{-\langle q_{j}+\textbf{1},u\rangle}du}{\sum_{I}e^{-\langle p_{j,I},u\rangle}r^{d(I)}}~~\mbox{for}~~0<r<1.
\end{align*}
Next, we divide the set $[0,+\infty)^{v}$ by some polyhedrons of the form
\[ \begin{aligned}
P_{j,I,r}&=\{u\in [0,+\infty)^v|e^{-\langle p_{j,I},u\rangle}r^{d(I)}=\max_{J}e^{-\langle p_{j,J},u\rangle}r^{d(J)}\}\\
&=\left\{u\in [0,+\infty)^v\bigg|\langle p_{j,I}-p_{j,J},u\rangle\leq (d(J)-d(I))\left(\ln \frac{1}{r}\right)~~\mbox{for all $n$-tuples}~J\right\},
\end{aligned}
\]
and show that
\[ J_{\Omega}(r) \approx\sum_{j=1}^l\int_{[0,+\infty)^{v}}\frac{e^{-\langle q_{j}+\textbf{1},u\rangle}du}{\sum_{I}e^{-\langle p_{j,I},u\rangle}r^{d(I)}}\approx \sum_{j=1}^l\sum_{I} \frac{1}{r^{d(I)}}\int_{P_{j,I,r}}e^{\langle p_{j,I}- q_{j}-\textbf{1},u\rangle} du~~\mbox{for}~~0<r<1.\]
After some refined analysis involving convex geometry, we finally obtain the following estimate:
\[ \frac{1}{r^{d(I)}}\int_{P_{j,I,r}}e^{\langle p_{j,I}- q_{j}-\textbf{1},u\rangle} du \approx r^{-\tilde{\alpha}}|\ln r|^{\tilde{d}}~~\mbox{as}~~r\to 0^{+} \]
with $\tilde{\alpha}\in \mathbb{Q}$ and $\tilde{d}\in \{0,1,\ldots, v\}$. Consequently, we conclude that
\[ J_{\Omega}(r)\approx r^{-Q_0}|\ln r|^{d_0}~~\mbox{as}~~r\to 0^{+},\]
where $Q_0\in \mathbb{Q}$ and $d_0\in \{0,1,\ldots, v\}$. Furthermore, we will also provide the optimal ranges of indexes $Q_0$ and $d_0$.

\subsection{Simplification procedures for the integral $J_{\Omega}(r)$}
We start with the following technical lemma.
\begin{lemma}
\label{lemma4-1}
For any given bounded open subsets $\Omega_{1}$ and $\Omega_{2}$ in $\mathbb{R}^n$ containing the origin, we have $J_{\Omega_{1}}(r)\approx J_{\Omega_{2}}(r)$, where the notation $J_{D}(r)$ denotes
\[ J_{D}(r)=\int_{D}\frac{dx}{\Lambda(x,r)} ~~~\mbox{for the set}~~D\subset\mathbb{R}^{n}.\]
Moreover, let $\{x_{i_{1}},\ldots,x_{i_{v}}\}$ be the collection of all degenerate components of $X$, and let $\mathbb{R}^{v}:=\mathbb{R}_{x_{i_{1}}}\times \cdots\times \mathbb{R}_{x_{i_{v}}}$ be the projection of $\mathbb{R}^n$ in directions $\{x_{i_{1}},\ldots,x_{i_{v}}\}$. Then we have
\[J_{\Omega_{1}}(r)\approx J_{\Omega_{2}}(r)\approx J_{\Omega_{3},v}(r)\approx J_{\Omega_{4},v}(r),\]
 where $\Omega_{3}$ and $\Omega_{4}$ are any bounded open subsets of $\mathbb{R}^{v}$ containing the origin, and the notation $J_{D',v}(r)$ denotes
 \[ J_{D',v}(r)=\int_{D'}\frac{dx_{i_{1}}\cdots dx_{i_{v}} }{\Lambda(x,r)} ~~~\mbox{for the set}~~D'\subset\mathbb{R}^{v}.\]
\end{lemma}

\begin{proof}
Given any positive constant $p>0$, we consider the corresponding $n$-dimensional $\delta_{t}$-box
\[ D(p):=\{(x_{1},\ldots,x_{n})\in \mathbb{R}^n| |x_{j}|<p^{\alpha_{j}}, j=1,\ldots,n\}. \]
By Proposition \ref{prop2-7}, we have for any $0<p<q$,
\begin{equation*}
\begin{aligned}
J_{D(p)}(r)&=\int_{D(p)}\frac{dx}{\Lambda(x,r)}
=  \int_{D(p)}\frac{dx}{\sum_{k=w}^{Q}f_{k}(x)r^{k}}\\
&=\left(\frac{p}{q}\right)^{Q}\int_{D(q)}\frac{dy}{\sum_{k=w}^{Q}f_{k}(\delta_{\frac{p}{q}}(y))r^{k}}
=\left(\frac{p}{q}\right)^{Q}\int_{D(q)}\frac{dy}{\sum_{k=w}^{Q}\left(\frac{p}{q}\right)^{Q-k}f_{k}(y)r^{k}}
\\ &\approx J_{D(q)}(r).
\end{aligned}
\end{equation*}
In particular, for any fixed $p>0$,
\begin{equation}\label{4-2}
  J_{D(p)}(r)\approx J_{D(1)}(r).
\end{equation}
For any bounded open set $\Omega_{1}\subset \mathbb{R}^n$ containing the origin, there exist positive constants $0<p<q$ such that $D(p)\subset \Omega_{1}\subset D(q)$. Applying \eqref{4-2} to
$ J_{D(p)}(r)\leq J_{\Omega_{1}}(r)\leq J_{D(q)}(r)$,
we derives $J_{\Omega_{1}}(r)\approx J_{D(1)}(r)$. This means $J_{\Omega_{1}}(r)\approx J_{\Omega_{2}}(r)$ for any bounded open sets $\Omega_{1},\Omega_{2}$ in $\mathbb{R}^n$ containing the origin.

Similarly, for any bounded open set $\Omega_{3}\subset \mathbb{R}^{v}$ that contains the origin, there exist $0<p<q$ such that the projections of $n$-dimensional $\delta_{t}$-boxes $D(p)$ and $D(q)$ in the directions $\{x_{i_{1}},\ldots,x_{i_{v}}\}$, denoted by  $D(p,v)$ and $D(q,v)$, respectively, satisfy
\begin{equation}\label{4-3}
 D(p,v)\subset \Omega_{3}\subset D(q,v),
\end{equation}
where $D(p,v)=\{(x_{i_{1}},\ldots,x_{i_{v}})\in \mathbb{R}^{v}||x_{i_{j}}|<p^{\alpha_{i_{j}}},~j=1,\ldots,v\}$. From \eqref{4-3} we have
\begin{equation}\label{4-4}
 J_{D(p,v),v}(r)\leq J_{\Omega_{3},v}(r)\leq J_{D(q,v),v}(r).
\end{equation}
Moreover, a straightforward calculation shows that
\[ 2^{n-v}p^{Q-\alpha(X)}J_{D(p,v),v}(r)=J_{D(p)}(r)~~~\mbox{and}~~~
2^{n-v}q^{Q-\alpha(X)}J_{D(q,v),v}(r)=J_{D(q)}(r),\]
where $\alpha(X)$ is the sum of all degenerate indexes of $X$ defined in \eqref{2-11}. Therefore, we have  $J_{\Omega_{3},v}(r)\approx J_{D(1)}(r)$, which implies that
\[J_{\Omega_{1}}(r)\approx J_{\Omega_{2}}(r)\approx J_{\Omega_{3},v}(r)\approx J_{\Omega_{4},v}(r)\]
 for any bounded open sets $\Omega_{1},\Omega_{2}\subset\mathbb{R}^{n}$ and $\Omega_{3},\Omega_{4}\subset \mathbb{R}^{v}$  containing the origin.
\end{proof}

If the vector fields $X$ have a unique degenerate component $x_{j}$, then the non-constant function $\lambda_{I}$ is a monomial that depends on the variable $x_{j}$.  It is worth noting that there are many homogeneous H\"{o}rmander vector fields that possess a unique degenerate component. A well-known example is the Grushin type vector fields $X=(\partial_{x_{1}},x_1\partial_{x_{2}})$ in $\mathbb{R}^2$. In the following, we will study the asymptotic behaviour of $J_{\Omega}(r)$ for this simple case.

\begin{proposition}
\label{prop4-1}
Let $X=(X_{1},X_{2},\ldots,X_{m})$ be homogeneous H\"{o}rmander vector fields defined on $\mathbb{R}^n$, and let $\Omega\subset \mathbb{R}^n$ be a bounded open domain containing the origin. If $x_{j}$ is the unique degenerate component of $X$ associated with the degenerate index $\alpha_{j}$, then we have
\[ J_{\Omega}(r)\approx\left\{
     \begin{array}{ll}
      \frac{1}{r^{Q-\alpha_{j}}}|\ln r|, & \hbox{if $Q-w=\alpha_{j}$;} \\[2mm]
      \frac{1}{r^{Q-\alpha_{j}}}, & \hbox{if $Q-w>\alpha_{j}$,}
     \end{array}
   \right.~~~\mbox{as}~~r\to 0^{+},\]
where $Q$ is the homogeneous dimension and $w=\min_{x\in\mathbb{R}^n}\nu(x)$.

\end{proposition}

\begin{proof}
Proposition \ref{prop2-7} gives that
$ \Lambda(x_{j},r)=\sum_{k=w}^{Q}f_k(x_{j})r^k$,
where $f_k(x_{j})=\sum_{d(I)=k}|\lambda_I(x_{j})|$ is a $\delta_t$-homogeneous non-negative continuous function of degree $Q-k$. For each $n$-tuple $I$ with $d(I)=k$, by Proposition \ref{prop2-4} we know that $\lambda_{I}(x_{j})$ is a polynomial in the form of \eqref{2-1} and satisfies $\lambda_{I}(t^{\alpha_{j}}x_{j})=t^{Q-k}\lambda_{I}(x_{j})$. This means $\lambda_{I}(x_{j})$ is a monomial and satisfies
\begin{equation}\label{4-5}
 \lambda_{I}(x_{j})=\left\{
                      \begin{array}{ll}
                        \lambda_{I}(1)x_{j}^{\frac{Q-k}{\alpha_{j}}}, & \hbox{if $\frac{Q-k}{\alpha_{j}}\in \mathbb{N}$ and $\lambda_{I}(x_{j})\not\equiv 0$;} \\[2mm]
                        0, & \hbox{if $\frac{Q-k}{\alpha_{j}}\notin \mathbb{N}$ or $\lambda_{I}(x_{j})\equiv 0$.}
                      \end{array}
                    \right.
\end{equation}
 If $Q-w<\alpha_j$, by \eqref{4-5} we have $f_w(x_{j})=\sum_{d(I)=w}|\lambda_I(x_{j})|\equiv 0$, which contradicts the fact that $f_w(x_{j})\not\equiv0$ as stated in Proposition \ref{prop2-7}. Thus, we only need to consider the following two cases:\\
\emph{\textbf{Case 1:}} $Q-w=\alpha_{j}$. Combining \eqref{2-6} and \eqref{4-5}, we have
$ \Lambda(x_{j},r)=c_{w}|x_{j}|r^{Q-\alpha_{j}}+f_{Q}r^{Q}$,
where $c_{w}$ and $f_{Q}$ are some positive constants.
Using Lemma \ref{lemma4-1}, we obtain
\begin{equation*}
\begin{aligned}
J_{\Omega}(r)&=\int_{\Omega}\frac{dx}{\Lambda(x,r)}
\approx \int_{-1}^{1}\frac{dx_{j}}{c_{w}|x_{j}|r^{Q-\alpha_{j}}+f_{Q}r^{Q}}\\
&=\frac{2}{r^{Q-\alpha_{j}}}\int_{0}^{1}\frac{dz}{c_{w}z+f_{Q}r^{\alpha_{j}}}
\approx \frac{1}{r^{Q-\alpha_{j}}}|\ln r|~~~~~\mbox{as}~~r\to 0^{+}.
\end{aligned}
\end{equation*}\\
\emph{\textbf{Case 2:}} $Q-w>\alpha_{j}$. It follows from  \eqref{2-6} and \eqref{4-5} that
$ \Lambda(x_{j},r)=\sum_{k=w}^{Q-1}c_{k}|x_{j}|^{\frac{Q-k}{\alpha_{j}}}r^{k}+f_{Q}r^{Q}$,
where $c_{w}, f_{Q}>0$ are some positive constants, and $c_{j}\geq 0$ for $w+1\leq j\leq Q-1$. Then, by Lemma \ref{lemma4-1} we have
\begin{equation*}
\begin{aligned}
J_{\Omega}(r)&=\int_{\Omega}\frac{dx}{\Lambda(x,r)}
\approx\int_{-1}^{1}\frac{dx_{j}}{\sum_{k=w}^{Q-1}c_{k}|x_{j}|^{\frac{Q-k}{\alpha_{j}}}r^{k}+f_{Q}r^{Q}}\\
&=2\int_{0}^{1}\frac{dz}{\sum_{k=w}^{Q-1}c_{k}z^{\frac{Q-k}{\alpha_{j}}}r^{k}+f_{Q}r^{Q}}
=\frac{2\alpha_{j}}{r^{Q-\alpha_{j}}}\int_{0}^{\frac{1}{r}}\frac{y^{\alpha_{j}-1}dy}
{\sum_{k=w}^{Q-1}c_{k}y^{Q-k}+f_{Q}}~~~~\mbox{as}~~r\to 0^{+}.
\end{aligned}
\end{equation*}
Because $Q-w-(\alpha_{j}-1)>1$, we have
$ C^{-1}\leq\int_{0}^{\frac{1}{r}}\frac{y^{\alpha_{j}-1}dy}{\sum_{k=w}^{Q-1}c_{k}y^{Q-k}+f_{Q}}\leq C$ for all $0<r<1$ and some positive constant $C>0$.
Therefore $ J_{\Omega}(r)\approx \frac{1}{r^{Q-\alpha_{j}}}$  as $r\to 0^{+}$.
\end{proof}

We now turn to the general case where the vector fields $X$ have more than one degenerate component. In this situation, every non-constant $\lambda_{I}$ is a polynomial that has zeroes in $\Omega$, and
 estimating  the asymptotic behaviour of $J_{\Omega}(r)$ requires more intricate analysis. To handle the integral $J_{\Omega}(r)$, we first invoke the resolution of singularities theory.

It is well-known that Hironaka's celebrated theorem \cite{hironaka1964} on the resolution of singularities is a profound result in algebraic geometry. The theorem was first proved in 1964, and since then, many different variations of it have been developed. The version we state below is due to Atiyah \cite{Atiyah1970}, and it has also been  employed by Watanabe \cite{Wtanabe2009} and Lin  \cite{Lin2011} in the construction singular learning theory.

\begin{proposition}[Resolution of singularities]
\label{prop4-2}
Let $f_1(x),\ldots,f_m(x)$ be a family of non-constant real analytic functions defined on an open set  $V\subset\mathbb{R}^{N}$ containing the origin. Suppose that $f_i(0)=0$ for $i=1,\ldots,m$. Then there exists a triple $(M,W,\rho)$ such that
\begin{enumerate}[(a)]
  \item $W\subset V$ is a neighbourhood of the origin in $\mathbb{R}^{N}$;
  \item $M$ is an $N$-dimensional real analytic manifold;
  \item $\rho:M\to W$ is a real analytic map.
\end{enumerate}
Furthermore, $(M,W,\rho)$  satisfies the following properties:
\begin{enumerate}[(A)]
  \item $\rho$ is proper, i.e. the inverse image of any compact set is compact;
  \item $\rho$ is a real analytic isomorphism between $M\setminus \bigcup_{i=1}^{m} Z(f_i\circ \rho)$ and $W\setminus \bigcup_{i=1}^{m} Z(f_i)$, where
      \[Z(f_{i}\circ\rho)=\{x\in M|f_{i}(\rho(x))=0\}\quad\mbox{and}\quad  Z(f_{i})=\{x\in W|f_{i}(x)=0\};\]
  \item For any $y\in \bigcup_{i=1}^{m} Z(f_i\circ \rho)$, there exists a local chart $U_y$ with the coordinates $u=(u_1, u_2,\ldots,u_N)$ such that $y$ corresponds to the origin and
      \[(f_i\circ\rho)(u)=a_{y,i}(u)u^{p_{y,i}}=a_{y,i}(u_1, u_2,\ldots,u_N)u_{1}^{p_{y}^{1}} u_{2}^{p_{y}^{2}} \cdots u_{N}^{p_{y}^{N}}~~~\mbox{for}~~i=1,\ldots,m, \]
      where $a_{y,i}$ is a non-vanishing analytic function on $U_y$, and $p_{y,i}=(p_{y}^{1},p_{y}^{2},\ldots,p_{y}^{N})$ is an $N$-dimensional multi-index.\footnote{Here, we are indicating  that $\psi_{y}(y)=0$ and $(f_i\circ\rho)(\psi_{y}^{-1}(u))=a_{y,i}(\psi_{y}^{-1}(u))u^{p_{y,i}}$, where $\psi_{y}:U_y\to \mathbb{R}^{N}$ is the coordinate map on the chart $U_y$.
      To simplify notation, we shall omit the coordinate map $\psi_{y}$. } Moreover, in local coordinates, the Jacobian  $J_\rho$ of map $\rho$ has the form
        \[ J_\rho(u)=b_y(u)u^{q_{y}}=b_y(u_1, u_2,\ldots,u_N)u_{1}^{q_{y}^{1}}u_{2}^{q_{y}^{2}}\cdots u_{N}^{q_{y}^{N}},\]
      where $b_y$ is an analytic non-vanishing function on $U_y$, and $q_y=(q_{y}^{1},q_{y}^{2},\ldots,q_{y}^{N})$ is an $N$-dimensional multi-index.
\end{enumerate}
\end{proposition}
\begin{proof}
See \cite{Atiyah1970}, \cite [Theorem 2.8]{Wtanabe2009} and  \cite[Corollary 3.4]{Lin2011}.
\end{proof}

Additionally, we can deduce that
\begin{lemma}
\label{lemma4-2}
The real analytic map $\rho  :M\to W$ given in Proposition \ref{prop4-2} is surjective. In particular, $\rho^{-1}(0)\neq \varnothing$.
\end{lemma}
\begin{proof}
We denote by
$ Z_{W}:=\bigcup_{i=1}^{m} Z(f_i)$ and $ Z_M:=\bigcup_{i=1}^{m} Z(f_i\circ \rho)$, respectively.
Proposition \ref{prop4-2} (B) indicates that $\rho$ is an analytic isomorphism from $M\setminus Z_M$ to $W\setminus Z_W$. As both $f_{i}$ and $f_{i}\circ\rho$ are analytic and non-zero functions for each $1\leq i\leq m$, it follows that $W\setminus Z_W$ and $M\setminus Z_M$ are dense in $W$ and $M$, respectively.

For any $y_{0}\in Z_W$, we can choose a sequence $\{y_{k}\}_{k=1}^{\infty}\subset W\setminus Z_W$ such that $y_{k}\to y_{0}$ as $k\to +\infty$. According to Proposition \ref{prop4-2} (B),  for $\{y_{k}\}_{k=1}^{\infty}\subset W\setminus Z_W$, there exists a sequence $\{x_{k}\}_{k=1}^{\infty}\subset M\setminus Z_M$ such that $\rho(x_{k})=y_{k}$ holds for all $k\geq 1$. Set $K:=\{y_{k}\}_{k=1}^{\infty}\cup\{y_{0}\}\subset W$. Clearly, $K$ is compact, and $\{x_{k}\}_{k=1}^{\infty}\subset\rho^{-1}(K)$ implies that $\rho^{-1}(K)$ is a non-empty compact subset of $M$. Therefore, there is a subsequence $\{x_{k_{j}}\}_{j=1}^{\infty}\subset\{x_{k}\}_{k=1}^{\infty}$ such that $x_{k_{j}}\to x_{0}\in \rho^{-1}(K)\subset M$ as $j\to +\infty$. By to the continuity of $\rho$, we have $\rho(x_{0})=y_{0}$. Hence, $\rho$ is a surjective map and $\rho^{-1}(0)\neq \varnothing$.
\end{proof}

By Proposition \ref{prop4-2} and Lemma \ref{lemma4-2}, we can derive the following conclusion.

\begin{proposition}
\label{prop4-3}
Let $X=(X_{1},X_{2},\ldots,X_{m})$ be homogeneous H\"{o}rmander vector fields defined on $\mathbb{R}^n$, with
 $v$ degenerate components where $2\leq v\leq n-1$. For any bounded domain $\Omega$ in $\mathbb{R}^n$ containing the origin, we have
\begin{equation}\label{4-6}
 J_{\Omega}(r)=\int_{\Omega}\frac{dx}{\Lambda(x,r)}\approx \sum_{j=1}^{l}\int_{(-1,1)^{v}}\frac{|u^{q_{j}}| du}{\sum_{I}|u^{p_{j,I}}|r^{d(I)}},
\end{equation}
where $l$ is a positive integer, $q_{j}$ and $p_{j,I}$ are $v$-dimensional multi-indexes.
\end{proposition}
\begin{proof}
We write the collection of all degenerate components of vector fields $X$ as $\{x_{i_{1}},\ldots,x_{i_{v}}\}$ and substitute $x_{i_{j}}=z_{j}$ for $1\leq j\leq v$. Then,
for any bounded subset $\Omega_{0}\subset\mathbb{R}^{v}$ containing an open neighbourhood of the origin in $\mathbb{R}^{v}$, we can use  Lemma \ref{lemma4-1} to derive
\begin{equation}\label{4-7}
J_{\Omega}(r)\approx J_{\Omega_{0},v}(r):=\int_{\Omega_{0}}\frac{dz}{\Lambda(z,r)},
\end{equation}
where $z=(z_{1},\ldots,z_{v})\in \mathbb{R}^{v}$. We shall choose a suitable set $\Omega_{0}$ to estimate  $J_{\Omega}(r)$.

Note that $\lambda_{I}(z)$ is a polynomial in $\mathbb{R}^{v}$ for each $n$-tuple of integers $I=(i_{1},\ldots,i_{n})$ with $1\leq i_{j}\leq q$. We define the set
\[ \mathcal{B}:=\{I=(i_{1},\ldots,i_{n})| 1\leq i_{j}\leq q,~1\leq j\leq n, ~\lambda_{I}(z)~\mbox{is a non-constant polynomial}\}. \]
It is clear that $\mathcal{B}\neq \varnothing$. For any $I\in \mathcal{B}$, Proposition \ref{prop2-4} yields that $d(I)<Q$  and $\lambda_{I}(0)=0$. Applying  Proposition \ref{prop4-2} to the family of polynomials $\{\lambda_I\}_{I\in \mathcal{B}}$,  we obtain a triple $(M,W,\rho)$ where:
\begin{enumerate}[(a)]
  \item $W\subset V$ is a neighbourhood of the origin in $\mathbb{R}^{v}$;
  \item $M$ is a $v$-dimensional real analytic manifold;
  \item $\rho:M\to W$ is a real analytic map.
\end{enumerate}
Furthermore, $(M,W,\rho)$ satisfies the following properties:
\begin{enumerate}[(A)]
  \item $\rho$ is proper, i.e. the inverse image of any compact set is compact;
  \item $\rho$ is a real analytic isomorphism between $M\setminus \bigcup_{I\in\mathcal{B}} Z(\lambda_{I}\circ \rho)$ and $W\setminus \bigcup_{I\in\mathcal{B}} Z(\lambda_{I})$, where
      \[Z(\lambda_{I}\circ \rho)=\{x\in M|\lambda_{I}(\rho(x))=0\} \quad \mbox{and}\quad Z(\lambda_{I})=\{x\in W|\lambda_{I}(x)=0\};\]
  \item For any $y\in \bigcup_{I\in\mathcal{B}} Z(\lambda_{I}\circ \rho)$, there exists a local chart $U_y$ with the coordinates $u=(u_1, u_2,\ldots,u_v)$ such that $y$ corresponds to the origin and
      \[(\lambda_{I}\circ \rho)(u)=a_{y,I}(u)u^{p_{y,I}}~~~\mbox{for all}~~I\in \mathcal{B}, \]
      where $a_{y,I}$ is a non-vanishing analytic function on $U_y$, and $p_{y,I}$ is a $v$-dimensional multi-index. Moreover, in local coordinates,
       the Jacobian  $J_\rho$ of map $\rho$ has the form
        \[ J_\rho(u)=b_y(u)u^{q_{y}},\]
      where $b_y$ is a non-vanishing analytic function on $U_y$, and $q_y$ is a $v$-dimensional multi-index.
\end{enumerate}

Using $\lambda_{I}(0)=0$ and  Lemma \ref{lemma4-2}, we conclude that $\rho^{-1}(0)\neq \varnothing$ and $\rho^{-1}(0)\subset  \bigcup_{I\in\mathcal{B}} Z(\lambda_{I}\circ \rho)$. For any $y\in \rho^{-1}(0)$, there exists a pre-compact open neighbourhood $U_{y}$ on which property (C) is satisfied, and the analytic functions $a_{y,I}(u)$ and $b_y(u)$ are non-vanishing on $\overline{U_{y}}$. Moreover, we can find a bump function $\Psi_{y}$ such that:
\begin{itemize}
  \item $\Psi_{y}\in C^{\infty}(M)$ and $0\leq \Psi_{y}\leq 1$;
  \item $\Psi_{y}(y)=1$;
  \item $\text{supp}~\Psi_{y}:=\overline{\{z\in M|\Psi_{y}(z)\neq 0\}}\subset U_{y}$.
\end{itemize}
Denoting by $V_{y}:=\{z\in M|\Psi_{y}(z)>0\}$ the open neighborhood of $y$, we have
$V_{y}\subset \text{supp}~\Psi_{y}\subset U_{y}$. Since
$\{V_{y}\}_{y\in \rho^{-1}(0)}$ constitutes an open cover of the compact set $\rho^{-1}(0)$, there exist a finite subcover $\{V_{y_{1}},V_{y_{2}},\ldots, V_{y_{l}}\}$ associated with the bump functions $\{\Psi_{y_{1}},\ldots,\Psi_{y_{l}}\}$ and pre-compact open neighborhoods $\{U_{y_{1}},U_{y_{2}},\ldots, U_{y_{l}}\}$ such that $V_{y_{i}}\subset  \text{supp}~\Psi_{y_{i}}\subset U_{y_{i}}$
for each $1\leq i\leq l$.

Let $\Psi:=\sum_{j=1}^{l}\Psi_{y_{j}}$. It follows that $\Psi\in C^{\infty}(M)$ and
\[ \rho^{-1}(0)\subset V_{0}\subset U_{0}\subset M, \]
where $U_{0}:=\bigcup_{j=1}^{l}U_{y_{j}}$ and $V_{0}:=\bigcup_{j=1}^{l}V_{y_{j}}=\{x\in M|\Psi(x)>0\}$ are pre-compact open subsets of $M$. For each $1\leq i\leq l$, denoting by $\varphi_{i}:=\frac{\Psi_{y_{i}}}{\Psi}$, we have $\varphi_{i}\in C^{\infty}(V_{0})$, $0\leq \varphi_{i}\leq 1$, $\text{supp}~\varphi_{i}\subset \text{supp}~\Psi_{y_{i}}\subset U_{y_{i}}$ and $\sum_{j=1}^{l}\varphi_{j}(y)=1$ for all $y\in V_{0}$.\par

Clearly, $\rho(V_{0})$ is a bounded set in $\mathbb{R}^{v}$ containing the origin. We further show that $\rho(V_{0})$ contains an open neighbourhood of the origin in $\mathbb{R}^{v}$. Suppose that the origin in $\mathbb{R}^{v}$ is not in the interior of $\rho(V_{0})$. Then there exists a sequence $\{y_{k}\}_{k=1}^{+\infty}\subset W\setminus\rho(V_{0})$ such that $y_{k}\to 0\in \mathbb{R}^{v}$ as $k\to +\infty$. It follows from Lemma \ref{lemma4-2} that $\rho: M\to W$ is a surjective map. Therefore, we can find a corresponding sequence $\{x_{k}\}_{k=1}^{+\infty}\subset M\setminus V_{0}$ such that $\rho(x_{k})=y_{k}$ for all $k\geq 1$. Setting $K_{1}:=\{y_{k}\}_{k=1}^{\infty}\cup\{0\}$, we see that $K_{1}$ is a compact subset in $\mathbb{R}^{v}$. This indicates that   $\rho^{-1}(K_{1})$ is a non-empty compact subset in $M$ since $\{x_{k}\}_{k=1}^{+\infty}\subset\rho^{-1}(K_{1})$. Thus, there exists a subsequence  $\{x_{k_{j}}\}_{j=1}^{+\infty}\subset \{x_{k}\}_{k=1}^{+\infty}$ that converges to some point $x_{0}\in \rho^{-1}(K_{1})\subset M$. The continuity of $\rho$ gives $\rho(x_{0})=0$ and $ x_{0}\in \rho^{-1}(0)\subset V_{0}$. Recalling that $V_{0}$ is an open subset in $M$, it follows that $\{x_{k_{j}}\}_{j=j_{0}}^{+\infty}\subset V_{0}$ for some $j_{0}\in \mathbb{N}^{+}$, which contradicts $\{x_{k}\}_{k=1}^{+\infty}\subset M\setminus V_{0}$. Consequently, $\rho(V_{0})$ must contain an open neighbourhood of the origin in $\mathbb{R}^{v}$.\par

To simplify notation,  we denote by $Z_W:=\bigcup_{I\in\mathcal{B}} Z(\lambda_{I})$ and $Z_{M}:=\bigcup_{I\in\mathcal{B}} Z(\lambda_{I}\circ \rho)$. Since for each $I\in \mathcal{B}$, $\lambda_I$ and $\lambda_I\circ\rho$ are non-identically-vanishing analytic functions,  it follows from a result in \cite{Mityagi2020} that $Z_W$ and $Z_M$ have $v$-dimensional zero measure. Thus, the property (B) yields that
\begin{equation}
\begin{aligned}\label{4-8}
  J_{\rho(V_{0}),v}(r)&=\int_{\rho(V_{0})}\frac{dz}{\Lambda(z,r)}
  =\int_{\rho(V_{0})}\frac{d z}{\sum_{I}|\lambda_I(z)|r^{d(I)}}\\
  &=\int_{V_{0}}\frac{|J_{\rho}(y)|dy}{\sum_{I}|\lambda_I(\rho(y))|r^{d(I)}}
  =\sum_{j=1}^{l}\int_{V_{0}}\frac{|J_{\rho}(y)|\varphi_{j}(y)dy}{\sum_{I}|\lambda_I(\rho(y))|r^{d(I)}},
\end{aligned}
\end{equation}
where $dy $ is the volume element in $M$, and $J_{\rho}$ denotes the Jacobian of the transition from $dz$ to $dy$.

In local coordinates, by property (C) we have
\begin{equation}\label{4-9}
\begin{aligned}
\int_{V_{0}}\frac{|J_{\rho}(y)|\varphi_{j}(y)dy}{\sum_{I}|\lambda_I(\rho(y))|r^{d(I)}}&=\int_{ \text{supp}~\varphi_{j}}\frac{|J_{\rho}(y)|\varphi_{j}(y)dy}{\sum_{I}|\lambda_I(\rho(y))|r^{d(I)}}\\
&\leq \int_{U_{y_{j}}}\frac{|J_{\rho}(y)|dy}{\sum_{I}|\lambda_I(\rho(y))|r^{d(I)}}
=\int_{U_{y_{j}}}\frac{|b_{y_{j}}(u)u^{q_{j}}| du}{\sum_{I}|a_{y_{j},I}(u)u^{p_{j,I}}|r^{d(I)}},
  \end{aligned}
\end{equation}
 where $q_{j}$ and $p_{j,I}$ are $v$-dimensional multi-indexes, $b_{y_{j}}(u)$ and $a_{y_{j},I}(u)$ are non-vanishing analytic functions on the compact set $\overline{U_{y_{j}}}$. Hence,
\begin{equation}\label{4-10}
\int_{U_{y_{j}}}\frac{|b_{y_{j}}(u)u^{q_{j}}| du}{\sum_{I}|a_{y_{j},I}(u)u^{p_{j,I}}|r^{d(I)}}\approx \int_{(-1,1)^{v}}\frac{|u^{q_{j}}| du}{\sum_{I}|u^{p_{j,I}}|r^{d(I)}}.
\end{equation}

 On the other hand, since $\varphi_{j}(y_{j})>0$, there exists an open neighbourhood $D_{j}$ of $y_{j}$ such that $\varphi_{j}(y)>\frac{1}{2}\varphi_{j}(y_{j})>0$ for all $y\in D_{j}$. This means $D_{j}\subset \text{supp}~\varphi_{j}\subset U_{y_{j}}$. Therefore,
\begin{equation}\label{4-11}
\begin{aligned}
\int_{V_{0}}\frac{|J_{\rho}(y)|\varphi_{j}(y)dy}{\sum_{I}|\lambda_I(\rho(y))|r^{d(I)}}&=\int_{ \text{supp}~\varphi_{j}}\frac{|J_{\rho}(y)|\varphi_{j}(y)dy}{\sum_{I}|\lambda_I(\rho(y))|r^{d(I)}}\\
&\geq \int_{D_{j}}\frac{|J_{\rho}(y)|\varphi_{j}(y)dy}{\sum_{I}|\lambda_I(\rho(y))|r^{d(I)}}
\geq \frac{1}{2}\varphi_{j}(y_{j})\int_{D_{j}}\frac{|J_{\rho}(y)|dy}{\sum_{I}|\lambda_I(\rho(y))|r^{d(I)}}
\\ &=\frac{1}{2}\varphi_{j}(y_{j})\int_{D_{j}}\frac{|b_{y_{j}}(u)u^{q_{j}}| du}{\sum_{I}|a_{y_{j},I}(u)u^{p_{j,I}}|r^{d(I)}}
\approx \int_{(-1,1)^{v}}\frac{|u^{q_{j}}| du}{\sum_{I}|u^{p_{j,I}}|r^{d(I)}}.
\end{aligned}
\end{equation}
Since $\rho(V_{0})$ is a bounded set containing an open neighbourhood of the origin in $\mathbb{R}^{v}$, by taking $\Omega_{0}=\rho(V_{0})$ we conclude from \eqref{4-7}-\eqref{4-11} that
\begin{equation}
   J_{\Omega}(r)\approx \sum_{j=1}^{l}\int_{(-1,1)^{v}}\frac{|u^{q_{j}}| du}{\sum_{I}|u^{p_{j,I}}|r^{d(I)}}.
\end{equation}
\end{proof}

\subsection{Estimates of integrals of rational functions}
Next, we pay attention to the explicit asymptotic behaviour of the integral
\begin{equation}\label{4-13}
\int_{(-1,1)^{v}}\frac{|u^{q_{j}}| du}{\sum_{I}|u^{p_{j,I}}|r^{d(I)}}~~~~\mbox{as}~~~r\to 0^{+}.
\end{equation}

Let us consider a more general class of integrals, which includes \eqref{4-13}. Suppose that $\mathfrak{G}$ is the collection of finitely many index pairs $(a,s)$, wherein each index pair $(a,s)$, $a=(a_{1},\ldots,a_{N})$ denotes the $N$-dimensional multi-index, and $s$ is a non-negative constant. Given a fixed
$N$-dimensional multi-index $b$, we define
\begin{equation}\label{4-14}
I(r):=\int_{(0,1]^N}\frac{x^{b} dx}{\sum_{(a,s)\in \mathfrak{G}} x^{a}r^{s}}\qquad\mbox{for}~~0<r<1.
\end{equation}
Additionally, we always assume that
\[ I(r)<+\infty\qquad \mbox{for any}~~ 0<r<1. \tag{A} \]
Clearly,  $I(r)$  reduces to the integral in \eqref{4-13} if we choose suitable $v$-dimensional multi-index $b$ and index pairs $(a,s)$.

 By using refined analysis involving the convex geometry, we shall provide the explicit asymptotic behaviour of $I(r)$ as $r\to 0^{+}$.  For this  purpose, we recall some preliminary definitions, notations and results in convex geometry. One can refer to \cite{Rockafellar1970,Brondsted2012,Hug2010} for more details.

Suppose that $N\geq 2$ is a positive integer. Let $\mathbb{R}^{N}$ be the $N$-dimensional Euclidean space equipped with the Euclidean distance
\[ |x-y|_{N}:=\sqrt{\langle x-y,x-y\rangle}, \]
where $\langle x,y\rangle:=\sum_{i=1}^{N}x_iy_i$ for the points $x=(x_{1},\ldots,x_{N})$ and $y=(y_{1},\ldots,y_{N})$ in $\mathbb{R}^N$. Given a subset set $A$ of $\mathbb{R}^{N}$, we denote by ${\rm conv}(A)$ the convex hull of $A$, which is the intersection of all convex sets containing set $A$. Obviously, if $A\subset B\subset \mathbb{R}^{N}$, then ${\rm conv}(A)\subset {\rm conv}(B)$. Moreover, we denote by
\[ A+B:=\{x+y|x\in A,~y\in B\}\]
the Minkowski sum of two sets of position vectors $A$ and $B$ in $\mathbb{R}^{N}$. \par

We then introduce the affine set and the dimension of the convex set. A set $M\subset \mathbb{R}^{N}$ is called the affine set if $tx+(1-t)y\in M$ for any $x,y\in M$ and any $t\in\mathbb{R}$. If an affine set $M\neq \varnothing$, then there is a unique linear subspace $L\subset\mathbb{R}^N$ such that
$M=\{y\}+L$ for any $y\in M$. The dimension of the affine set $M$ is defined by $\dim M:=\dim L$. Furthermore, the intersection of affine sets is also an affine set. For a convex set $A\subset\mathbb{R}^N$, the intersection of all affine sets containing $A$ is called the affine hull of $A$, denoted by $\text{aff}~A$. The dimension of convex $A$ is defined as the dimension of affine hull of $A$, i.e., $\dim A:=\dim (\text{aff}~A)$. \par

 For any convex set $A\subset \mathbb{R}^N$, we denote by
\[\text{relint}(A):=\{x\in \text{aff}~A|\exists  \varepsilon>0,  \overline{B_{\varepsilon}^{N}(x)}\cap \text{aff}~A\subset A\}\]
the relative interior of $A$, where $B_{\varepsilon}^{N}(x):=\{y\in\mathbb{R}^N| |y-x|_{N}<\varepsilon\}$ denotes the Euclidean  ball in $\mathbb{R}^{N}$. It is known that  for any  non-empty convex set $A$, $\text{relint}(A)$ is a non-empty convex set. In addition, $\overline{\text{relint}(A)}=\overline{A}$ and $\text{relint}( \overline{A})=\text{relint}(A)$.  If $\dim A=N$, we have $\text{aff}~A=\mathbb{R}^N$ and $\text{relint}(A)= A^{\circ}$, where $A^{\circ}$ denotes the interior of $A$ in $\mathbb{R}^N$. \par

We shall also examine the volumes of convex sets with different dimensions. For the sake of clarity, we denote by $V_m(A)$  the $m$-dimensional volume of a convex set $A$.  In particular, if $A$ is a non-empty $d$-dimensional convex set with $d\geq 1$, we have  $V_d(A)>0$, while $V_m(A)=0$ provided $m>d$.

As an important class of convex set, the polyhedron in $\mathbb{R}^{N}$ is defined by the intersection of finitely many closed half-spaces. Here, the closed half-space in $\mathbb{R}^{N}$ is the set $\{x\in \mathbb{R}^{N}|f(x)\leq 0\}$ associated with the given affine function $f$. Moreover, a polytope is the convex hull of finitely many points. According to \cite[Corollary 8.7]{Brondsted2012} and \cite[Theorem 1.20]{Hug2010}, we know that a non-empty bounded polyhedron is a polytope and vice versa. \par

For a given polyhedron $P\subset \mathbb{R}^{N}$, if there is an affine function $f$ such that $P\subset \{x\in \mathbb{R}^{N}|f(x)\leq 0\}$ and $F:=\{x\in \mathbb{R}^{N}|f(x)=0\}\cap P\neq\varnothing$, then we call $F$ the face of $P$. We say $F$ is a $d$-face if $\dim F=d$. In particular, the $0$-face is called the vertex, and we will not distinguish between $0$-faces and vertices. From the definitions above, we can deduce that the face of a polyhedron is also a polyhedron. It is worth pointing out that, for an unbounded polyhedron $P\subset \mathbb{R}^{N}$, there exists at least one non-zero vector $q\in \mathbb{R}^{N}$  such that $p+\lambda q\in P$ for all $\lambda\geq 0$ and $p\in P$, where the non-zero vector $q$ is known as the direction of polyhedron $P$ (see \cite[Theorem 8.4]{Rockafellar1970}).\par

To begin our estimation of $I(r)$,  we first show that $I(r)$ is asymptotically  equal to the sum of some integrals on polyhedrons. Precisely,  we have

\begin{proposition}
\label{prop4-4}
For each index pair $(a,s)\in \mathfrak{G}$, let $P_{a,s}$ be a polyhedron in $[0,+\infty)^N$ defined as
\begin{equation}\label{4-15}
P_{a,s}:=\{y\in[0,+\infty)^N|f_{(a,s)(a',s')}(y):=\langle a-a',y\rangle-(s'-s)\leq 0, ~\forall (a',s')\in \mathfrak{G}\}.
\end{equation}
Let
 \begin{equation}\label{4-16}
    \phi_{a}(y):=\langle a-b-\mathbf{1},y\rangle
  \end{equation}
be a linear function, where $\mathbf{1}:=e_{1}+e_{2}+\cdots+e_{N}$ with $e_{j}:=(0,\ldots,\mathop{1}\limits_{j},\ldots,0)$ denoting the unit vector in $\mathbb{R}^{N}$. If $I(r)$ satisfies assumption (A) (i.e., $I(r)<+\infty$ for all $0<r<1$), then there exists a positive constant $0<C<1$  depending only on $\mathfrak{G}$, such that
\begin{equation}\label{4-17}
C\sum_{(a,s)\in\mathfrak{G}^\circ}J_{a,s}(r)\leq I(r)\leq  \sum_{(a,s)\in\mathfrak{G}^\circ}J_{a,s}(r)~~~\mbox{for all}~~r\in (0,1),
\end{equation}
where
\begin{equation}\label{4-18}
  J_{a,s}(r):=\left(\ln\frac{1}{r}\right)^{N}\frac{1}{r^{s}}\int_{P_{a,s}}e^{\left(\ln \frac{1}{r}\right)\phi_{a}(y)}dy,
\end{equation}
 and $\mathfrak{G}^\circ:=\{(a,s)\in\mathfrak{G}|V_{N}(P_{a,s})>0\}$ is the collection of index pairs $(a,s)$ for which the polyhedron $P_{a,s}$ has a positive $N$-dimensional volume.
\end{proposition}

\begin{proof}
Substituting $x_{i}=e^{-y_{i}}$ for $i=1,\ldots,N$, we have
\begin{equation}
\label{4-19}
  I(r)=\int_{(0,1]^N}\frac{x^{b} dx}{\sum_{(a,s)\in \mathfrak{G}} x^{a}r^{s}}=\int_{[0,+\infty)^N}\frac{e^{-\langle b+\mathbf{1},y\rangle }d y}{\sum_{(a,s)\in \mathfrak{G}}e^{-\langle a,y\rangle } r^{s}}.
\end{equation}
Then, for each index pair $(a,s)\in \mathfrak{G}$ and $0<r<1$, we define
\begin{equation}\label{4-20}
  A_{a,s}(r):=\left\{y\in [0,+\infty)^{N}\bigg|\langle a-a',y\rangle \leq (s'-s)\left(\ln\frac{1}{r}\right),~\forall (a',s')\in \mathfrak{G}\right\}.
\end{equation}
Obviously, $A_{a,s}(r)$ is a polyhedron in $[0,+\infty)^{N}$ depending on the index pair $(a,s)$ and the parameter $r$. Note that $A_{a,s}(r)$ can be rewritten as
\[ A_{a,s}(r)=\left\{y\in [0,+\infty)^{N}\bigg|e^{-\langle a, y\rangle}r^{s}=\max_{(a',s')\in \mathfrak{G}}e^{-\langle a', y\rangle}r^{s'} \right\}.\]
The polyhedrons $A_{a,s}(r)$ and $P_{a,s}$ satisfy the following properties:
\begin{enumerate}[(i)]
  \item  For $0<r<1$, let $T_{r}(x):=\left(\ln \frac{1}{r}\right)x$ be the automorphism in $\mathbb{R}^{N}$, then
  \[ T_{r}(P_{a,s})=A_{a,s}(r).\]
    \item For $0<r<1$, we have
    \[\bigcup_{(a,s)\in \mathfrak{G}}{A_{a,s}(r)}=\bigcup_{(a,s)\in\mathfrak{G}}{P_{a,s}}=[0,+\infty)^N.\]
    \item For $0<r<1$, $V_{N}(A_{a_1,s_1}(r)\cap A_{a_2,s_2}(r))=0$ and $V_N(P_{a_1,s_1}\cap P_{a_2,s_2})=0$ if $(a_1,s_1)\neq (a_2,s_2)$.
\end{enumerate}
Indeed, the property (i) follows clearly from the definitions \eqref{4-15} and \eqref{4-20}. Because $A_{a,s}(r)$ is isomorphic to $P_{a,s}$
for any $0<r<1$, we only need to verify properties (ii) and (iii) for either $A_{a,s}(r)$ or $P_{a,s}$. Observe that $\mathfrak{G}$ is the collection of finitely many index pairs. Hence, for any $y_{0}\in [0,+\infty)^{N}$ and $0<r<1$, there exists an index pair $(a,s)\in \mathfrak{G}$ such that
\[  e^{-\langle a, y_{0}\rangle}r^{s}=\max_{(a',s')\in \mathfrak{G}}e^{-\langle a', y_{0}\rangle}r^{s'}, \]
 which implies $y_{0}\in A_{a,s}(r)$ and $\bigcup_{(a,s)\in \mathfrak{G}}{A_{a,s}(r)}=[0,+\infty)^{N}$. Furthermore,  \eqref{4-15} gives that $P_{a_{1},s_{1}}\cap P_{a_{2},s_{2}}\subset \{y\in[0,+\infty)^{N}|f_{(a_1,s_1)(a_2,s_2)}(y)=0\}$
 for $(a_1,s_1)\neq (a_2,s_2)$. Since $\{y\in[0,+\infty)^{N}|f_{(a_1,s_1)(a_2,s_2)}(y)=0\}$ is a hyperplane with dimension at most $N-1$, it follows that $V_{N}(P_{a_1,s_1}\cap P_{a_2,s_2})=0$.

Owing to properties (i)-(iii), we have
\begin{equation}\label{4-21}
 I(r)=\int_{[0,+\infty)^N}\frac{e^{-\langle b+\mathbf{1},y\rangle }d y}{\sum_{(a',s')\in \mathfrak{G}}e^{-\langle a',y\rangle } r^{s'}}=\sum_{(a,s)\in\mathfrak{G}^\circ}\int_{A_{a,s}(r)}\frac{e^{-\langle b+\mathbf{1},y\rangle }d y}{\sum_{(a',s')\in \mathfrak{G}}e^{-\langle a',y\rangle } r^{s'}},
\end{equation}
where $\mathfrak{G}^\circ:=\{(a,s)\in\mathfrak{G}|V_{N}(P_{a,s})>0\}$. In addition, for each $(a,s)\in \mathfrak{G}^\circ$,
\begin{equation}\label{4-22}
\frac{C}{r^{s}}\int_{A_{a,s}(r)}e^{\langle a-b-\mathbf{1},y\rangle}d y\leq\int_{A_{a,s}(r)}\frac{e^{-\langle b+\mathbf{1},y\rangle }d y}{\sum_{(a',s')\in \mathfrak{G}}e^{-\langle a',y\rangle } r^{s'}}\leq\frac{1}{r^{s}}\int_{A_{a,s}(r)}e^{\langle a-b-\mathbf{1},y\rangle}d y,
\end{equation}
where  $0<C<1$ is a positive constant that depends only on $\mathfrak{G}$. Setting $\phi_{a}(y):=\langle a-b-\mathbf{1},y\rangle$, it follows from \eqref{4-22} and property (i) that for any $0<r<1$,
\begin{equation}\label{4-23}
\frac{C}{r^{s}}\int_{T_{r}(P_{a,s})}e^{\phi_{a}(y)}d y\leq\int_{A_{a,s}(r)}\frac{e^{-\langle b+\mathbf{1},y\rangle }d y}{\sum_{(a',s')\in \mathfrak{G}}e^{-\langle a',y\rangle } r^{s'}}\leq\frac{1}{r^{s}}\int_{T_{r}(P_{a,s})}e^{\phi_{a}(y)}d y.
\end{equation}
Consequently, by \eqref{4-21} and \eqref{4-23} we obtain
\begin{equation}
C\sum_{(a,s)\in\mathfrak{G}^\circ}\frac{1}{r^{s}}\left(\ln \frac{1}{r} \right)^{N}\int_{P_{a,s}}e^{\left(\ln \frac{1}{r} \right)\phi_{a}(y)}d y\leq I(r)
\leq \sum_{(a,s)\in\mathfrak{G}^\circ}\frac{1}{r^{s}}\left(\ln \frac{1}{r} \right)^{N}\int_{P_{a,s}}e^{\left(\ln \frac{1}{r} \right)\phi_{a}(y)}dy
\end{equation}
holds for all $0<r<1$.
\end{proof}

We now investigate the asymptotic behaviour of  $J_{a,s}(r)$ as $r\to 0^{+}$.  It follows from \eqref{4-17} that the assumption (A) is equivalent to
\begin{equation}
\label{4-25}
  J_{a,s}(r)=\left(\ln\frac{1}{r}\right)^{N}\frac{1}{r^{s}}\int_{P_{a,s}}e^{\left(\ln \frac{1}{r}\right)\phi_{a}(y)}dy<+\infty~~~~\mbox{for all}~~(a,s)\in \mathfrak{G}^{\circ},~0<r<1.
\end{equation}
 For the index pair $(a,s)\in \mathfrak{G}^\circ$, if $a-b-\mathbf{1}=0$, then $\phi_{a}(y)\equiv 0$. In this case,  $V_{N}(P_{a,s})<+\infty$ and $P_{a,s}$ is bounded. As a result,
\begin{equation}
\label{4-26}
  J_{a,s}(r)=\left(\ln\frac{1}{r}\right)^{N}\frac{1}{r^{s}}V_{N}(P_{a,s})~~\mbox{for}~~a-b-\mathbf{1}=0,~0<r<1.
\end{equation}

If $a-b-\mathbf{1}\neq 0$, we will derive the asymptotic estimates of $J_{a,s}(r)$ by using several results as follows.

\begin{lemma}
\label{lemma4-3}
For any fixed multi-index $b$ and index pair $(a,s)\in \mathfrak{G}^{\circ}$ with $a-b-\mathbf{1}\neq 0$, the linear function $\phi_{a}(x)=\langle a-b-\mathbf{1},x\rangle$ attains its maximum value on the  polyhedron $P_{a,s}$. Furthermore, the set $M_{a,s}:=\{x\in P_{a,s}|\phi_{a}(x)=m_{a,s}\}$ is a bounded face of polyhedron $P_{a,s}$ with dimension $0\leq d_{a,s}\leq N-1$, where $m_{a,s}:=\max_{x\in P_{a,s}}\phi_{a}(x)$.
\end{lemma}
\begin{proof}
According to the theory of linear programming (see \cite[p. 15]{Grotschel2012}), if $\phi_{a}$ attains its maximum value on the polyhedron $P_{a,s}$, then $M_{a,s}$ is a $d_{a,s}$-face of $P_{a,s}$ with $0\leq d_{a,s}\leq  N-1$. Therefore, the remaining task is to prove that the maximum of $\phi_{a}$ is achieved, and that $M_{a,s}$ is bounded. \par

Since the polyhedron $P_{a,s}\subset [0,+\infty)^{N}$ cannot contain any line in $\mathbb{R}^{N}$, we can infer from \cite[Theorem 2.6]{Bertsimas1997} that $P_{a,s}$ has at least one extreme point (i.e., vertex, see \cite[Theorem 2.3]{Bertsimas1997}). If $\phi_{a}$ cannot attain its maximum value on $P_{a,s}$, then  \cite[Theorem 2.8]{Bertsimas1997} gives that $\sup_{x\in P_{a,s}}\phi_{a}(x)=+\infty$, which further implies that $P_{a,s}$ is an unbounded closed set in $\mathbb{R}^N$.
 Additionally, \cite[Corollary 2.5]{Bertsimas1997} indicates that $\phi_{a}(P_{a,s})$ is also a polyhedron in $\mathbb{R}$, and thus $[c,+\infty)\subset \phi_{a}(P_{a,s})$ for some  constant $c>0$.\par

  Consider the affine function $g_{\varphi}(y):=\phi_{a}(y)-\varphi$ for some $\varphi>0$. For any $\varphi>c$, there is a point $y_{0}\in P_{a,s}$ such that $g_{\varphi}(y_{0})>0$. Observing that $P_{a,s}$ is an $N$-dimensional polyhedron with $V_N(P_{a,s})>0$, we have $\text{relint}(P_{a,s})=P_{a,s}^{\circ}\neq \varnothing$, which allows us to choose a point $\widetilde{y_{0}}$ in the neighbourhood of $y_{0}$ such that $\widetilde{y_{0}}\in P_{a,s}^{\circ}$ and $g_{\varphi}(\widetilde{y_{0}})>0$. Hence, for any  $\varphi>c$,  the open set  $D_{\varphi}$ given by $ D_{\varphi}:=\{y\in \mathbb{R}^{N}|g_{\varphi}(y)> 0\}\cap P_{a,s}^{\circ}$ is non-empty and satisfies $V_{N}(D_{\varphi})>0$.

Then, we further show that
\begin{equation}\label{4-27}
\lim_{\varphi\to+\infty}V_{N}(D_{\varphi})=+\infty.
\end{equation}
Assume there exists a sequence $\{\varphi_{k}\}_{k\geq 1}$ such that $c<\varphi_{k}<\varphi_{k+1}$, $\varphi_{k}\to +\infty$ as $k\to +\infty$, and $0<V_{N}(D_{\varphi_{k}})\leq c_{1}$ for some $c_{1}>0$ and all $k\geq 1$. Since $\overline{D_{\varphi_{k}}}$ is a polyhedron possessing non-empty interior $D_{\varphi_{k}}$ and  satisfying $0<V_{N}(\overline{D_{\varphi_{k}}})\leq c_{1}$,  we can deduce that  $\overline{D_{\varphi_k}}$ is the polytope for each $k\geq 1$. Thus, $\phi_{a}$ attains its maximum value $m_k:=\sup_{x\in\overline{D_{\varphi_{k}}}}\phi_{a}(x)$ on the compact set $\overline{D_{\varphi_k}}$. Clearly, $D_{\varphi_{k+1}}\subset D_{\varphi_{k}} $ and $m_{k+1}\leq m_{k}$ for $k\geq 1$. Because $\varphi_{k}\to +\infty$ as $k\to +\infty$, we obtain $\varphi_{k_{0}}>m_{1}$ and $D_{\varphi_{k_{0}}}\neq \varnothing$ for some integer $k_{0}\geq 1$. For any $y\in D_{\varphi_{k_{0}}}$, we have $\phi_{a}(y)>\varphi_{k_{0}}$, which contradicts $ \phi_{a}(y)\leq m_{k_{0}}\leq m_{1}<\varphi_{k_{0}}$. As a result,  $\lim_{\varphi\to+\infty}V_{N}(D_{\varphi})=+\infty$.

According to \eqref{4-27}, if $\phi_{a}$ cannot attain its maximum value on the polyhedron $P_{a,s}$, then for any $r\in (0,1)$, we have
\[
   \int_{P_{a,s}}e^{\left(\ln\frac{1}{r}\right)\phi_{a}(y)}d y\geq\int_{D_\varphi}e^{\left(\ln\frac{1}{r}\right)\phi_{a}(y)} dy
   \geq e^{\left(\ln\frac{1}{r}\right)\varphi}V_{N}(D_\varphi)\to +\infty\]
 as $\varphi\to+\infty$. This leads a contradiction to \eqref{4-25}. Hence, the linear function $\phi_{a}$ must attain its maximum value $m_{a,s}:=\max_{x\in P_{a,s}}\phi_{a}(x)$ on the polyhedron $P_{a,s}$.

The set $M_{a,s}=\{x\in P_{a,s}|\phi_{a}(x)=m_{a,s}\}$ is a face of polyhedron $P_{a,s}$, which is also a polyhedron in $\mathbb{R}^{N}$. If $M_{a,s}$ is an unbounded polyhedron, there exists a direction $q\in \mathbb{R}^{N}$ such that $\{p+tq|t\geq 0\}\subset M_{a,s}\subset P_{a,s}$ for any $t\geq 0$ and any $p\in M_{a,s}$.  By \cite[Theorem 8.3]{Rockafellar1970}, we see that $q$ is also the direction of polyhedron $P_{a,s}$. For any $x_{0}\in M_{a,s}$, we have $\phi_{a}(x_{0}+tq)=\phi_{a}(x_{0})=m_{a,s}$ for all $t\geq 0$, which implies $\phi_{a}(q)=0$. Recalling that  $a-b-\mathbf{1}\neq 0$, we let $B_{R}^{N}(x_{0}+tq)=\{y\in \mathbb{R}^{N}||y-x_{0}-tq|_{N}<R\}$ be the $N$-dimensional ball whose centre is $x_{0}+tq$ with radius $R=\frac{|m_{a,s}|+1}{\sum_{j=1}^{N}|a_{j}-b_{j}-1|}>0$. It follows that $B_{R}^{N}(x_{0})\cap P_{a,s}^{\circ}\neq \varnothing$ and
\begin{equation}
\label{4-28}
  \phi_{a}(y)\geq m_{a,s}-|m_{a,s}|-1~~~~\mbox{for all}~~y\in B_{R}^{N}(x_{0})\cap P_{a,s}.
\end{equation}
 Meanwhile, for any $y\in B_{R}^{N}(x_{0})\cap P_{a,s}$, we have $y+tq\in P_{a,s}$ and $y+tq\in B_{R}^{N}(x_{0}+tq)$ for any $t\geq 0$. That means
$ B_{R}^{N}(x_{0})\cap P_{a,s}+\{tq\}\subset B_{R}^{N}(x_{0}+tq)\cap P_{a,s}$
and
\[ V_{N}(B_{R}^{N}(x_{0}+tq)\cap P_{a,s})\geq V_{N}(B_{R}^{N}(x_{0})\cap P_{a,s})>0\]
holds for any $t\geq 0$. Denoting by $\Gamma_{R}:=\bigcup_{t\geq 0} B_{R}^{N}(x_{0}+tq)$, we can verify that $\Gamma_{R}$ is a convex set and $V_{N}\left(\Gamma_{R}\cap P_{a,s}\right)=+\infty$. Furthermore, by using $\phi_{a}(q)=0$ and \eqref{4-28}, we have
\[ \phi_{a}(y)\geq m_{a,s}-|m_{a,s}|-1~~~\mbox{for all}~~y\in \Gamma_{R}\cap P_{a,s}.   \]
Thus, we get
\[ \int_{P_{a,s}}e^{\left(\ln\frac{1}{r}\right)\phi_{a}(y)}d y\geq\int_{\Gamma_{R}\cap P_{a,s}}e^{\left(\ln\frac{1}{r}\right)\phi_{a}(y)} dy
   \geq e^{\left(\ln\frac{1}{r}\right)\cdot (m_{a,s}-|m_{a,s}|-1)}V_{N}\left(\Gamma_{R}\cap P_{a,s}\right)= +\infty\]
holds for any $0<r<1$, which also  contradicts \eqref{4-25}. Consequently, $M_{a,s}$ is a bounded $d_{a,s}$-face of $P_{a,s}$ with $0\leq d_{a,s}\leq N-1$.
\end{proof}

 Lemma \ref{lemma4-3} derives the following result of $J_{a,s}(r)$ for $a-b-\mathbf{1}\neq 0$.

\begin{proposition}
\label{prop4-5}
Suppose that $(a,s)\in \mathfrak{G}^{\circ}$ with $a-b-\mathbf{1}\neq 0$. For the integral $J_{a,s}(r)$ defined in \eqref{4-18} with $0<r<1$, we have
\begin{equation}\label{4-29}
J_{a,s}(r)=\left(\ln\frac{1}{r}\right)^{N}\frac{C_{0}}{r^{s+m_{a,s}}}
    \int_{0}^{+\infty}e^{-\left(\ln\frac{1}{r}\right)x_{N}}\left(\int_{\mathbb{R}^{N-1}}\mathbf{1}_{\mathcal{H}(P_{a,s})}(x)d x_1\cdots d x_{N-1} \right) d x_N,
\end{equation}
where $C_{0}>0$ is a positive constant depending only on the multi-indexes $a$ and $b$, $\mathcal{H}:\mathbb{R}^{N}\to \mathbb{R}^{N}$ is a non-degenerate affine transform such that $\mathcal{H}(M_{a,s})=\mathcal{H}(P_{a,s})\cap \{x=(x_{1},\ldots,x_{N})\in\mathbb{R}^{N}|x_N=0\}$ and $\mathcal{H}(P_{a,s})\subset \{x=(x_{1},\ldots,x_{N})\in\mathbb{R}^{N}|x_N\geq 0\}$. Furthermore,  the dimensions of $M_{a,s}$ and $P_{a,s}$ are invariant under the non-degenerate affine transform $\mathcal{H}$.

\end{proposition}

\begin{proof}
Without loss of generality, we may suppose that $a_{N}-b_{N}-1\neq 0$. By Lemma \ref{lemma4-3}, the function $\phi_{a}$ attains its maximum value $m_{a,s}$ on the bounded face $M_{a,s}$ of the polyhedron $P_{a,s}$. Then for any $x=(x_{1},\ldots,x_{N})\in \mathbb{R}^N$,  we define
\[ \mathcal{H}(x):=(x_{1},\ldots,x_{N-1}, m_{a,s}-\phi_{a}(x)). \]
Clearly, $\mathcal{H}:\mathbb{R}^{N}\to \mathbb{R}^{N}$ is a non-degenerate affine transform such that $\mathcal{H}(M_{a,s})=\mathcal{H}(P_{a,s})\cap \{x=(x_{1},\ldots,x_{N})\in\mathbb{R}^{N}|x_N=0\}$ and $\mathcal{H}(P_{a,s})\subset \{x=(x_{1},\ldots,x_{N})\in\mathbb{R}^{N}|x_N\geq 0\}$. According to \cite[Theorem 3.1]{Barvinok2008},  $\mathcal{H}(M_{a,s})$ and $\mathcal{H}(P_{a,s})$ are polyhedrons such that $\dim M_{a,s}=\dim\mathcal{H}(M_{a,s})$ and $\dim P_{a,s}=\dim\mathcal{H}(P_{a,s})$.
Thus, the Fubini's theorem derives
\begin{equation*}
\begin{aligned}
J_{a,s}(r)&=\frac{\left(\ln\frac{1}{r}\right)^{N}}{r^{s}}\int_{P_{a,s}}e^{\left(\ln \frac{1}{r}\right)\phi_{a}(y)}dy=\frac{\left(\ln\frac{1}{r}\right)^{N}}{r^{s}|1+b_{N}-a_{N}|}\int_{\mathcal{H}(P_{a,s})}e^{-\left(\ln\frac{1}{r}\right)(x_N-m_{a,s})}dx \\
&=\frac{\left(\ln\frac{1}{r}\right)^{N}}{r^{s}|1+b_{N}-a_{N}|}\int_{\mathbb{R}^{N}}\mathbf{1}_{\mathcal{H}(P_{a,s})}(x)e^{-\left(\ln\frac{1}{r}\right)(x_N-m_{a,s})}dx \\
    &=\frac{\left(\ln\frac{1}{r}\right)^{N}}{r^{s+m_{a,s}}|1+b_{N}-a_{N}|}\int_{0}^{+\infty}e^{-\left(\ln\frac{1}{r}\right)x_N}\left( \int_{\mathbb{R}^{N-1}}\mathbf{1}_{\mathcal{H}(P_{a,s})}(x)dx_{1}\cdots d x_{N-1}\right)dx_{N}.
\end{aligned}
\end{equation*}
\end{proof}

According to Proposition \ref{prop4-5}, when $a-b-\mathbf{1}\neq 0$, the asymptotic behaviour of $J_{a,s}(r)$ can be achieved by estimating
\begin{equation}\label{4-30}
 S(x_{N}):=\int_{\mathbb{R}^{N-1}}\mathbf{1}_{\mathcal{H}(P_{a,s})}(x)dx_{1}\cdots d x_{N-1}.
\end{equation}
To proceed with our estimation, we decompose $\mathbb{R}^{N}$ as $\mathbb{R}^{N}=E\oplus E^{\perp}$, where
\[E:={\rm span}\{e_{j}|1\leq j\leq N-1\}. \]
We also set $E_{+}:=\{x=(x_{1},\ldots,x_{N})\in \mathbb{R}^{N}|x_{N}>0\}$. Let $E_{d}$ be a
$d$-dimensional linear subspace of $E$ with $1\leq d\leq N-1$. For any $x\in E_{d}$ and $r>0$, we denote by
\[ B_{r}^{d}(x):=\{y\in E_{d}||y-x|_{N} <r\} \]
the $d$-dimensional ball in $E_{d}$. Especially, if $d=0$, we define $E_{0}=\{0\}$ and in this case, the $0$-dimensional ball is $\{0\}$, which is denoted by $B_{r}^{0}(0)$. Additionally, $B_r^N(x):=\{y\in\mathbb{R}^N||y-x|_{N} <r\}$ denotes the $N$-dimensional Euclidean ball in $\mathbb{R}^{N}$.\par

Consider the quadruple $(E_{d},x_{0},u,\varepsilon)$, where $E_d\subset E$ is a $d$-dimensional linear subspace of $E$, $\varepsilon>0$ is a positive constant, and $x_{0}\in E, u\in \mathbb{S}^{N-1}\cap E_+$ are some points. Then for any $r>0$,  we define the corresponding set that varies along $t\geq 0$ as follows:
\begin{equation}\label{4-31}
  F(t):=\overline{B_{r}^{d}(0)}+\overline{B_{t\varepsilon}^{N-1}(0)}+\{x_{0}+t\eta u\},
\end{equation}
where $\eta=\langle u,e_N\rangle^{-1}$. The definition \eqref{4-31} implies the following properties:
\begin{itemize}
  \item For all $t\geq 0$, $F(t)\subset \{x=(x_{1},\ldots,x_{N})\in \mathbb{R}^{N}|x_{N}=t\}$.
  \item For $t_{1}\neq t_{2}$, $F(t_{1})\cap F(t_{2})=\varnothing$.
  \item $F(t)$ is convex and $V_{N-1}(F(t))>0$ for all $t>0$.
\end{itemize}

 Furthermore, we have
\begin{lemma}
\label{lemma4-4}
For any $t\geq 0$, the set $F(t)$ defined in \eqref{4-31} satisfying
  \begin{equation}
  \label{4-32}
V_{N-1}(F(t))=c_{d}(\varepsilon t)^{N-1-d}+c_{d-1}(\varepsilon t)^{N-d}+\cdots+c_0(\varepsilon t)^{N-1},
  \end{equation}
  where $c_{j}>0$ for $j=0,\ldots,d$.
\end{lemma}
\begin{proof}
The transform invariance of Lebesgue measure gives that
\begin{equation}\label{4-33}
V_{N-1}(F(t))=V_{N-1}\left(\overline{B_{r}^{d}(0)}+\overline{B_{t\varepsilon}^{N-1}(0)}\right).
\end{equation}
 By the Steiner's formula (see \cite[Theorem 3.10]{Hug2010}), we get
  \begin{equation}\label{4-34}
    V_{N-1}\left(\overline{B_{r}^{d}(0)}+\overline{B_{t\varepsilon}^{N-1}(0)}\right)=\sum_{i=0}^{N-1}\frac{(N-1)!(\varepsilon t)^{i}}{(N-1-i)!i!}V(\underbrace{\overline{B_{r}^{d}(0)},\ldots,\overline{B_{r}^{d}(0)}}_{N-1-i},\underbrace{\overline{B_{1}^{N-1}(0)},\ldots,\overline{B_{1}^{N-1}(0)}}_{i}),
  \end{equation}
  where \[ V(\underbrace{\overline{B_{r}^{d}(0)},\ldots,\overline{B_{r}^{d}(0)}}_{N-1-i},\underbrace{\overline{B_{1}^{N-1}(0)},\ldots,\overline{B_{1}^{N-1}(0)}}_{i}),\]
abbreviated as $V^{i}$, is called the quermassintegral or $(N-1)$-dimensional mixed volume. According to \cite[Remark 3.15]{Hug2010}, we deduce that if $N-1-i>d$, then $V^{i}=0$, while $V^{i}>0$ if $N-1-i\leq d$, since there are at most $d$ segments in $\overline{B_{r}^{d}(0)}$ with linearly independent directions. Hence, we conclude from \eqref{4-33} and \eqref{4-34} that
\begin{equation*}
V_{N-1}(F(t))=c_{d}(\varepsilon t)^{N-1-d}+c_{d-1}(\varepsilon t)^{N-d}+\cdots+c_0(\varepsilon t)^{N-1},
\end{equation*}
 where $c_{j}>0$ for $j=0,\ldots,d$.
\end{proof}

  Then, for any given $\delta>0$, we define a tube by
\begin{equation}\label{4-35}
  T(\delta):=\bigcup_{0\leq t\leq \delta}F(t).
\end{equation}
\begin{lemma}
\label{lemma4-5}
For any $\delta>0$, $T(\delta)={\rm conv}(F(\delta)\cup F(0))$.
\end{lemma}
\begin{proof}
For any $x\in T(\delta)$, by \eqref{4-31} and \eqref{4-35} we know there exists a $t\in [0,\delta]$ such that
$$x=a+t\varepsilon b+x_{0}+t\eta u\in F(t),$$ where $a\in \overline{B_{r}^{d}(0)}$ and $b\in \overline{B_{1}^{N-1}(0)}$. Since $0\leq t\leq \delta$, we have
\[ x=\left(1-\frac{t}{\delta}\right)(a+x_{0})+\frac{t}{\delta}(a+\delta \varepsilon b+x_{0}+\delta \eta u)\in {\rm conv}(F(\delta)\cup F(0)),\]
which gives $T(\delta)\subset {\rm conv}(F(\delta)\cup F(0))$.

We next prove that $T(\delta)$ is convex. For $x_1,x_2\in T(\delta)$, there exist $t_1,t_2\in[0,\delta]$, $a_1,a_2\in \overline{B_{r}^{d}(0)}$ and $b_1, b_2\in \overline{B_{1}^{N-1}(0)}$ such that
$x_i=a_i+t_i\varepsilon b_i+x_{0}+t_i\eta u\in F(t_{i})$ for $i=1,2$.
 Thus, for any $\lambda\in [0,1]$,
\[ \lambda x_1+(1-\lambda) x_2 = [\lambda a_1+(1-\lambda) a_2]+[\lambda t_1 \varepsilon b_1+(1-\lambda) t_2 \varepsilon b_2]+x_{0}+
  [\lambda t_1+(1-\lambda) t_2]\eta u.\]
Since
$\lambda a_1+(1-\lambda) a_2\in \overline{B_{r}^{d}(0)}$ and $\lambda t_1 \varepsilon b_1+(1-\lambda) t_2 \varepsilon b_2\in \overline{B_{(\lambda t_1+(1-\lambda) t_2)\varepsilon}^{N-1}(0)}$, we deduce that
 $\lambda x_1+(1-\lambda) x_2\in F(\lambda t_1+(1-\lambda) t_2)\subset T(\delta)$, which implies that $T(\delta)$ is convex.

 Consequently,
${\rm conv}(F(\delta)\cup F(0))\subset {\rm conv}(T(\delta))=T(\delta)\subset {\rm conv}(F(\delta)\cup F(0))$.
\end{proof}

On the other hand, the plane section of polyhedron in $\mathbb{R}^{N}$ satisfies the following property.
\begin{lemma}
\label{lemma4-6}
Suppose that $P\subset\mathbb{R}^N$ is a polyhedron with $V_N(P)>0$, and $f(x)$ is an affine function. If the plane section $Q:=\{x\in\mathbb{R}^N|f(x)=0\}\cap P\neq \varnothing$ is not a face of $P$, then $V_{N-1}(Q)>0$.
\end{lemma}
\begin{proof}
Since $V_N(P)>0$, we know that $P$ has a non-empty $N$-dimensional relative interior $\text{relint}(P)$ satisfying $\text{relint}(P)=P^{\circ}$. Hence, $V_{N-1}(Q)>0$ amounts to proving that $\{x\in\mathbb{R}^N|f(x)=0\}\cap \text{relint}(P)\neq \varnothing$.

Suppose that $\{x\in\mathbb{R}^N|f(x)=0\}\cap \text{relint}(P)= \varnothing$. Because $f$ is continuous and $\text{relint}(P)$ is connected, it follows that either $P\subset \{x\in \mathbb{R}^{N}|f(x)\leq 0\}$ or  $P\subset \{x\in \mathbb{R}^{N}|f(x)\geq 0\}$. In either case, the plane section $Q=\{x\in\mathbb{R}^N|f(x)=0\}\cap P$  must be a face of $P$, which contradicts the assumption that $Q$ is not a face. Therefore, we have $V_{N-1}(Q)>0$.
\end{proof}

To estimate the integral $S(x_N)$ in \eqref{4-30}, we introduce the $(N-1)$-dimensional volume of the section of a set. Specifically, let $M$ be a subset of $\mathbb{R}^N$, we define
\begin{equation}\label{4-36}
  X_{M}(t):=\int_{\mathbb{R}^{N-1}}\mathbf{1}_M(y_{1},\ldots,y_{N-1},t)d y_1\cdots d y_{N-1}~~\mbox{for}~~t\geq 0.
\end{equation}
Combining \eqref{4-30} and \eqref{4-36}, we have
\begin{equation}\label{4-37}
  S(t)=\int_{\mathbb{R}^{N-1}}\mathbf{1}_{\mathcal{H}(P_{a,s})}(x_{1},\ldots,x_{N-1},t)dx_{1}\cdots d x_{N-1}=X_{\mathcal{H}(P_{a,s})}(t).
\end{equation}
Recalling that $F(t_{1})\cap F(t_{2})=\varnothing$ if $t_{1}\neq t_{2}$, we can derive from  \eqref{4-35} and \eqref{4-36} that
\begin{equation}
\begin{aligned}\label{4-38}
X_{T(\delta)}(t)&=\int_{\mathbb{R}^{N-1}}\mathbf{1}_{T(\delta)}(y_{1},\ldots,y_{N-1},t)d y_1\cdots d y_{N-1}\\
&=\int_{\mathbb{R}^{N-1}}\mathbf{1}_{F(t)}(y_{1},\ldots,y_{N-1},t)d y_1\cdots d y_{N-1}=V_{N-1}(F(t)).
\end{aligned}
\end{equation}

Based on Lemma \ref{lemma4-4}-Lemma \ref{lemma4-6}, we can establish the following two technical lemmas concerning the estimation of $S(t)$.

\begin{lemma}
\label{lemma4-7}
For $(a,s)\in \mathfrak{G}^{\circ}$ with $a-b-\mathbf{1}\neq 0$, let $\mathcal{H}(P_{a,s})$ and $\mathcal{H}(M_{a,s})$ be the sets given in Proposition \ref{prop4-5}. Denote by $d_{a,s}:=\dim M_{a,s}$ the dimension of $M_{a,s}$. Then, there exists a $\delta>0$ such that
\[ p_{2}(t)\leq S(t)\leq p_{1}(t)\qquad \mbox{for all}~~0\leq t\leq \delta, \]
where $p_{i}(t)=v_{i,d_{a,s}}t^{N-1-d_{a,s}}+\cdots+v_{i,0}t^{N-1}$ is the polynomial with positive coefficients $v_{i,j}>0$ for $i=1,2$ and $j=0,\ldots,d_{a,s}$.
\end{lemma}
\begin{proof}
By Lemma \ref{lemma4-3} and Proposition \ref{prop4-5}, we can deduce that $\mathcal{H}(M_{a,s})$ is a bounded $d_{a,s}$-face of the polyhedron $\mathcal{H}(P_{a,s})$, implying that $\mathcal{H}(M_{a,s})$ is a polytope.
For any $t\geq 0$, let  $Q(t):=\mathcal{H}(P_{a,s})\cap \{x\in\mathbb{R}^{N}|x_N=t\}$ be the plane section of $\mathcal{H}(P_{a,s})$. Then, we find a $\delta>0$ such that the plane section $Q(\delta)$ has a positive $(N-1)$-dimensional volume.

Note that the polyhedron $\mathcal{H}(P_{a,s})\subset \overline{E_{+}}$ has a finite number of  vertexes in $\overline{E_{+}}$ (see \cite[Corollary 2.1]{Bertsimas1997}). In the case where $\mathcal{H}(P_{a,s})$ has $m$ vertexes in $E_{+}$, we denote the orthogonal projections of these vertexes onto $E^\bot$ as
 $0<z_1\leq z_2\leq\cdots\leq z_m$. We then choose a $\delta\in(0,z_1)$ and let $g(x):=x_{N}-\delta$. Obviously, the non-empty set $\mathcal{H}(P_{a,s})\cap\{x\in \mathbb{R}^{N}|g(x)=0\}=Q(\delta)$ is not the face of $\mathcal{H}(P_{a,s})$ in this case, since $g(x)<0$ for $x\in \mathcal{H}(M_{a,s})\subset \mathcal{H}(P_{a,s})$ and $g(y)>0$ for $y:=(y_1,\ldots,y_{N-1},z_1)$ being the vertex of $\mathcal{H}(P_{a,s})$. On the other hand, if all the vertexes of $\mathcal{H}(P_{a,s})$ lie in $E$, we can find a point $w=(w_1,\ldots,w_{N-1},w_N)\in\text{relint}(\mathcal{H}(P_{a,s}))\subset E_{+}$ with $w_N>0$ due to $V_N(\mathcal{H}(P_{a,s}))>0$. By taking $\delta\in(0,w_N)$ and using the similar arguments as above, we deduce that the non-empty plane section $Q(\delta)$ is not a face of $\mathcal{H}(P_{a,s})$. Hence in both cases, Lemma \ref{lemma4-6} indicates that  $Q(\delta)$ admits a positive $(N-1)$-dimensional volume.

The boundedness of $\mathcal{H}(M_{a,s})$ and \cite[Corollary 8.4.1]{Rockafellar1970} imply that the plane section $Q(t)$ is bounded for any $t\in[0,\delta]$. Therefore
\[ P(\delta):=\mathcal{H}(P_{a,s})\cap\{x\in \mathbb{R}^{N}|x_{N}\leq \delta\}=\bigcup_{0\leq t\leq \delta}Q(t) \]
is a bounded polyhedron (i.e., polytope) with its vertexes lying in $Q(0)\cup Q(\delta)$, because $\mathcal{H}(P_{a,s})$ has no vertexes in $\{x\in \mathbb{R}^{N}|0<x_{N}<\delta\}$ by construction. From \cite[Theorem 1.21]{Hug2010} and \cite[Theorem 2.3]{Bertsimas1997}, we have $P(\delta)={\rm conv}(Q(0)\cup Q(\delta))$. Moreover, $Q(\delta)$ is an $(N-1)$-face of $P(\delta)$ with $V_{N-1}(Q(\delta))>0$.

Since  $\mathcal{H}(M_{a,s})$ has a non-empty $d_{a,s}$-dimensional relative interior $\text{relint}(\mathcal{H}(M_{a,s}))$, for any $x_{0}\in \text{relint}(\mathcal{H}(M_{a,s}))$,  there exists a unique $d_{a,s}$-dimensional linear subspace $E_{d_{a,s}}\subset E$ and a $d_{a,s}$-dimensional ball $B_{r_{1}}^{d_{a,s}}(0)\subset E_{d_{a,s}}$  such that
\[ \mathcal{H}(M_{a,s})\subset \{x_{0}\}+\overline{B_{r_{1}}^{d_{a,s}}(0)}.\]
 In particular, in the case of $d_{a,s}=0$, we have $\mathcal{H}(M_{a,s})=\{x_{0}\}\subset E$ and $B_{r_{1}}^{d_{a,s}}(0)=\{0\}$.

Following the definition of $F(t)$ in \eqref{4-31}, we can find a quadruple $(E_{d_{a,s}},x_{0},u_{1},\varepsilon_{1})$ with $u_1\in \mathbb{S}^{N-1}\cap E_+$ and sufficiently large $\varepsilon_{1}>0$, such that
   the corresponding set $F_{1}(t)$ satisfies $Q(\delta)\subset F_{1}(\delta)$ and $Q(0)\subset F_{1}(0)$. Here $F_{1}(t)$ is given by
\[ F_{1}(t):=    \overline{B_{r_{1}}^{d_{a,s}}(0)}+\overline{B_{t\varepsilon_{1}}^{N-1}(0)}+\{x_{0}+t\eta_{1} u_{1}\}
\]
for $0\leq t\leq \delta$, where $\eta_{1}=\langle u_{1},e_{N}\rangle^{-1}$. According to Lemma \ref{lemma4-5}, we obtain
\begin{equation*}
  P(\delta)={\rm conv}(Q(0)\cup Q(\delta))\subset {\rm conv}(F_{1}(0)\cup F_{1}(\delta))=\bigcup_{0\leq t\leq \delta}F_{1}(t):=T_{1}(\delta),
\end{equation*}
and
\begin{equation}\label{4-39}
  Q(t)\subset F_{1}(t)~~~~\mbox{for all}~~~~0\leq t\leq \delta.
\end{equation}

Furthermore, there exists a point $u_2\in \mathbb{S}^{N-1}\cap E_+$ such that $x_{0}+\delta \eta_{2}u_{2}\in\text{relint}(Q(\delta))$ with $\eta_{2}=\langle u_{2},e_{N}\rangle^{-1}$, since $Q(\delta)$ is an $(N-1)$-dimensional non-empty polytope. It follows that $Q(\delta)-\{x_{0}+\delta \eta_{2}u_{2}\}$ is a bounded closed subset in $E$, and $0\in \text{relint}(Q(\delta)-\{x_{0}+\delta \eta_{2}u_{2}\})$.  Thus, there is an $(N-1)$-dimensional ball $B_{r_{3}}^{N-1}(0)\subset E$ such that
\[ \{x_{0}+\delta \eta_{2}u_{2}\}+\overline{B_{r_{3}}^{N-1}(0)}\subset Q(\delta). \]
Similarly, we can also find a $d_{a,s}$-dimensional ball $B_{r_{2}}^{d_{a,s}}(0)\subset E_{d_{a,s}}$  such that $0<r_{2}<r_{3}$ and
\[ \{x_{0}\}+\overline{B_{r_{2}}^{d_{a,s}}(0)}\subset \mathcal{H}(M_{a,s}).\]
We mention that $B_{r_{2}}^{d_{a,s}}(0)=\{0\}$ if $d_{a,s}=0$. Hence, for the quadruple $(E_{d_{a,s}},x_{0},u_{2},\varepsilon_{2})$ with $\varepsilon_{2}=\frac{r_{3}-r_{2}}{\delta}>0$, the corresponding set
 \[ F_{2}(t):=\overline{B_{r_{2}}^{d_{a,s}}(0)}+\overline{B_{t\varepsilon_{2}}^{N-1}(0)}+\{x_{0}+t\eta_{2}u_{2}\} \]
satisfies
\[ \begin{aligned}
F_{2}(\delta)&=\{x_{0}+\delta \eta_{2}u_{2}\}+\overline{B_{r_{2}}^{d_{a,s}}(0)}+\overline{B_{\delta \varepsilon_{2}}^{N-1}(0)}\\
&\subset \{x_{0}+\delta \eta_{2}u_{2}\}+ \overline{B_{r_{2}+\delta\varepsilon_{2}}^{N-1}(0)}= \{x_{0}+\delta \eta_{2}u_{2}\}+\overline{B_{r_{3}}^{N-1}(0)}\subset Q(\delta)
\end{aligned}
\]
and
\[ F_{2}(0)= \{x_{0}\}+\overline{B_{r_{2}}^{d_{a,s}}(0)}\subset \mathcal{H}(M_{a,s}).\]
Therefore, we obtain
\begin{equation*}
 P(\delta)={\rm conv}(Q(0)\cup Q(\delta))\supset {\rm conv}(F_{2}(0)\cup F_{2}(\delta))=\bigcup_{0\leq t\leq \delta}F_{2}(t):=T_{2}(\delta),
\end{equation*}
and
\begin{equation}\label{4-40}
  F_{2}(t)\subset Q(t)~~\mbox{for all}~~0\leq t\leq \delta.
\end{equation}

As a result of \eqref{4-39} and \eqref{4-40}, we have
\begin{equation}\label{4-41}
 X_{T_{2}(\delta)}(t)\leq S(t)\leq X_{T_{1}(\delta)}(t)~~~~\mbox{for all}~~~0\leq t\leq \delta.
\end{equation}
Consequently, it follows from Lemma \ref{lemma4-4}, \eqref{4-38} and \eqref{4-41} that
\[ p_{2}(t)\leq S(t)\leq p_{1}(t)\qquad \mbox{for all}~~0\leq t\leq \delta, \]
where $p_{i}(t)=v_{i,d_{a,s}}t^{N-1-d_{a,s}}+\cdots+v_{i,0}t^{N-1}$ is the polynomial with $v_{i,j}>0$ for $i=1,2$ and $j=0,\ldots,d_{a,s}$.
\end{proof}

\begin{lemma}
\label{lemma4-8}
For $(a,s)\in \mathfrak{G}^{\circ}$ with $a-b-\mathbf{1}\neq 0$, let $\mathcal{H}(P_{a,s})$ and $\mathcal{H}(M_{a,s})$ be the sets given in Proposition \ref{prop4-5}. Suppose that $\mathcal{H}(P_{a,s})$ is unbounded and $\delta>0$ is the positive constant given in Lemma \ref{lemma4-7}, then there exists a positive constant $C_{1}$ such that $ S(t)\leq C_{1}t^{N-1}$ for all $t\geq \delta$.
\end{lemma}

\begin{proof}
We denote by $C_{\mathcal{H}(P_{a,s})}:=\{y\in\mathbb{R}^{N}|x+\gamma y\in \mathcal{H}(P_{a,s}),~\forall x\in \mathcal{H}(P_{a,s}),~\forall \gamma\geq0\}$ the recession cone consisting of the origin in $\mathbb{R}^N$ and all directions of $\mathcal{H}(P_{a,s})$. The unboundedness of  $\mathcal{H}(P_{a,s})$ indicates that $C_{\mathcal{H}(P_{a,s})}$ must contain at least one non-zero vector. Additionally, it follows from \cite[Theorem 8.2]{Rockafellar1970}  that $C_{\mathcal{H}(P_{a,s})}$ is closed.  By virtue of the Minkowski-Weyl's decomposition theorem for polyhedron (see \cite[Section 8]{Schrijver1998}), we have
\begin{equation}\label{4-42}
\mathcal{H}(P_{a,s})=P_{1}+C_{\mathcal{H}(P_{a,s})},
\end{equation}
where $P_{1}$ is a polytope.\par

Since the plane section $Q(t)=\mathcal{H}(P_{a,s})\cap \{x\in\mathbb{R}^{N}|x_N=t\}$ is bounded for any $t\in[0,\delta]$ and $\mathcal{H}(P_{a,s})\subset \overline{E_{+}}$, it derives that
 any non-zero vector $z=(z_{1},\ldots,z_{N})\in  C_{\mathcal{H}(P_{a,s})}$ must have
 $z_{N}>0$.  Let $ f(z):=\langle z,e_{N}\rangle $ be the continuous function defined on the compact set $C_{\mathcal{H}(P_{a,s})}\cap \mathbb{S}^{N-1}$. Then
$f(z)\geq c_{0}>0$  for all $z\in C_{\mathcal{H}(P_{a,s})}\cap \mathbb{S}^{N-1}$ and some $0<c_{0}<1$. Thus, $\frac{z_{N}}{|z|_{N}}\geq c_{0}>0$ holds for any non-zero vector $z\in C_{\mathcal{H}(P_{a,s})}$.

Recalling that $P_{1}$ is a polytope, we can find an $N$-dimensional ball $B_{r_{4}}^{N}(0)$ such that $P_{1}\subset B_{r_{4}}^{N}(0)$. Therefore,  \eqref{4-42} gives that
\begin{equation}\label{4-43}
 \mathcal{H}(P_{a,s})=P_{1}+C_{\mathcal{H}(P_{a,s})}\subset B_{r_{4}}^{N}(0)+C_{\mathcal{H}(P_{a,s})}.
\end{equation}
Combining \eqref{4-36}, \eqref{4-37} and \eqref{4-43}, we obtain
\begin{equation}\label{4-44}
\begin{aligned}
  S(t)&=X_{\mathcal{H}(P_{a,s})}(t)=V_{N-1}(\mathcal{H}(P_{a,s})\cap\{x\in \mathbb{R}^{N}|x_{N}=t\})\\
  &\leq   V_{N-1}((B_{r_{4}}^{N}(0)+C_{\mathcal{H}(P_{a,s})})\cap\{x\in \mathbb{R}^{N}|x_{N}=t\}).
\end{aligned}
\end{equation}

We next estimate the upper bound of $V_{N-1}((B_{r_{4}}^{N}(0)+C_{\mathcal{H}(P_{a,s})})\cap\{x\in \mathbb{R}^{N}|x_{N}=t\})$.
For any $x\in  B_{r_{4}}^{N}(0)+C_{\mathcal{H}(P_{a,s})}$, we have $x=y+z$, where $y\in B_{r_{4}}^{N}(0)$ and $z\in C_{\mathcal{H}(P_{a,s})}$. Additionally, examining the coordinates $x=(x_{1},\ldots,x_{N})=(y_{1}+z_{1},\ldots,y_{N}+z_{N})$, we get that $|y|_{N}\leq r_{4}$ and $\frac{z_{N}}{|z|_{N}}\geq c_{0}>0$ if $|z|_{N}\neq 0$. Setting $x_{N}=y_{N}+z_{N}=t>0$, the upper bound  will be examined in the following two cases:\\
 \emph{\textbf{Case 1:}} $|z|_{N}\neq 0$. It follows that
\begin{equation}\label{4-45}
0<c_{0}|z|_{N}\leq z_{N}=t-y_{N}\leq t+r_{4},
\end{equation}
which means
\begin{equation}\label{4-46}
z_{1}^{2}+\cdots+z_{N-1}^{2}\leq \frac{1-c_{0}^{2}}{c_{0}^{2}}z_{N}^{2}.
\end{equation}
Combining  \eqref{4-45} and \eqref{4-46}, we obtain $|z_{j}|\leq \frac{1}{c_{0}}z_{N}\leq \frac{t+r_{4}}{c_{0}}$ and $|x_{j}|\leq |y_{j}|+|z_{j}|\leq \frac{t+r_{4}}{c_{0}}+r_{4}$ for $1\leq j\leq N-1$.
Thus, for $t>0$ we have
\begin{equation}\label{4-47}
 V_{N-1}\left((B_{r_{4}}^{N}(0)+C_{\mathcal{H}(P_{a,s})})\cap\{x\in \mathbb{R}^{N}|x_{N}=t\}\right)\leq 2^{N-1} \left(\frac{t+r_{4}}{c_{0}}+r_{4}\right)^{N-1}.
\end{equation}
\emph{\textbf{Case 2:}} $|z|_{N}= 0$. In this case, $y_{N}=t>0$ and $|x_{j}|=|y_{j}|\leq r_{4}$ for $1\leq j\leq N-1$. Hence,
 \begin{equation}\label{4-48}
 V_{N-1}\left((B_{r_{4}}^{N}(0)+C_{\mathcal{H}(P_{a,s})})\cap\{x\in \mathbb{R}^{N}|x_{N}=t\}\right)\leq (2r_{4})^{N-1}\leq 2^{N-1}\left(\frac{t+r_{4}}{c_{0}}+r_{4}\right)^{N-1}.
\end{equation}

According to  \eqref{4-44}, \eqref{4-47} and \eqref{4-48}, there is a positive constant $C_{1}>0$ such that
$ S(t)\leq C_{1}t^{N-1}$ for all $t\geq \delta$.
\end{proof}

We can now apply Proposition \ref{prop4-5}, Lemma \ref{lemma4-7} and Lemma \ref{lemma4-8} to derive the asymptotic behaviour of integral $J_{a,s}(r)$ when $a-b-\mathbf{1}\neq 0$.
\begin{proposition}
\label{prop4-6}
For $(a,s)\in \mathfrak{G}^{\circ}$ with $a-b-\mathbf{1}\neq 0$, let $\mathcal{H}(P_{a,s})$ and $\mathcal{H}(M_{a,s})$ be the sets given in Proposition \ref{prop4-5}. Then the integral $J_{a,s}(r)$ defined in \eqref{4-29} admits the following asymptotic behaviour:
\begin{equation}
\label{4-49}
 J_{a,s}(r)\approx \frac{|\ln r|^{d_{a,s}}}{r^{s+m_{a,s}}}~~~\mbox{as}~~r\to 0^{+},
\end{equation}
where $m_{a,s}$ and $d_{a,s}$ are the constants defined in Lemma \ref{lemma4-3} above.
\end{proposition}

\begin{proof}
Lemma \ref{lemma4-7} gives that $p_{2}(t)\leq S(t)\leq p_{1}(t)$ for all $0\leq t\leq \delta$,
where $\delta>0$ is a positive constant, and $p_{i}(t)=v_{i,d_{a,s}}t^{N-1-d_{a,s}}+\cdots+v_{i,0}t^{N-1}$ with $v_{i,j}>0$ for $i=1,2$ and $j=0,\ldots,d_{a,s}$.
 Then, by \eqref{4-29} and \eqref{4-30} we have for $0<r<1$,
\begin{equation}
\begin{aligned}\label{4-50}
  J_{a,s}(r)&=\left(\ln\frac{1}{r}\right)^{N}\frac{C_{0}}{r^{s+m_{a,s}}}
    \int_{0}^{\delta}e^{-\left(\ln\frac{1}{r}\right)t}S(t) d t+\left(\ln\frac{1}{r}\right)^{N}\frac{C_{0}}{r^{s+m_{a,s}}}
    \int_{\delta}^{+\infty}e^{-\left(\ln\frac{1}{r}\right)t}S(t) d t\\
&:=J_{a,s,1}(r)+ J_{a,s,2}(r),
\end{aligned}
\end{equation}
where $\delta>0$ is the positive constant given in Lemma \ref{lemma4-7}.

A direct calculation yields that for $i=1,2$ and $0<r<1$,
\begin{equation}
\begin{aligned}\label{4-51}
 \int_{0}^{\delta}e^{-\left(\ln\frac{1}{r}\right)t}p_i(t) dt &= \sum_{k=N-1-d_{a,s}}^{N-1}v_{i,N-1-k}\int_{0}^{\delta}e^{-\left(\ln\frac{1}{r}\right)t}t^k dt \\
   &= \sum_{k=N-1-d_{a,s}}^{N-1}\frac{v_{i,N-1-k}}{\left(\ln \frac{1}{r}\right)^{k+1}}\int_{0}^{\delta\ln \frac{1}{r}}e^{-u}u^k du \\
   & \approx \frac{1}{\left(\ln \frac{1}{r}\right)^{N-d_{a,s}}}~~~\mbox{as}~~r\to 0^{+},
\end{aligned}
\end{equation}
which means
\begin{equation}\label{4-52}
  J_{a,s,1}(r)=\left(\ln\frac{1}{r}\right)^{N}\frac{C_{0}}{r^{s+m_{a,s}}}
    \int_{0}^{\delta}e^{-\left(\ln\frac{1}{r}\right)t}S(t) d t\approx \frac{|\ln r|^{d_{a,s}}}{r^{s+m_{a,s}}}~~~\mbox{as}~~r\to 0^{+}.
\end{equation}
Furthermore,  by Lemma \ref{lemma4-8} we have $S(t)\leq C_{1}t^{N-1}$ for all $t\geq \delta$. Therefore,
\begin{equation}\label{4-53}
\begin{aligned}
 0&\leq \left(\ln \frac{1}{r}\right)^{N-d_{a,s}} \int_{\delta}^{+\infty}e^{-\left(\ln\frac{1}{r}\right)t}S(t) dt\leq C_{1}\left(\ln \frac{1}{r}\right)^{N-d_{a,s}} \int_{\delta}^{+\infty}e^{-\left(\ln\frac{1}{r}\right)t}t^{N-1}dt\\
&=C_{1}\left(\ln \frac{1}{r}\right)^{-d_{a,s}}\int_{\delta\ln \frac{1}{r}}^{+\infty}e^{-u}u^{N-1}du\to 0\qquad\mbox{as}~~~r\to 0^{+}.
\end{aligned}
\end{equation}
It follows from \eqref{4-53} that
\begin{equation}\label{4-54}
J_{a,s,2}(r)=\left(\ln\frac{1}{r}\right)^{N}\frac{C_{0}}{r^{s+m_{a,s}}}
    \int_{\delta}^{+\infty}e^{-\left(\ln\frac{1}{r}\right)t}S(t) d t=o\left(\frac{|\ln r|^{d_{a,s}}}{r^{s+m_{a,s}}} \right)~~\mbox{as}~r\to 0^{+}.
\end{equation}
Combining \eqref{4-50}, \eqref{4-52} and \eqref{4-54}, we have
\[
 J_{a,s}(r)=\left(\ln\frac{1}{r}\right)^{N}\frac{C_{0}}{r^{s+m_{a,s}}}
    \int_{0}^{+\infty}e^{-\left(\ln\frac{1}{r}\right)t}S(t) d t\approx \frac{|\ln r|^{d_{a,s}}}{r^{s+m_{a,s}}}~~~\mbox{as}~~r\to 0^{+}.
\]
\end{proof}

Finally, let us summarize the asymptotic results in this section.
\begin{proposition}
\label{prop4-7}
Let $\mathfrak{G}$ be the collection of finitely many index pairs $(a,s)$, and let $b$ be an
$N$-dimensional multi-index. Assume that for each index pair $(a,s)\in \mathfrak{G}$, $a=(a_{1},\ldots,a_{N})$ is an $N$-dimensional multi-index, and $s\in [0,+\infty)$ is a non-negative constant. Suppose further that
\[
I(r)=\int_{(0,1]^N}\frac{x^{b} dx}{\sum_{(a,s)\in \mathfrak{G}} x^{a}r^{s}}<+\infty\qquad\mbox{for all}\qquad0<r<1,\]
 then we have
  \begin{equation}
\label{4-55}
   I(r)\approx\sum_{(a,s)\in \mathfrak{G}^\circ}r^{-s-m_{a,s}}|\ln r|^{d_{a,s}}\approx r^{-\alpha_0}|\ln r|^{d_0}~~\mbox{as}~~r\to 0^{+}.
  \end{equation}
Here, $\mathfrak{G}^\circ=\{(a,s)\in\mathfrak{G}|V_{N}(P_{a,s})>0\}$, $P_{a,s}$ is the polyhedron defined in \eqref{4-15}, $m_{a,s}=\max_{x\in P_{a,s}}\phi_{a}(x)$ and $d_{a,s}=\dim M_{a,s}$, where  $\phi_{a}(y)=\langle a-b-\mathbf{1},y\rangle$ is a linear function and $M_{a,s}=\{x\in P_{a,s}|\phi_{a}(x)=m_{a,s}\}$ is the face of $P_{a,s}$. Additionally, the indexes $\alpha_0$ and $d_0$ in \eqref{4-55} are given by
  \begin{equation}\label{4-56}
    \alpha_0:=\max\{s+m_{a,s}|(a,s)\in \mathfrak{G}^\circ\}~~\mbox{and}~~d_0:=\max\{d_{a,s}|(a,s)\in \mathfrak{G}^\circ, s+m_{a,s}=\alpha_0\}.
  \end{equation}
Furthermore, if for each index pair $(a,s)\in \mathfrak{G}$, the corresponding index $s$ is a non-negative integer, then we have $\alpha_{0}\in \mathbb{Q}$.
\end{proposition}

\begin{proof}
The asymptotic estimate \eqref{4-55} is a direct result of \eqref{4-17}, \eqref{4-26} and \eqref{4-49}. Moreover, if for each index pair $(a,s)\in \mathfrak{G}$, the corresponding index $s$ is a non-negative integer, then all $P_{a,s}$ given by \eqref{4-15} are rational polyhedrons (see \cite[Chapter 13]{Barvinok2008}).
 For any index pair $(a,s)\in \mathfrak{G}$,  since $\phi_{a}(y)=\langle a-b-\mathbf{1},y\rangle$ is a rational transformation from $\mathbb{R}^{N}$ to $\mathbb{R}$, it follows from \cite[Chapter 13, p. 108]{Barvinok2008} that $\phi_{a}(P_{a,s})$ is also a rational polyhedron in $\mathbb{R}$. Using  Lemma \ref{lemma4-3}, we can conclude that $\phi_{a}(P_{a,s})=\{x\in \mathbb{R}|x\leq m_{a,s}\}$ or $\phi_{a}(P_{a,s})=\{x\in \mathbb{R}|c\leq x\leq m_{a,s}\}$ for some constant $c\leq m_{a,s}$, which are both rational polyhedrons. This means that $m_{a,s}\in \mathbb{Q}$ for all $(a,s)\in \mathfrak{G}$ and $\alpha_{0}\in \mathbb{Q}$.
\end{proof}

Using Proposition \ref{prop4-7}, we obtain the following explicit asymptotic estimate of $J_{\Omega}(r)$.
\begin{proposition}
\label{prop4-8}
Let $X=(X_{1},X_{2},\ldots,X_{m})$ be the homogeneous H\"{o}rmander vector fields defined on $\mathbb{R}^n$. Suppose that $\{x_{i_{1}},\ldots,x_{i_{v}}\}$ is the collection of all degenerate components of $X$, and $\Omega\subset \mathbb{R}^n$ is a bounded open domain containing the origin. Then
\begin{equation}\label{4-57}
  J_{\Omega}(r)=\int_{\Omega}\frac{dx}{\Lambda(x,r)}\approx r^{-Q_0}|\ln r|^{d_0}~~~\mbox{as}~~r\to 0^+,
\end{equation}
where $Q_0\in\mathbb{Q}$ and $d_0\in\{0,1,\ldots,v\}$ with $0\leq v\leq n-1$.
\end{proposition}

\begin{proof}
 Remark \ref{remark2-2} indicates that  $0\leq v\leq n-1$.
If $v=0$, Proposition \ref{prop2-7} and Proposition \ref{prop2-10} imply that $\Lambda(x,r)=f_{Q}r^{Q}$ with $f_{Q}>0$, which derives \eqref{4-57}. Clearly, Proposition \ref{prop4-1} yields  \eqref{4-57} in the case of $v=1$. For $2\leq v\leq n-1$, by Proposition \ref{prop4-3} we have
\begin{equation}
\label{4-58}
  J_{\Omega}(r)=\int_{\Omega}\frac{dx}{\Lambda(x,r)}\approx \sum_{j=1}^{l}\int_{(-1,1)^{v}}\frac{|u^{q_{j}}| du}{\sum_{I}|u^{p_{j,I}}|r^{d(I)}}\approx \sum_{j=1}^{l}\int_{(0,1]^{v}}\frac{u^{q_{j}} du}{\sum_{I}u^{p_{j,I}}r^{d(I)}},
\end{equation}
where $l\in \mathbb{N}^{+}$ is a positive integer, $q_{j}$ and $p_{j,I}$ are $v$-dimensional multi-indexes. Owing to Proposition \ref{prop4-7}, we obtain for $1\leq j\leq l$,
\begin{equation}\label{4-59}
 \int_{(0,1]^{v}}\frac{u^{q_{j}} du}{\sum_{I}u^{p_{j,I}}r^{d(I)}}\approx r^{-\tilde{\alpha}_j}|\ln r|^{\tilde{d}_j}~~\mbox{as}~~r\to 0^{+},
\end{equation}
where $\tilde{\alpha}_{j}\in \mathbb{Q}$ and $\tilde{d}_{j}\in \{0,1,\ldots, v\}$. Hence, combining \eqref{4-58} and \eqref{4-59} we get
\[
  J_{\Omega}(r)=\int_{\Omega}\frac{dx}{\Lambda(x,r)}\approx r^{-Q_0}|\ln r|^{d_0}~~~\mbox{as}~~r\to 0^+,\]
where $Q_0\in\mathbb{Q}$ and $d_0\in\{0,1,\ldots,v\}$ with $0\leq v\leq n-1$.
\end{proof}

\section{Proofs of main results}
\label{Section5}
\subsection{Proof of Theorem \ref{thm1}}
\begin{proof}[Proof of Theorem \ref{thm1}]
By  \eqref{3-21} and \eqref{3-43}, we have
\begin{equation}\label{5-1}
   \int_{\Omega}h_D(x,x,t) dx\leq A_{1} \int_{\Omega}\frac{d x}{|B_{d_{X}}(x,\sqrt{t})|}~~~\mbox{for all}~~t>0.
\end{equation}
According to \eqref{3-21} and Proposition \ref{prop3-7}, for any compact subset $K\subset \Omega$, there exist  $\eta(K):=\frac{d_{X}^{2}(K,\partial\Omega)}{A_{1}Q}$ and $\eta_{1}(K)$ with $0<\eta_{1}(K)\leq  \eta(K)$, such that for any $0<t\leq \eta_{1}(K)$,
\begin{equation}\label{5-2}
\begin{aligned}
 \int_{\Omega}h_{D}(x,x,t)dx&\geq \int_{K}h_{D}(x,x,t)dx=\int_{K}h(x,x,t)dx- \int_{K}E(x,x,t)dx\\
 &\geq \int_{K}\frac{1}{A_{1}|B_{d_{X}}(x,\sqrt{t})|}dx-\int_{K}\frac{2A_{1}C_{3}}{|B_{d_{X}}(x,\sqrt{t})|}e^{-\frac{d_{X}^{2}(x,\partial\Omega)}{A_{1}t} }dx\\
 &\geq \left(\frac{1}{A_{1}}-2A_{1}C_{3} e^{-\frac{d_{X}^2(K,\partial\Omega)}{A_{1}t}}\right)\int_{K} \frac{dx}{|B_{d_{X}}(x,\sqrt{t})|}\\
&\geq \frac{1}{2 A_{1}}\int_{K} \frac{dx}{|B_{d_{X}}(x,\sqrt{t})|}.
\end{aligned}
\end{equation}

For $0<\varepsilon<1$, we let $K_{\varepsilon}:=\overline{\delta_{\varepsilon}(\Omega)}$ be the compact subset of $\mathbb{R}^{n}$, where $\delta_{\varepsilon}(\Omega):=\{\delta_{\varepsilon}(x)|x\in \Omega\}$ and $\delta_{\varepsilon}(x)$ is the dilation given in assumption (H.1). We then find a $\varepsilon_{0}\in (0,1)$ such that $K_{\varepsilon_{0}}\subset \Omega$.

 Recalling that $\Omega$ is a bounded open domain containing the origin, there exist $0<R_{1}<R_{2}$ such that
$B_{R_{1}}^{n}(0)\subset \Omega \subset B_{R_{2}}^{n}(0)$, where $B_{R}^{n}(0):=\{x\in \mathbb{R}^{n}||x|_{n}<R\}$ denotes the classical Euclidean ball in $\mathbb{R}^{n}$.
For any $x\in \delta_{\varepsilon}(B_{R_{2}}^{n}(0))$ with $0<\varepsilon<1$, there is a unique point $y=(y_{1},\ldots,y_{n})\in B_{R_{2}}^{n}(0)$ such that $x=\delta_{\varepsilon}(y)$. Then we have
$ |x|_{n}^{2}=\varepsilon^{2\alpha_{1}}y_{1}^{2}+\cdots+ \varepsilon^{2\alpha_{n}}y_{n}^{2}\leq \varepsilon^{2}(y_{1}^{2}+\cdots+y_{n}^{2})<\varepsilon^{2}R_{2}^{2}$, and
 $\delta_{\varepsilon}(B_{R_{2}}^{n}(0))\subset B_{\varepsilon R_{2}}^{n}(0)$. Choosing $\varepsilon_{0}=\frac{R_{1}}{2R_{2}}\in (0,1)$, we get
$ \delta_{\varepsilon_{0}}(\Omega)\subset \delta_{\varepsilon_{0}}(B_{R_{2}}^{n}(0))\subset B_{\frac{R_{1}}{2}}^{n}(0)$.
This means $K_{\varepsilon_{0}}=\overline{\delta_{\varepsilon_{0}}(\Omega)}\subset \overline{B_{\frac{R_{1}}{2}}^{n}(0)}\subset B_{R_{1}}^{n}(0)\subset \Omega$.

Next, by Corollary \ref{corollary2-1} and the homogeneity property (3) of $B_{d_{X}}(x,r)$  we have
\begin{equation}\label{5-3}
\begin{aligned}
\int_{K_{\varepsilon_{0}}}\frac{dx}{|B_{d_{X}}(x,\sqrt{t})|}&=\int_{\delta_{\varepsilon_{0}}(\Omega)}\frac{dx}{|B_{d_{X}}(x,\sqrt{t})|}=\int_{\Omega}\frac{\varepsilon_{0}^{Q} dy}{|B_{d_{X}}(\delta_{\varepsilon_{0}}(y),\sqrt{t})|}
=\int_{\Omega}\frac{dy}{|B_{d_{X}}(y,\frac{\sqrt{t}}{\varepsilon_{0}})|}\\
&\geq \frac{\varepsilon_{0}^{Q}}{C_{3}}\int_{\Omega}\frac{dy}{|B_{d_{X}}(y,\sqrt{t})|}~~~\mbox{for all}~~ t>0.
\end{aligned}
\end{equation}
Hence, it follows from \eqref{5-1}-\eqref{5-3} that
 \begin{equation*}
\int_{\Omega}h_{D}(x,x,t) dx\approx\int_{\Omega}\frac{dx}{|B_{d_X}(x,\sqrt{t})|}~~ \mbox{as}~~~t\to 0^+.
\end{equation*}
The proof of Theorem \ref{thm1} is complete.
\end{proof}

\subsection{Proof of Theorem \ref{thm2}}

We introduce the following useful lemma to prove Theorem \ref{thm2}.

\begin{lemma}
\label{lemma5-1}
Let $f:\mathbb{R}^{+}\to \mathbb{R}$ be a real function such that $f(x)\approx x^{-\mu_0}|\ln x|^b$ as $x\to 0^+$ with some $\mu_0\in \mathbb{R}$ and $b>0$. Denote by $g_a(x):=x^af(x)$ for $a>0$. If there exists some positive constants $\mu_2>\mu_1>0$ such that
\begin{enumerate}[(1)]
    \item $\liminf_{x\to 0^+}g_{\mu_1}(x)>0$;
    \item $\lim_{x\to 0^+}g_{\mu_2+\varepsilon}(x)=0$ holds for any $\varepsilon\in (0,1)$.
  \end{enumerate}
Then $\mu_0\in[\mu_1,\mu_2]$.
\end{lemma}
\begin{proof}
Since $f(x)\approx x^{-\mu_0}|\ln x|^b$ as $x\to 0^+$, there are some constants $0<c_1\leq c_2<+\infty$ and $\delta>0$ such that
 \[c_1 x^{-\mu_0}|\ln x|^b \leq f(x) \leq  c_2x^{-\mu_0}|\ln x|^b ~~~~\mbox{for all}~~x\in (0,\delta). \]
If $\mu_1>\mu_0$, it follows that
 \[ \liminf_{x\to 0^+}g_{\mu_1}(x)\leq c_{2}\limsup_{x\to 0^+} x^{\mu_1-\mu_0}|\ln x|^b=0,\]
which contradicts condition (1). This means $\mu_{0}\geq \mu_{1}$. Additionally, if we assume $\mu_2<\mu_0$, there exists a positive constant $l>1$ such that $\varepsilon=\frac{1}{l}(\mu_0-\mu_2)\in (0,1)$. Then
\[ g_{\mu_{2}+\varepsilon}(x)=x^{\frac{1}{l}\mu_{0}+\left(1-\frac{1}{l} \right)\mu_{2}}f(x)\geq c_{1}x^{\left(1-\frac{1}{l}\right)(\mu_{2}-\mu_{0})}|\ln x|^{b} \]
holds for all $x\in (0,\delta)$. This implies $\liminf_{x\to 0^+}g_{\mu_{2}+\varepsilon}(x)=+\infty$, which contradicts  condition (2). Hence, we have $\mu_0\in[\mu_1,\mu_2]$.
  \end{proof}

\begin{proof}[Proof of Theorem \ref{thm2}]
Proposition \ref{prop2-8} yields that
 \begin{equation}\label{5-4}
\int_{\Omega}\frac{d x}{|B_{d_X}(x,r)|}\approx \int_{\Omega}\frac{dx}{\Lambda(x,r)}=J_{\Omega}(r).
\end{equation}
Thus,  we only need to be concerned with the explicit asymptotic behavior of $J_{\Omega}(r)$ as $r\to 0^{+}$. Our estimates will be derived in the following two cases:\\
\emph{\textbf{Case 1:}} $w=Q$. By Proposition \ref{prop2-7}  we have
\begin{equation}
\label{5-5}
 J_{\Omega}(r)=\int_{\Omega}\frac{dx}{\Lambda(x,r)}=\frac{|\Omega|}{f_{Q}(0)}r^{-Q}.
\end{equation}
\emph{\textbf{Case 2:}} $w\leq Q-1$. Suppose that $\{x_{i_{1}},\ldots,x_{i_{v}}\}$ is the collection of all degenerate components of vector fields $X$ associated with the degenerate indexes $\{\alpha_{i_{1}},\ldots,\alpha_{i_{v}}\}$, and $\alpha(X)=\alpha_{i_1}+\cdots+\alpha_{i_v}$ is the sum of all degenerate indexes. Proposition \ref{prop4-8} shows that
\begin{equation}\label{5-6}
  J_{\Omega}(r)=\int_{\Omega}\frac{dx}{\Lambda(x,r)}\approx r^{-Q_0}|\ln r|^{d_0}~~~\mbox{as}~~r\to 0^+,
\end{equation}
where $Q_0\in\mathbb{Q}$ and $d_0\in\{0,1,\ldots,v\}$ with $0\leq v\leq n-1$. Hence, it remains to examine the bounds of index $Q_{0}$.

Using Proposition \ref{prop2-7} and Lemma \ref{lemma4-1}, we obtain
\begin{equation}
\begin{aligned}\label{5-7}
 J_{\Omega}(r)&\approx J_{(-1,1)^v,v}(r)
 =\int_{(-1,1)^v}\frac{dx_{i_{1}}\cdots dx_{i_{v}}}{\Lambda(x,r)}\\
 &=\frac{1}{r^Q}\int_{(-1,1)^v}\frac{dx_{i_{1}}\cdots dx_{i_{v}}}{\sum_{k=w}^{Q}f_k(x)r^{k-Q}}
 =\frac{1}{r^Q}\int_{(-1,1)^v}\frac{dx_{i_{1}}\cdots dx_{i_{v}}}{\sum_{k=w}^{Q}f_k(\delta_{\frac{1}{r}}(x))}\\
 &=\frac{1}{r^{Q-\alpha(X)}}\int_{\prod_{j=1}^{v}(-r^{-\alpha_{i_{j}}},r^{-\alpha_{i_{j}}})}\frac{dy_{1}\cdots dy_{v}}{\sum_{k=w}^{Q}f_k(y)}.
\end{aligned}
\end{equation}
Here, we change the variables $y_{j}=r^{-\alpha_{i_{j}}}x_{i_{j}}$ for $1\leq j\leq v$ in the last step of \eqref{5-7}. We then consider the function
$ g_{a}(r):=r^{a}J_{\Omega}(r)$ for $r>0$.
It follows from Proposition \ref{prop2-7} that $\Lambda(y,1)=\sum_{k=w}^{Q}f_k(y)\geq f_{Q}(0)>0$. Therefore, \eqref{5-7} gives that
\begin{equation}\label{5-8}
\begin{aligned}
\liminf_{r\to 0^{+}}g_{Q-\alpha(X)}(r)&=\liminf_{r\to 0^{+}} r^{Q-\alpha(X)} J_{\Omega}(r)\\
&\geq C\cdot\liminf_{r\to 0^{+}}\int_{\prod_{j=1}^{v}(-r^{-\alpha_{i_{j}}},r^{-\alpha_{i_{j}}})}\frac{dy_{1}\cdots dy_{v}}{\Lambda(y,1)}>0,
\end{aligned}
\end{equation}
where $C>0$ is a positive constant, and the last term in \eqref{5-8} is finite or positive infinity. Furthermore, we have
\begin{equation}\label{5-9}
\begin{aligned}
    g_w(r)&=r^wJ_{\Omega}(r)
    =\int_{\Omega}\frac{d x}{f_w(x)+f_{w+1}(x)r+\cdots+f_Q(x)r^{Q-w}}\\
    &\geq\int_{\Omega}\frac{d x}{\sum\limits_{k=w}^{Q}f_k(x)}>0~~~~\mbox{for all}~~0<r<1,
\end{aligned}
  \end{equation}
which implies $\liminf_{r\to 0^+}g_{w}(r)>0$. Hence, $\mu_1:=\max\{Q-\alpha(X),w\}>0$ satisfies the condition (1) of Lemma \ref{lemma5-1}. Next, we show that for any $\varepsilon\in (0,1)$,
\begin{equation}\label{5-10}
  \lim_{r\to 0^{+}}g_{Q-1+\varepsilon}(r)=\lim_{r\to 0^{+}}r^{Q-1+\varepsilon}J_{\Omega}(r)=0,
\end{equation}
which indicates that $\mu_2:=Q-1>0$ admits the condition (2) of Lemma \ref{lemma5-1}.
Using Proposition \ref{prop2-7} again, we obtain $f_{w}(x_{0})\neq 0$ for some $x_{0}\in  \mathbb{R}^n$ and $f_Q(x)=f_{Q}(0)>0$ for all $x\in \mathbb{R}^n$, provided $w\leq Q-1$. That means $\lambda_{\widetilde{I}}(x)\not\equiv 0$ holds for some $n$-tuple $\widetilde{I}$ with $d(\widetilde{I})=w$.
 For any $\varepsilon\in (0,1)$ and $r>0$, we have
\begin{equation}
\begin{aligned}\label{5-11}
  0\leq g_{Q-1+\varepsilon }(r)&= r^{Q-1+\varepsilon}J_\Omega(r)
  =r^{Q-1}\int_{\Omega}\frac{r^{\varepsilon} dx}{\Lambda(x,r)}\\
   &=\int_{\Omega}\frac{r^{\varepsilon} dx}{\sum_{k=w}^{Q}f_k(x)r^{k-Q+1}}
   \leq\int_{\Omega}\frac{r^{\varepsilon} dx}{|\lambda_{\widetilde{I}}(x)|r^{w-Q+1}+f_Q(0)r}.\\
\end{aligned}
\end{equation}

On the other hand, we mention that the set $Z(\lambda_{\widetilde{I}}):=\{x\in \mathbb{R}^{n}|\lambda_{\widetilde{I}}(x)=0\}$ has zero $n$-dimensional measure since $\lambda_{\widetilde{I}}$ is a polynomial. Thus for $r>0$, the function $h_{\varepsilon}(x,r)$ given by \[h_{\varepsilon}(x,r):=\frac{r^{\varepsilon}}{|\lambda_{\widetilde{I}}(x)|r^{w-Q+1}+f_Q(0)r}=\frac{r^{Q-1-w+\varepsilon}}{|\lambda_{\widetilde{I}}(x)|+f_Q(0)r^{Q-w}}\]
satisfies that
\begin{equation}\label{5-12}
 \int_{\Omega}\frac{r^{\varepsilon} dx}{|\lambda_{\widetilde{I}}(x)|r^{w-Q+1}+f_Q(0)r}=\int_{\Omega\setminus Z(\lambda_{\widetilde{I}})}h_{\varepsilon}(x,r)d x,
\end{equation}
and
\begin{equation}\label{5-13}
\lim_{r\to 0^{+}}h_{\varepsilon}(x,r)=0\qquad \forall x\in \Omega\setminus Z(\lambda_{\widetilde{I}}).
\end{equation}
 Moreover, for any $x\in \Omega\setminus Z(\lambda_{\widetilde{I}})$ and $r>0$,
\begin{equation*}
    \frac{1}{h_\varepsilon(x,r)}=\frac{|\lambda_{\widetilde{I}}(x)|}{r^{Q+\varepsilon-w-1}}+f_{Q}(0)r^{1-\varepsilon}\geq \frac{1}{C}|\lambda_{\widetilde{I}}(x)|^{\frac{1-\varepsilon}{Q-w}}
  \end{equation*}
holds for some positive constant $C>0$, which means
  \begin{equation}\label{5-14}
  h_{\varepsilon}(x,r)\leq \frac{C}{|\lambda_{\widetilde{I}}(x)|^{\frac{1-\varepsilon}{Q-w}}}\qquad \forall x\in \Omega\setminus Z(\lambda_{\widetilde{I}}),~ r>0.
\end{equation}
It follows from \eqref{2-1} and Proposition \ref{prop2-4} that $\lambda_{\widetilde{I}}$ is a $\delta_t$-homogeneous polynomial of degree $Q-w$ having the form
\[\lambda_{\widetilde{I}}(x)=\sum_{\sum_{i=1}^{n}\alpha_{i}\beta_{i}=Q-w}c_{\beta_{1},\ldots,\beta_{n}}x_{1}^{\beta_{1}}x_{2}^{\beta_{2}}\cdots x_{n}^{\beta_{n}},\]
where $\beta_{1},\beta_{2},\ldots,\beta_{n}$ are non-negative integers. For
 each monomial $c_{\beta_{1},\ldots,\beta_{n}}x_{1}^{\beta_{1}}x_{2}^{\beta_{2}}\cdots x_{n}^{\beta_{n}}$, we have
\[\frac{1-\varepsilon}{Q-w}(\beta_1+\cdots+\beta_n)\leq \frac{1-\varepsilon}{Q-w}(\alpha_1\beta_1+\cdots+\alpha_n\beta_n)=1-\varepsilon<1.\]
Thus, by using \cite[Proposition 4.1]{Dieu2018} we get
\begin{equation}\label{5-15}
 \int_{\Omega}\frac{d x}{|\lambda_{\widetilde{I}}(x)|^{\frac{1-\varepsilon}{Q-w}}}
<+\infty.
\end{equation}
Combining \eqref{5-11}-\eqref{5-15} and the Lebesgue's dominated convergence theorem, we obtain \eqref{5-10}.
 Consequently, Lemma \ref{lemma5-1} yields that $ n\leq\max\{Q-\alpha(X),w\}\leq Q_{0}\leq Q-1$.

The proof of Theorem \ref{thm2} is complete.
\end{proof}

\subsection{Proof of Theorem \ref{thm3}}

\begin{proof}[Proof of Theorem \ref{thm3}]
According to Theorem \ref{thm1}, Theorem \ref{thm2} and Proposition \ref{prop3-6}, we have
\begin{equation}\label{5-16}
  \sum_{k=1}^{\infty}e^{-t\lambda_k}=\int_{\Omega}h_{D}(x,x,t)dx\approx \int_{\Omega}\frac{dx}{|B_{d_X}(x,\sqrt{t})|}\approx t^{-\frac{Q_0}{2}}|\ln t|^{d_0}~~~\mbox{as}~~t\to 0^{+},
\end{equation}
where $Q_0$ and $d_0$ are the indexes given in Theorem \ref{thm2}. Applying the Tauberian theorem (see \cite[Proposition B.0.13]{karamata}) to \eqref{5-16}, we get
\begin{equation}\label{5-17}
N(\lambda)\approx\int_{\Omega}\frac{dx}{|B_{d_X}(x,\lambda^{-\frac{1}{2}})|}\approx \lambda^{\frac{Q_0}{2}}|\ln \lambda|^{d_0}~~~~\mbox{as}~~\lambda\to +\infty,
\end{equation}
where $N(\lambda):=\{k|\lambda_{k}\leq \lambda\}$ is the counting function.

We next present the explicit asymptotic behaviour of $\lambda_{k}$. By \eqref{5-17}, there
exist positive constants $M>e^{2}$ and $C>1$, such that for any $\lambda> M$,
\begin{equation}\label{5-18}
0<\frac{1}{C}\lambda^{\frac{Q_0}{2}}|\ln \lambda|^{d_0}\leq N(\lambda)\leq C\lambda^{\frac{Q_0}{2}}|\ln \lambda|^{d_0}.
\end{equation}
Because $\lambda_{k}\to +\infty$ as $k\to +\infty$, \eqref{5-18} yields
\begin{equation}\label{5-19}
  k\leq N(\lambda_{k})\leq C\lambda_{k}^{\frac{Q_{0}}{2}}|\ln \lambda_{k}|^{d_{0}}~~~\mbox{for all}~~k\geq k_{0},
\end{equation}
 where $k_{0}$ is a positive integer such that $\lambda_{k}> M$ for any $k\geq k_{0}$. For $\lambda>\lambda_{1}$, we define
\begin{equation}\label{5-20}
M(\lambda):=\lim_{p\to \lambda^{-}}N(p).
\end{equation}
Clearly, $M(\lambda)$ is a left continuous function with $M(\lambda_{k})<k$ for all $k\geq 1$. For any $\lambda_{0}>M$, \eqref{5-18} and \eqref{5-20} indicates that
\begin{equation}\label{5-21}
  M(\lambda_{0})=\lim_{p\to \lambda_{0}^{-}}N(p)\geq \lim_{p\to \lambda_{0}^{-}}\left(\frac{ p^{\frac{Q_0}{2}}}{C}|\ln p|^{d_0}\right)=\frac{1}{C}\lambda_{0}^{\frac{Q_0}{2}}|\ln \lambda_{0}|^{d_0}.
\end{equation}
Combining \eqref{5-19}-\eqref{5-21} we get
\begin{equation}\label{5-22}
\frac{1}{C}\lambda_{k}^{\frac{Q_{0}}{2}}|\ln \lambda_{k}|^{d_{0}}\leq k\leq C\lambda_{k}^{\frac{Q_{0}}{2}}|\ln \lambda_{k}|^{d_{0}}~~~\mbox{for all}~~k\geq k_{0}.
\end{equation}
As a consequence of \eqref{5-22}, we obtain
\begin{equation*}
\limsup_{k\to +\infty}\frac{\lambda_k(\ln k)^{\frac{2d_0}{Q_0}}}{k^{\frac{2}{Q_0}}}\leq \limsup_{k\to +\infty}\frac{\left(\ln C+\frac{Q_{0}}{2}\ln \lambda_{k}+d_{0}\ln \ln \lambda_{k} \right)^{\frac{2d_{0}}{Q_{0}}}}{C^{-\frac{2}{Q_{0}}}(\ln \lambda_{k})^{\frac{2d_{0}}{Q_{0}}}}=C^{\frac{2}{Q_{0}}}\left(\frac{Q_0}{2}\right)^{\frac{2d_0}{Q_0}},
\end{equation*}
and
\begin{equation*}
 \liminf_{k\to +\infty}\frac{\lambda_k(\ln k)^{\frac{2d_0}{Q_0}}}{k^{\frac{2}{Q_0}}}\geq \liminf_{k\to +\infty} \frac{\left(\ln \frac{1}{C}+\frac{Q_{0}}{2}\ln \lambda_{k}+d_{0}\ln \ln \lambda_{k} \right)^{\frac{2d_{0}}{Q_{0}}}}{C^{\frac{2}{Q_{0}}}(\ln \lambda_{k})^{\frac{2d_{0}}{Q_{0}}}}=C^{-\frac{2}{Q_{0}}}\left(\frac{Q_0}{2}\right)^{\frac{2d_0}{Q_0}}.
\end{equation*}
That means
$ \lambda_{k}\approx k^{\frac{2}{Q_{0}}}(\ln k)^{-\frac{2d_{0}}{Q_{0}}}$ as $k\to +\infty$.

The proof of Theorem \ref{thm3} is complete.
\end{proof}

\section{Some examples}
\label{Section6}

As further applications of Theorem \ref{thm2} and Theorem \ref{thm3}, we present some examples as follows.

\begin{example}
\label{ex6-1}
For $l\in \mathbb{N}^{+}$, we let $X=(\partial_{x_{1}},\ldots,\partial_{x_{n-1}},x_{1}^{l}\partial_{x_{n}})$ be the Grushin type vector fields defined on  $\mathbb{R}^n$.  The Grushin operator (see \cite{Grushin1970}) generated by $X$ is given by
\[ \triangle_{G}:=\frac{\partial^{2}}{\partial x_{1}^{2}}+\cdots+\frac{\partial^{2}}{\partial x_{n-1}^{2}}+x_{1}^{2l}\frac{\partial^{2}}{\partial x_{n}^{2}}. \]
The vector fields $X$ satisfy the assumption (H.1) with
the dilation
\[ \delta_{t}(x)=(t^{\alpha_{1}}x_{1},t^{\alpha_{2}}x_{2},\ldots,t^{\alpha_{n-1}}x_{n-1},t^{\alpha_{n}}x_{n})=(tx_{1},tx_{2},\ldots,tx_{n-1},t^{l+1}x_{n}), \]
and homogeneous dimension $Q=n+l$. Also, $X$ satisfy H\"{o}rmander's condition in $\mathbb{R}^n$ with the H\"{o}rmander index $r=l+1$. Moreover,  $x_{1}$ is the unique degenerate component of $X$ with corresponding degenerate index $\alpha_{1}=1$, and $w=\min_{x\in \mathbb{R}^{n}}\nu(x)=n$.

Assume  $\Omega\subset\mathbb{R}^n$ is a bounded open domain containing the origin. Proposition \ref{prop4-1} yields that
\begin{equation*}
 J_{\Omega}(r)=\int_{\Omega}\frac{dx}{\Lambda(x,r)}\approx\left\{
     \begin{array}{ll}
      \frac{1}{r^{Q-1}}|\ln r|, & \hbox{if $l=1$;} \\[2mm]
      \frac{1}{r^{Q-1}}, & \hbox{if $l>1$,}
     \end{array}
   \right.~~~\mbox{as}~~r\to 0^{+},
\end{equation*}
which gives $Q_{0}=Q-1=\tilde{\nu}-1$ and
\[ d_{0}=\left\{
           \begin{array}{ll}
             1, & \hbox{if $l=1$;} \\[2mm]
             0, & \hbox{if $l>1$.}
           \end{array}
         \right.\]
Denote by $\lambda_{k}$ the $k$-th Dirichlet eigenvalue of the Grushin operator $\triangle_{G}$ on $\Omega$. It follows from Theorem \ref{thm3} that
  \begin{equation}\label{6-1}
    \lambda_k\approx k^{\frac{2}{Q_{0}}}(\ln k)^{-\frac{2d_{0}}{Q_{0}}}\approx\left\{
                                                                                \begin{array}{ll}
                                                                                  \left(\frac{k}{\ln k} \right)^{\frac{2}{Q-1}}, & \hbox{if $l=1$;} \\[2mm]
                                                                                  k^{\frac{2}{Q-1}}, & \hbox{if $l>1$.}
                                                                                \end{array}
                                                                              \right.
~~~\mbox{as}~~ k\to+\infty.
  \end{equation}
\end{example}
\begin{remark}
The estimate \eqref{6-1}  for Dirichlet eigenvalues of the Grushin operator $\triangle_{G}$ improves the estimates by Chen and Luo presented in \cite[Theorem 1.2]{chen-luo2015}.
\end{remark}

\begin{example}
\label{ex6-2}
Let $X=(X_{1},X_{2})$ with
\[ X_{1}=\partial_{x_{1}},\quad X_{2}=x_{1}\partial_{x_{2}}+x_{1}^{2}\partial_{x_{3}}+\cdots+x_{1}^{n-1}\partial_{x_{n}} \]
be the Bony type vector fields defined on $\mathbb{R}^{n}$ (see \cite{Bony1969}). The Bony operator generated by $X$ is given by
\[ \triangle_{B}=\partial_{ x_{1}}^{2}+\left(x_{1}\partial_{x_{2}}+x_{1}^{2}\partial_{x_{3}}+\cdots+x_{1}^{n-1}\partial_{x_{n}}  \right)^{2}, \]
which satisfies H\"{o}rmander's condition but nevertheless with a very degenerate characteristic form. A direct calculation shows that $X$ satisfy the assumption (H.1) with
the dilation
\[ \delta_{t}(x)=(tx_{1},t^{2}x_{2},\ldots,t^{n}x_{n}), \]
and the homogeneous dimension $Q=\frac{n(n+1)}{2}$. In addition, $x_{1}$ is the unique degenerate component of $X$ with corresponding degenerate index $\alpha_{1}=1$, and $w=\min_{x\in \mathbb{R}^{n}}\nu(x)=Q-n+1$.

Suppose  $\Omega\subset\mathbb{R}^{n}$ is a bounded open domain containing the origin. By Proposition \ref{prop4-1},
\begin{equation*}
 J_{\Omega}(r)=\int_{\Omega}\frac{dx}{\Lambda(x,r)}\approx\left\{
     \begin{array}{ll}
      \frac{1}{r^{Q-1}}|\ln r|, & \hbox{if $n=2$;} \\[2mm]
      \frac{1}{r^{Q-1}}, & \hbox{if $n\geq 3$,}
     \end{array}
   \right.~~~\mbox{as}~~r\to 0^{+},
\end{equation*}
which gives  that
  \begin{equation}
    \lambda_k\approx \left\{
                                                                                \begin{array}{ll}
                                                                                  \left(\frac{k}{\ln k} \right)^{\frac{2}{Q-1}}, & \hbox{if $n=2$;} \\[2mm]
                                                                                  k^{\frac{2}{Q-1}}, & \hbox{if $n\geq 3$,}
                                                                                \end{array}
                                                                              \right.
~~~\mbox{as}~~ k\to+\infty,
  \end{equation}
where $\lambda_{k}$ is the $k$-th Dirichlet eigenvalue of the Bony operator on $\Omega$.

\end{example}

\begin{example}
\label{ex6-3}
Let $X=(X_{1},X_{2})$  be the Martinet type vector fields defined on  $\mathbb{R}^3$, where $X_{1}=\partial_{x_{1}}$ and $X_{2}=\partial_{x_{2}}+x_{1}^{2}\partial_{x_{3}}$. The Martinet operator generated by $X$ is given by
\[ \triangle_{M}:=\partial_{x_{1}}^{2}+\left(\partial_{x_{2}}+x_{1}^{2}\partial_{x_{3}}\right)^{2}. \]
We can verify that $X$ satisfy the assumption (H.1) with
the dilation $\delta_{t}(x)=(tx_{1},tx_{2},t^{3}x_{3})$,
and homogeneous dimension $Q=5$.
Meanwhile, $X$ satisfy H\"{o}rmander's condition in $\mathbb{R}^3$ with the H\"{o}rmander index $r=3$. Moreover, $x_{1}$ is the unique degenerate component of $X$ with corresponding degenerate index $\alpha_{1}=1$, and $w=\min_{x\in \mathbb{R}^{3}}\nu(x)=4$.

Assume that $\Omega\subset\mathbb{R}^{3}$ is a bounded open domain containing the origin. Using Proposition \ref{prop4-1} and Theorem \ref{thm3}, we have
\begin{equation}
  \lambda_{k}\approx \left(\frac{k}{\ln k}\right)^{\frac{2}{Q-1}}~~~\mbox{as}~~k\to+\infty,
\end{equation}
where $\lambda_{k}$ denotes the $k$-th Dirichlet eigenvalue of $\triangle_{M}$ on $\Omega$.

\end{example}
We mention that the three examples above  have only one degenerate component $x_{1}$. In this case, Proposition \ref{prop4-1} together with Theorem \ref{thm3} provide a precise
 asymptotic behaviour of the Dirichlet eigenvalue, complete with an exact growth rate. Furthermore, Examples \ref{ex6-1}-\ref{ex6-3} suggest that the upper bound $Q_{0}\leq Q-1$ in \eqref{1-12} for the index $Q_{0}$ is optimal. The subsequent example demonstrates that the index $Q_{0}$ in Theorem $\ref{thm3}$ may be fractional and presents a computational method for determining the indexes $Q_{0}$ and $d_{0}$.

\begin{example}
\label{ex6-4}
Consider the vector fields $X=(X_{1},X_{2},X_{3})$ defined on $\mathbb{R}^{3}$ as follows:
\[ X_{1}=\partial_{x_{1}},\quad X_{2}=x_{1}\partial_{x_{2}}+x_{2}\partial_{x_{3}},~~\mbox{and}~~X_{3}=x_{1}^{2}\partial_{x_{3}}. \]
The dilation of $X$ is given by $ \delta_t(x)=(tx_{1},t^2x_{2},t^3x_3)$,
which implies the homogeneous dimension $Q=6$. Clearly, $X$ satisfy H\"{o}rmander's condition in $\mathbb{R}^3$ with the H\"{o}rmander index $r=3$.\par

Assume $\Omega\subset\mathbb{R}^{3}$ is a bounded open domain containing the origin. It follows that
\[ \Lambda(x,r)\approx |x_1|^3r^3+(|x_1|^2+|x_2|)r^4+|x_1|r^5+r^6. \]
Therefore, Lemma \ref{lemma4-1} derives
\begin{equation}\label{6-4}
 J_{\Omega}(r)=\int_{\Omega}\frac{dx}{\Lambda(x,r)}\approx \int_{(0,1]^{2}}\frac{dx_{1}dx_{2}}{x_{1}^{3}r^{3}+(x_{1}^{2}+x_{2})r^{4}+x_{1}r^{5}+r^{6}}.
\end{equation}

We now estimate \eqref{6-4} using the method outlined in Proposition \ref{prop4-7}. The set of index pairs is given by
\[\mathfrak{G}=\{(a_1,a_2,s)|(3,0,3),(2,0,4),(0,1,4),(1,0,5),(0,0,6)\}. \]
For each index pair $(a_1,a_2,s)\in \mathfrak{G}$, we let
\[ P_{a_1,a_2,s}=\{(y_{1},y_{2})\in [0,+\infty)^{2}|(a_{1}-a_{1}')y_{1}+(a_{2}-a_{2}')y_{2}\leq s'-s,~\forall (a_{1}',a_{2}',s')\in \mathfrak{G}\} \]
be the polyhedron in $[0,+\infty)^{2}$, and let
$ \phi_{a_1,a_2}(y)=(a_{1}-1)y_{1}+(a_{2}-1)y_{2} $
be the corresponding linear function. It follows from Proposition \ref{prop4-4} that
\begin{equation}\label{6-5}
  J_{\Omega}(r)\approx J_{3,0,3}(r)+J_{2,0,4}(r)+J_{0,1,4}(r)+J_{1,0,5}(r)+J_{0,0,6}(r),
\end{equation}
where
\begin{equation}\label{6-6}
  J_{a_1,a_2,s}(r):=\left(\ln\frac{1}{r}\right)^{2}\frac{1}{r^{s}}\int_{P_{a_1,a_2,s}}e^{\left(\ln \frac{1}{r}\right)\phi_{a_1,a_2}(y)}dy.
\end{equation}
According to Lemma \ref{lemma4-3}, we can find the maximum value $m_{a_1,a_2,s}$ of $\phi_{a_1,a_2}$ in polyhedron $P_{a_1,a_2,s}$, if $V_2(P_{a_1,a_2,s})\neq 0$, and the dimension $d_{a_1,a_2,s}=\dim\{x\in P_{a_1,a_2,s}|\phi_{a_1,a_2}(x)=m_{a_1,a_2,s}\}$.

For index pairs $(2,0,4)$ and $(1,0,5)$, we have $V_2(P_{2,0,4})=V_2(P_{1,0,5})=0$, which implies $J_{2,0,4}(r)=J_{1,0,5}(r)=0$. For other index pairs, employing the linear programming and Proposition \ref{prop4-6} we obtain
\begin{enumerate}[(1)]
  \item  $(3,0,3)$: $m_{3,0,3}=\frac{2}{3}$, $d_{3,0,3}=0$ and $ J_{3,0,3}(r)\approx r^{-\frac{11}{3}}$ as $r\to 0^{+}$.
  \item $(0,1,4)$: $m_{0,1,4}=-\frac{1}{3}$, $d_{0,1,4}=0$ and
$ J_{0,1,4}(r)\approx r^{-\frac{11}{3}}$ as $r\to 0^{+}$.
  \item $(0,0,6)$: $m_{0,0,6}=-3$, $d_{0,0,6}=0$ and
$ J_{0,0,6}(r)\approx r^{-3}$ as $r\to 0^{+}$.
\end{enumerate}
In conclusion, we have $ J_{\Omega}(r)\approx r^{-\frac{11}{3}}$ as $r\to 0^{+}$. Consequently,   $\lambda_{k}\approx k^{\frac{6}{11}}=k^{\frac{2}{11/3}}$ as $k\to +\infty$, where $\lambda_{k}$ is the $k$-th Dirichlet eigenvalue of the operator $\triangle_{X}=X_{1}^{2}+X_{2}^{2}+X_{3}^{2}$ on $\Omega$. It is noteworthy that
 $\frac{11}{3}$ is not an integer. This type of asymptotic behaviour of Dirichlet eigenvalues for degenerate elliptic operators unveils a completely new phenomenon that we have not encountered before.
\end{example}

Finally, we provide two examples satisfying  $Q_{0}=Q-\alpha(X)>w$ and $Q-\alpha(X)<w=Q_{0}$, respectively. These examples suggest that the lower bound $Q_0\geq \max\{Q-\alpha(X),w\}$ in \eqref{1-12} is optimal.

\begin{example}($Q_{0}=Q-\alpha(X)>w$) Consider the vector fields $X=(X_{1},X_{2},X_{3})$ in $\mathbb{R}^3$, where
\[ X_{1}=\partial_{x_{1}},\quad X_{2}=x_{1}\partial_{x_{2}}+x_{1}^{3}\partial_{x_{3}},~~\mbox{and}~~X_{3}=x_{1}x_{2}\partial_{x_{3}}. \]
The corresponding dilation is given by $\delta_t(x)=(tx_1,t^2x_2,t^4x_3)$,
and the homogeneous dimension $Q=7$. Moreover, $X$ satisfy H\"{o}rmander's condition in $\mathbb{R}^3$ with the H\"{o}rmander index $r=4$.

Suppose that $\Omega\subset\mathbb{R}^{3}$ is a bounded open domain containing the origin. A direct calculation yields that
\begin{equation}\label{6-7}
  \Lambda(x,r)\approx |x_1|^2|x_2|r^3+(|x_1x_2|+|x_1|^3)r^4+(|x_2|+|x_1|^2)r^5+|x_1|r^6+r^7,
\end{equation}
which implies $Q-\alpha(X)=4>w=3$. Therefore, by Lemma \ref{lemma4-1} and Proposition \ref{prop4-7} we have
\begin{equation}
\begin{aligned}
 J_{\Omega}(r)&=\int_{\Omega}\frac{dx}{\Lambda(x,r)}\approx\int_{(0,1]^{2}}\frac{dx_{1}dx_{2}}{x_1^2x_2r^3+(x_1x_2+x_1^3)r^4+(x_2+x_1^2)r^5+x_1r^6+r^7}\\
 &\approx \frac{1}{r^4}|\ln r|~~\mbox{as}~~r\to 0^+.
\end{aligned}
\end{equation}
This means $Q_{0}=Q-\alpha(X)=4>w=3$ and $d_{0}=1$, which yields $\lambda_k\approx k^{\frac{1}{2}}(\ln k)^{-\frac{1}{2}}$ as $k\to +\infty$.

\end{example}

\begin{example}($Q_{0}=w>Q-\alpha(X)$)
\label{ex6-6}
 Consider the vector fields $X=(X_{1},X_{2})$ in $\mathbb{R}^3$, where
\[ X_{1}=\partial_{x_{1}}-x_{2}^{2}\partial_{x_{3}}~~\mbox{and}~~X_{2}=\partial_{x_{2}}+x_{1}^{2}\partial_{x_{3}}. \]
The corresponding dilation is given by $\delta_t(x)=(tx_1,tx_2,t^3x_3)$,
and the  homogeneous dimension $Q=5$. Obviously, $X$ satisfy H\"{o}rmander's condition in $\mathbb{R}^3$ with the H\"{o}rmander index $r=3$.

Let $\Omega\subset\mathbb{R}^{3}$ be a bounded open domain containing the origin. Then
\begin{equation}\label{6-9}
  \Lambda(x,r)\approx |x_1+x_2|r^4+r^5,
\end{equation}
which gives  $w=4>Q-\alpha(X)=3$. It follows  from Lemma \ref{lemma4-1} and \eqref{6-9} that
\begin{equation}\label{6-10}
  J_{\Omega}(r)=\int_{\Omega}\frac{dx}{\Lambda(x,r)}\approx\int_{(-1,1)^2}\frac{dx_1 dx_2}{|x_1+x_2|r^4+r^5}.
\end{equation}

A general method for transforming the integrand $(|x_1+x_2|r^4+r^5)^{-1}$
 into the form specified in the property (C) of Proposition \ref{prop4-2} is the blow-up technique in algebraic geometry (see \cite[Section 3]{Wtanabe2009}). However, this method involves many tedious calculations.
As an alternative, we will consider a different variable transformation to handle \eqref{6-10}.

Changing $x_1+x_2=u_1$ and $x_2=u_2$ in \eqref{6-10}, we obtain
\begin{equation}\label{6-11}
  \int_{(-1,1)^2}\frac{dx_1 dx_2}{|x_1+x_2|r^4+r^5}=\int_{M}\frac{d u_1du_2}{|u_1|r^4+r^5},
\end{equation}
where $M=\{(u_1,u_2)\in \mathbb{R}^2|u_2-1< u_1<u_2+1, -1<u_2< 1\}$ is an open domain containing the origin. Hence, using Lemma \ref{lemma4-1}  again, we have
\begin{equation}\label{6-12}
 J_{\Omega}(r)=\int_{\Omega}\frac{dx}{\Lambda(x,r)}\approx \int_{M}\frac{d u_1du_2}{|u_1|r^4+r^5}\approx \int_{0}^{1}\frac{d u_1}{u_{1}r^4+r^5}\approx \frac{1}{r^4}|\ln r|~~\mbox{as}~~r\to 0^+.
\end{equation}
This means $Q_{0}=w=4>Q-\alpha(X)=3$ and $d_{0}=1$, which yields that $\lambda_k\approx k^{\frac{1}{2}}(\ln k)^{-\frac{1}{2}}$ as $k\to +\infty$.
\end{example}
\begin{remark}
The change of variables in \eqref{6-11} is a straightforward application of resolution of singularities. In this case, the triple $(M,W,\rho)$ specified in Proposition \ref{prop4-2} can be given by
\[  M=\{(u_1,u_2)\in \mathbb{R}^2|u_2-1< u_1<u_2+1, -1<u_2< 1\},~~~~ W=(-1,1)^2, \]
and $\rho:M\to W$ is the real analytic map such that
$ \rho(u_1,u_2)=(u_1-u_2,u_2)$.
\end{remark}

\section*{Acknowledgements}
 Hua Chen is supported by National Natural Science Foundation of China (Grant No. 12131017) and National Key R$\&$D Program of China (no. 2022YFA1005602). Hong-Ge Chen is supported by National Natural Science Foundation of China (Grant No. 12201607), Knowledge Innovation Program of Wuhan-Shuguang Project (Grant No. 2023010201020286) and China Postdoctoral Science Foundation (Grant Nos. 2023T160655 and 2021M703282). Jin-Ning Li is supported by China National Postdoctoral Program for Innovative Talents (Grant No. BX20230270).

\bibliographystyle{plain}
\bibliography{reference.bib}

\end{document}